\documentclass{amsart}
\usepackage{enumerate}
\usepackage{tikz}
\usetikzlibrary{calc,decorations.markings,decorations.pathmorphing,decorations.pathreplacing,automata,positioning}
\usepackage{hyperref}
\usepackage{mathrsfs}
\usepackage{mathabx}
\usepackage{amsfonts}
\usepackage{latexsym}
\usepackage{epsfig}
\usepackage{indentfirst}
\graphicspath{{figures/}}
\usepackage{amsmath}
\usepackage{mathtools}
\usepackage{amssymb}
\usepackage{amsthm}
\usepackage{graphicx}
\usepackage{amssymb}
\usepackage{amscd}
\usepackage{latexsym}
\usepackage{bbm}
\usepackage{tikz-cd}
\usepackage{color}
\usepackage{adjustbox}
\DeclareGraphicsExtensions{.eps}
\usepackage{geometry}
\usepackage{fancyhdr}       

\newtheorem{theorem}{Theorem}
\newtheorem{Theorem}{Theorem}[section]
\newtheorem{lem}[Theorem]{Lemma}
\newtheorem{prop}[Theorem]{Proposition}

\theoremstyle{definition}
\newtheorem{defni}[Theorem]{Definition}

\theoremstyle{remark}
\newtheorem{Remark}[Theorem]{Remark}

\newcommand{\abs}[1]{\left| #1 \right|}

\newcommand{\calA}{\mathcal{A}}
\newcommand{\calB}{\mathcal{B}}
\newcommand{\calC}{\mathcal{C}}
\newcommand{\calD}{\mathcal{D}}

\newcommand{\calH}{\mathcal{H}}

\newcommand{\calK}{\mathcal{K}}
\newcommand{\calL}{\mathcal{L}}
\newcommand{\calM}{\mathcal{M}}
\newcommand{\calP}{\mathcal{P}}

\newcommand{\calX}{\mathcal{X}}
\newcommand{\calY}{\mathcal{Y}}
\newcommand{\calZ}{\mathcal{Z}}

\newcommand{\scrA}{\mathscr{A}}
\newcommand{\scrU}{\mathscr{U}}
\newcommand{\scrV}{\mathscr{V}}

\newcommand\Z{_{\mathcal{K}_{1}}\mathcal{Z}^{\mathcal{K}_{2}}}
\newcommand\abZ{_{\mathcal{K}_{1}}\overline{\mathcal{Z}}^{\mathcal{K}_{2}}}
\newcommand\AbZ{_{\mathcal{K}}\overline{\mathcal{Z}}^{\mathcal{K}}}

\newcommand\bbE{\mathbb{E}}
\newcommand\bbF{\mathbb{F}}
\newcommand\bbH{\mathbb{H}}
\newcommand\bbI{\mathbb{I}}
\newcommand\bbL{\mathbb{L}}
\newcommand\bbN{\mathbb{N}}
\newcommand\bbP{\mathbb{P}}
\newcommand\bbR{\mathbb{R}}
\newcommand\bbS{\mathbb{S}}
\newcommand\bbT{\mathbb{T}}
\newcommand\bbQ{\mathbb{Q}}

\newcommand\bbX{\mathbb{X}}
\newcommand\bbZ{\mathbb{Z}}

\newcommand{\bfa}{\mathbf{a}}
\newcommand{\bfd}{\mathbf{d}}
\newcommand{\bfe}{\mathbf{e}}
\newcommand{\bfi}{\mathbf{i}}
\newcommand{\bfj}{\mathbf{j}}

\newcommand{\bfx}{\mathbf{x}}
\newcommand{\bfy}{\mathbf{y}}

\newcommand{\bfE}{\mathbf{E}}
\newcommand{\bfI}{\mathbf{I}}

\newcommand{\bfX}{\mathbf{X}}

\newcommand{\frs}{\mathfrak{s}}

\newcommand{\ve}[1]{\boldsymbol{\mathbf{#1}}}
\newcommand{\es}{\ve{e}}
\newcommand{\ks}{\ve{k}}
\newcommand{\js}{\ve{j}}
\newcommand{\xs}{\ve{x}}
\newcommand{\ys}{\ve{y}}
\newcommand{\zs}{\ve{z}}
\newcommand{\ws}{\ve{w}}

\newcommand{\ps}{\ve{p}}

\newcommand{\as}{\ve{\alpha}}
\newcommand{\bs}{\ve{\beta}}
\newcommand{\gs}{\ve{\gamma}}

\newcommand\Ta{\mathbb{T}_\alpha}
\newcommand\Tb{\mathbb{T}_\beta}

\newcommand\alphas{\mbox{\boldmath$\alpha$}}
\newcommand\betas{\mbox{\boldmath$\beta$}}

\newcommand\Boxtimes{\Hat{\boxtimes}}

\newcommand{\CFL}{\mathcal{CFL}}
\newcommand{\HFL}{\mathcal{HFL}}

\newcommand\HFm{\mathbf{HF}^-}

\newcommand\CFKm{\mathcal{CFK}^{-}}

\newcommand\spin{\mathrm{Spin}}
\newcommand\Sym{\mathrm{Sym}}
\newcommand\gr{\ve{\mathrm{gr}}}

\newcommand\lk {{\operatorname{lk}}}

\newcommand{\DAperturbdiagram}[3]{%
	\begin{tikzcd}[
		row sep=small,
		column sep=0.35em,
		ampersand replacement=\&
		]
		a \arrow[d, Rightarrow] \& x \arrow[d] \& \\
		\Delta \arrow[dd, Rightarrow]
		\arrow[dr, Rightarrow, bend right=10]
		\& #1 \arrow[d] \arrow[ddddddddr, bend left=10] \& \\
		\& \delta^1_{>1} \arrow[d]
		\arrow[dddddddr, bend left=10] \& \\
		\Delta \arrow[dd, Rightarrow]
		\arrow[dr, Rightarrow, bend right=10]
		\& h^1 \arrow[d] \arrow[ddddddr, bend left=10] \& \\
		\& \delta^1_{>1} \arrow[d]
		\arrow[dddddr, bend left=10] \& \\
		\vdots \arrow[ddr, Rightarrow, bend right=20]
		\& \vdots \arrow[d] \arrow[ddddr, bend left=10] \& \\
		\& h^1 \arrow[d] \arrow[dddr, bend left=10] \& \\
		\& \delta^1_{>1} \arrow[d]
		\arrow[ddr, bend left=10] \& \\
		\& #2 \arrow[dd] \arrow[dr, bend left=10] \& \\
		\& \& \mu_{\mathcal B} \arrow[d] \\
		\& #3 \& \mathcal B
	\end{tikzcd}%
}

\title{Cabling formula and colored knot Floer homology}
\author{Chen Zhang}
\address{Simons Center for Geometry and Physics}
\email{czhang@scgp.stonybrook.edu}
\begin{document}
\maketitle
\begin{abstract}
	We study colored knot Floer homology via a bordered cabling construction. We construct a cabling module and identify the connecting maps in the defining direct system at the level of this module. By taking the colimit at the chain level, we obtain a general formula for colored knot Floer homology of an arbitrary knot in terms of its ordinary knot Floer homology and the colored unknot algebra. This extends the computation of Alishahi--Gorsky--Liu from L-space knots to all knots.
\end{abstract}
\section{Introduction}
The behavior of knot Floer homology under cabling has been studied extensively since the early development of the theory. Hedden initiated the study of knot Floer homology under cabling, obtaining explicit computations for sufficiently twisted cables \cite{hedden2005knot,hedden2009knot}. Bordered Floer techniques have also been applied to the study of cabling; see, for example, Petkova \cite{petkova2013cables} and Hom \cite{hom2011note,hom2014bordered}. For multi-component cables, Gorsky and Hom studied cable links of $L$-space knots and gave explicit descriptions of their link Floer homology \cite{gorsky2017cable}. More recently, cabling formulas have also been studied using the immersed-curve formulation of bordered Floer homology \cite{hanselman2023cabling}.

Besides being a natural operation in knot theory, cabling also appears in several constructions connecting knot Floer homology with four-dimensional topology and with colored knot invariants. In \cite{chen2022floer}, Chen introduced Floer lasagna modules, an analogue of skein lasagna modules defined using link Floer homology. For $4$-manifolds obtained from $B^4$ by attaching $2$-handles, these modules admit a description in terms of the link Floer homology of cables of the attaching link, together with maps induced by link cobordisms. Thus computations of cable link Floer homology provide a
natural input for the study of Floer lasagna modules.

Cabling also plays a central role in recent constructions of colored versions of knot Floer homology. Cooper and DeYeso
\cite{cooper2025holonomicity} construct $S^r$-colored knot Floer homologies using a stabilization involving the knot Floer homology of highly twisted cables, and prove categorified recurrence relations for these invariants. More recently, Alishahi, Gorsky, and Liu \cite{alishahi2025colored} define $n$-colored knot Floer homology as the colimit of the link Floer homologies of the $(n,nm)$-cables $K_{n,nm}$ under cobordism maps corresponding to the addition of a full twist. Their construction uses the full version of link Floer homology and, in contrast with the one-component cables appearing in \cite{cooper2025holonomicity}, naturally involves $n$-component cable links.

\subsection{Cabling modules}
In this paper, we focus on the study of $(n,nm)$-cables. Viewing the $(n,nm)$-cabling pattern as a link in a solid torus, we use Zemke's bordered surgery package \cite{zemke2021bordered} to compute its type-$A$ module over the knot surgery algebra $\mathcal K$, denoted by ${}_{\mathcal K}\mathcal C_{n,m}$. Our approach is closely related to the satellite formula of Chen-Zemke-Zhou \cite{chen2024space}, which uses the bordered reinterpretation of the link surgery formula and its connected sum formula to compute the effect of satellite operations on
knot Floer homology.

Let $\mathcal X_0(K)^{\mathcal K}$ denote the type-$D$ module associated to the $0$-framed complement of $K$. Our first main result is a cabling formula of the following form.

\begin{theorem}\label{theorem:cabling-module-intro}
	For every knot $K\subset S^3$ and integers $n>0$ and $m$, there is an isomorphism
	\[
	H_*\left(\mathcal X_0(K)^{\mathcal K}\widehat{\boxtimes}{}_{\mathcal K}\mathcal C_{n,m}\right)\cong\HFL(K_{n,nm})
	\]
	as modules over $R_n=\mathbb F[\scrU_1,\scrV_1,\ldots,\scrU_n,\scrV_n]$. Moreover, the type-$A$ module ${}_{\mathcal K}\mathcal C_{n,m}$ admits an explicit description by generators, relations, and structure maps.
\end{theorem}

In Section~\ref{sec:cabling_computations}, we apply the cabling formula to compute $\HFL(K_{n,nm})$ for several companion knots. When the companion knot is an $L$-space knot, Gorsky and Hom \cite{gorsky2017cable} showed that sufficiently positive cable links are $L$-space links and gave an explicit description of their link Floer homology. Our bordered cabling formula recovers their result by directly computing the homology of the corresponding paired complex.

\begin{Remark}
	As mentioned above, the cabling modules constructed in this paper are also motivated by the study of Floer lasagna modules. At present, however, our description does not seem sufficiently flexible to give a general computation of these modules. One difficulty is that the $R_n$-modules arising from our cabling formula are in general not free. This makes it substantially harder to manipulate the resulting homology and to keep track of the coefficient actions. For example, in Chen's construction \cite{chen2022floer}, one works with the hat version of link Floer homology, which requires setting the variables $\scrU_i$ and $\scrV_i$ equal to zero; carrying this operation out directly from our non-free $R_n$-module description is not immediate.
	
	There is also additional cobordism information that is not addressed in the present paper. A computation of the Floer lasagna module also requires an understanding of the pair-of-pants, braid, and basepoint-moving cobordism maps. We expect to explore these maps in future work.
\end{Remark}

\subsection{Colored knot Floer homology}
In \cite{alishahi2025colored}, Alishahi, Gorsky, and Liu defined colored knot Floer homology using a direct system of cable links. Given a knot $K\subset S^3$ and an integer $n>0$, consider the family of $n$-component cables $K_{n,mn}$ for $m\geq0$ and their link Floer homology $\HFL(K_{n,mn})$. Let $M$ be an unknot bounding a disk $D$ that intersects every component of $K_{n,mn}$ positively at exactly one point, then attaching a $(-1)$-framed 2 handle along $M$ induces a cobordism 
\[W_m\colon (S^3,K_{n,mn})\longrightarrow(S^3,K_{n,(m+1)n}),\] 
see Figure \ref{fig:framingchangecob}. Fix the $\spin^c$ structure $\frs_0$ on $W_m$ satisfying $\langle c_1(\frs_0),[S^2]\rangle=1$, where $[S^2]$ is the generator of the second homology class of the cobordism. Let $\phi_0:=F_{W_m,\frs_0}$ denote the corresponding cobordism map. The \emph{n-colored knot Floer homology} of $K$ is defined as the colimit of the directed system
\begin{equation}\label{eq:coloreddef}
	\calH_n(K)=\underrightarrow{\mathrm{lim}}\left[\HFL(K_{n,0})\xrightarrow{\phi_0}\HFL(K_{n,n})\xrightarrow{\phi_0}\HFL(K_{n,2n})\rightarrow\cdots\right].
\end{equation}

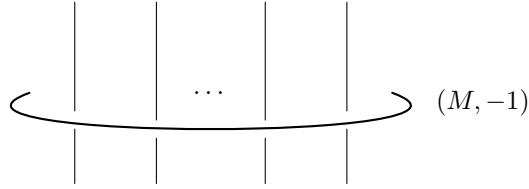
\begin{figure}[ht]
\centering
\begin{tikzpicture}[scale=1.2, line cap=round]
\draw[thick] (-2,0) to[out=200,in=340] (2,0);
\draw (-1.5,1) to (-1.5,-0.2);
\draw (-0.6,1) to (-0.6,-0.3);
\draw (0.6,1) to (0.6,-0.3);
\draw (1.5,1) to (1.5,-0.2);
\draw (-1.5,-0.4) to (-1.5,-1);
\draw (-0.6,-0.5) to (-0.6,-1);
\draw (0.6,-0.5) to (0.6,-1);
\draw (1.5,-0.4) to (1.5,-1);
\node at (0,0) {$\dots$};
\node at (3,-0.1) {$(M,-1)$};
\end{tikzpicture}
\caption{The local $2$-handle attachment that adds a full twist.}
\label{fig:framingchangecob}
\end{figure}

Our cabling module allows us to compute the individual homology groups appearing in the direct system. To identify the connecting maps, we use another important feature of the link surgery complex: $2$-handle cobordism maps can be realized as inclusions of subcomplexes in a link surgery hypercube \cite{manolescu2010heegaard,ozsvath2008knot}. The original result studies the cobordism map between $3$-manifolds. Zemke's general surgery formula allows us to retain an unsurgered sublink throughout this construction. Using this framework, we extend the interpretation to $2$-handle cobordism maps for links; see Theorem \ref{theorem:linkcobsurgery}.

Moreover, the relevant subcomplex is determined by Alexander grading on the link surgery complex. In our case, this grading only depends on the cabling module and is independent of the companion knot input. Hence, we can interpret the corresponding cobordism map entirely as a map on the cabling module. More precisely, we construct type-$A$ maps, again denoted by $\phi_0$,
\[
\phi_0:{}_{\calK}\calC_{n,m}\longrightarrow{}_{\calK}\calC_{n,m+1}.
\]
which are independent of the companion knot, and whose pairing with $\mathcal X_0(K)^{\mathcal K}$ agrees with the $2$-handle cobordism map defining colored knot Floer homology.

\begin{theorem}\label{theorem:cobordism-pattern-intro}
	Under the identifications of Theorem~\ref{theorem:cabling-module-intro}, the map
	\[
	\widehat{\boxtimes}\phi_0:\mathcal X_0(K)^{\mathcal K}\widehat{\boxtimes}{}_{\mathcal K}\mathcal C_{n,m}\longrightarrow\mathcal X_0(K)^{\mathcal K}\widehat{\boxtimes}{}_{\mathcal K}\mathcal C_{n,m+1}
	\]
	induces the link Floer cobordism map
	\[
	\phi_0:\HFL(K_{n,nm})\longrightarrow\HFL(K_{n,n(m+1)}).
	\]
\end{theorem}

Consequently, the computation of colored knot Floer homology becomes an algebraic problem concerning the directed system of the complexes obtained from these type-$A$ modules. Alishahi-Gorsky-Liu computed the colored knot Floer homology for L-space knots. More precisely, they showed that if $K$ is an L-space knot, then its $n$-colored knot Floer homology is obtained from the ordinary knot Floer homology of $K$ by tensoring with the colored unknot algebra \cite[Theorem~1.17]{alishahi2025colored}:
\[
\mathcal H_n(K)\cong\HFL(K)\otimes_{\mathbb F[\scrU,\scrV]}\mathcal A_n^{\mathrm{col}},
\]
where the colored unknot algebra $\calA_n^\mathrm{col}$ is the commutative algebra generated by $\scrU_1,\ldots,\scrU_n,\scrV_1,\ldots,\scrV_n,A$ and relations 
\begin{equation}\label{eq:Arelation}
	\scrU_i=\mathsf A\prod_{j\neq i}\scrV_j,\;\scrU_i\scrV_i=\scrU_{i'}\scrV_{i'}\qquad i,i'=1,\ldots,n.
\end{equation}
The $\mathbb F[\scrU,\scrV]$-module structure on $\mathcal A_n^{\mathrm{col}}$ is induced by
\[
\scrU\longmapsto A,\qquad\scrV\longmapsto \scrV_1\cdots \scrV_n.
\]
Using the cabling module and the description of the connecting maps above, we remove the L-space assumption and obtain the same formula for an arbitrary knot.

\begin{theorem}\label{theorem:colored-general-intro}
	For any knot $K\subset S^3$, there is an isomorphism
	\[
	\mathcal H_n(K)\cong\HFL(K)\otimes_{\mathbb F[\scrU,\scrV]}\mathcal A_n^{\mathrm{col}}.
	\]
\end{theorem}

The definition of colored knot Floer homology in
\cite{alishahi2025colored} is inspired by $S^n$-colored Khovanov
homology. In the construction of Cooper--Krushkal
\cite{cooper2012categorification}, the color $n$ is represented by
the categorified Jones--Wenzl projector $P_n$, a chain complex of
Temperley--Lieb diagrams with differentials given by cobordisms.
The colored knot complex is obtained by inserting $P_n$ into the
$n$-parallel cable of the knot and evaluating the resulting closed
complex.

This construction has a similar flavor to our bordered approach:
both combine universal coloring or cabling data with knot-dependent
chain data. In the Khovanov setting, the projector is paired with
the complex of the cabled tangle by planar composition. In the
Floer setting, the cabling module is paired with the bordered
invariant of the knot complement by the box tensor product.

In our setting, the chain-level pairing admits a further
simplification. After passing to the full-twist colimit and canceling the additional
cabling summands, the remaining complex is obtained from the
uncolored knot Floer complex by changing the coefficient ring
from $\mathbb F[U,V]$ to $\mathcal A_n^{\mathrm{col}}$.
Flatness then allows this tensor-product description to pass to
homology, yielding the formula of Theorem~\ref{theorem:colored-general-intro}. Thus the colored knot Floer homology considered here is determined
by the uncolored knot Floer homology as a graded
$\mathbb F[\scrU,\scrV]$-module.

\subsection*{AI Disclosure} The author used the cabling formula in Section \ref{sec:cabling_computations} to compute colored knot Floer homology for knots whose knot Floer homology is formal, as well as to carry out an explicit computation for the left-handed trefoil. The strategy in that section is to first compute the homology of the individual terms in the direct system and then analyze the resulting colimit.

The general argument in Section \ref{sec:chaincolimit} was developed with assistance from ChatGPT~6. In particular, ChatGPT~6 suggested the key observation that, rather than first taking homology and then forming the colimit, one should first take the colimit at the chain level and then compute its homology. This observation led to the general proof presented in that section. All mathematical arguments and claims were subsequently checked and verified by the author.

\subsection*{Outline}The paper is organized as follows. Sections~\ref{sec:algebra} and~\ref{sec:linkfloer} review the algebraic and link Floer background. Section~\ref{sec:surgerycob} proves the cobordism-map interpretation, and Section~\ref{sec:bordered} recalls the bordered surgery formalism. Sections~\ref{sec:cabling} and~\ref{sec:coloredmodule} construct the cabling modules and full-twist maps. Section~\ref{sec:cabling_computations} gives the examples, and Section~\ref{sec:chaincolimit} develops the algebraic colimit calculation.

\section*{Acknowledgments} The author would like to thank Matt Hedden and Wenzhao Chen for helpful comments on an earlier draft, and John Pardon and Claudius Zibrowius for helpful discussions. The author also thanks Matt Hedden for bringing the paper on colored knot Floer homology to the author's attention. Part of this work was carried out during the author's postdoctoral appointment at Ruhr University Bochum, with support from ERC grant (project no. 101115822).

\section{Algebraic background}\label{sec:algebra}

In this section, we review the necessary algebra, including relevant background in $A_\infty$-modules, homological perturbation, and linear topological spaces.

\subsection{$A_\infty$-modules}\label{subsec:DAmodule}
We give a brief review of $A_\infty$-modules in this subsection. For a thorough treatment, see \cite{lipshitz2015bimodules,lipshitz2018bordered}.
\begin{defni}\label{def:Amodule}
	Suppose $\calA$ is an associative algebra over $\ve{k}$. A \emph{left $A_\infty$-module} (or, succinctly, \emph{left type $A$-module}) over $\mathcal{A}$, denoted by ${}_\mathcal{A} M$, is a left $\ve{k}$-module $M$ equipped with a family of $\ve{k}$-module maps 
	\[
	\{m_{j+1}: \mathcal{A}^{\otimes j} \otimes_{\ve{k}} M \to M\}_{j \ge 0}
	\]
	such that, for each $a_1, \ldots, a_n \in \mathcal{A}$ and $x \in M$,
	\begin{itemize}
		\item $m_1^2=0$,
		\item $m_2(a_1, m_1(x)) = m_1(m_2(a_1, x))$, and
		\item for each $n\ge 2$, 
		\[
		\sum_{j=0}^n m_{n-j+1}(a_n, \ldots, a_{j+1}, m_{j+1}(a_j, \ldots, a_1, x)) = \sum_{k=1}^{n-1}m_n(a_n, \ldots, a_{k+1}a_k, \ldots, a_1, x).
		\]
	\end{itemize}
	We call ${}_\mathcal{A} M$ \emph{strictly unital} if
	\begin{itemize}
		\item $m_2(1, x) = x$ for each $x\in M$, and
		\item for each $j \geq 2$, $m_{j+1}(a_j, \ldots, a_1, x) = 0$ whenever $a_i \in \ks$ for some $i=1,\ldots, j$.
	\end{itemize}
\end{defni}

\begin{defni}\label{def:typed}
	A \emph{right type-$D$ module} over $\mathcal{A}$, denoted by $N^\mathcal{A}$, is a right $\ve{k}$-module $N$ equipped with a $\ve{k}$-linear structure map
	\[
	\delta^1: N \to N \otimes_{\ve{k}} \mathcal{A}
	\]
	which satisfies
	\[
	(\mathrm{Id}_N \otimes \mu) \circ (\delta^1 \otimes \mathrm{Id}_\mathcal{A}) \circ \delta^1 = 0.
	\]
\end{defni}

Combining a left type-$A$ structure over $\mathcal A$ with a right type-$D$ structure over $\mathcal B$ gives a type-$DA$ bimodule.

\begin{defni}
	Suppose that $\calA$ and $\calB$ are algebras over $\ve{j}$ and $\ve{k}$, respectively. A \emph{type-$DA$ $(\mathcal{A}, \mathcal{B})$-bimodule}, denoted by ${}_\mathcal{A} M^\mathcal{B}$, consists of a $(\ve{j}, \ve{k})$-bimodule $M$ equipped with $(\ve{j}, \ve{k})$-linear maps
	\[
	\delta_{j+1}^1: \mathcal{A}^{\otimes j} \otimes_{\ve{j}} M \to M \otimes_{\ve{k}} \mathcal{B}
	\]
	for each $j\ge 0$ which satisfies
	\begin{align*}
		\sum_{j=0}^n \big( (\mathrm{Id}_M \otimes \mu_\mathcal{B}) & \circ (\delta_{n-j+1}^1 \otimes \mathrm{Id}_\mathcal{B}) \big) (a_n, \ldots, a_{j+1}, \delta_{j+1}^1(a_j, \ldots, a_1, x))
		\\ & =
		\sum_{k=1}^{n-1}\delta_n^1(a_n, \ldots, \mu_\mathcal{A}(a_{k+1}, a_k), \ldots, a_1, x).
	\end{align*}
\end{defni}
We also recall two types of tensor products that will be used to combine modules.

\emph{External tensor product:} Let $\calA$ and $\calB$ be two algebras over rings $\bfi$ and $\bfj$ of characteristic 2, and let $N^\calA=(N,\delta^1)$ and $M^\calB=(M,d^1)$ be type-$D$ modules. Then the \emph{external tensor product}, denoted $(N\otimes_\bbF M)^{\calA\otimes_\bbF\calB}$, is the module $N\otimes_\bbF M$ equipped with the  type-$D$ structure map $\mathrm{id}_N\otimes1_\calA\otimes d^1+\delta^1\otimes\mathrm{id}_M\otimes1_\calB$.

\emph{Box tensor product:} Given a right type-$D$ module $N^\mathcal{A}$ and a left type-$A$ module  ${}_{\mathcal{A}} M$, Lipshitz-Ozsv\'ath-Thurston \cite{lipshitz2018bordered} provide an associated chain complex. The \emph{box tensor product} $N^\mathcal{A} \boxtimes {}_{\mathcal{A}} M$ is the chain complex $(N \otimes_{\ve{k}} M, \partial^\boxtimes)$ defined as follows. First, using the structure map $\delta^1$ from $N^\mathcal{A}$, define the map $\delta^j: N \to N \otimes \mathcal{A}^j$ inductively, setting $\delta^0 = \mathrm{Id}_N$ and 
\[
\delta^j = (\delta^1 \otimes \mathrm{Id}_{\mathcal{A}^\otimes{(j-1)}}) \circ\delta^{j-1}.
\]
Then, $\partial^\boxtimes$ is defined by the formula
\[
\partial^\boxtimes(x \otimes y) = \sum_{j=0}^\infty (\mathrm{Id}_N \otimes m_{j+1})(\delta^j(x), y).
\]
The map $\partial^\boxtimes$ is a differential assuming that one of $M$ or $N$ satisfies a boundedness assumption; see \cite{lipshitz2015bimodules}. 

The box tensor extends naturally to bimodules; for example, one may form $AD \boxtimes AD$, $AD \boxtimes A$, and $D \boxtimes AD$. In this paper, one input for the box tensor will always be a type-$D$ module. As a result, the corresponding differentials have the same general form as in the $D \boxtimes A$ case and can be described schematically as follows.

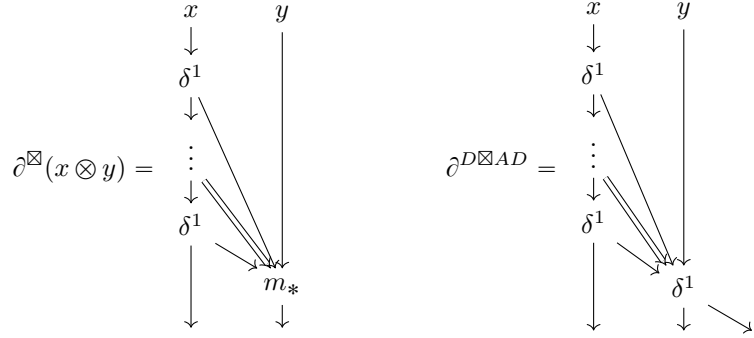
\[
\begin{aligned}
	\partial^\boxtimes (x \otimes y)
	&=
	\begin{tikzcd}[row sep=small, column sep=small]
		x \arrow[d] & y \arrow[dddd]
		\\
		\delta^1 \arrow[d] \arrow[dddr] &
		\\
		\vdots \arrow[d] \arrow[ddr, Rightarrow] &
		\\
		\delta^1 \arrow[dd] \arrow[dr] &
		\\
		& m_* \arrow[d]
		\\
		\phantom{{}={}} & \phantom{{}={}}
	\end{tikzcd}
	&
	\qquad\qquad
	\partial^{D \boxtimes AD}
	&=
	\begin{tikzcd}[row sep=small, column sep=small]
		x \arrow[d] & y \arrow[dddd] &
		\\
		\delta^1 \arrow[d] \arrow[dddr] & &
		\\
		\vdots \arrow[d] \arrow[ddr, Rightarrow] & &
		\\
		\delta^1 \arrow[dd] \arrow[dr] & &
		\\
		& \delta^1 \arrow[d] \arrow[dr] &
		\\
		\phantom{{}={}} & \phantom{{}={}} & \phantom{{}={}}
	\end{tikzcd}
\end{aligned}
\]

\subsection{Hypercubes}
We now describe the algebraic structure for the link surgery formula, which is called a {\em hypercube of chain complexes} in \cite{manolescu2010heegaard}.

Define $$ \bbE_n =   \{0,1\}^n$$ as the set of vertices of the $n$-dimensional unit hypercube. If $\varepsilon$, ${\varepsilon}^{'}\in E_n$, we write $\varepsilon\leq{\varepsilon}^{'}$ if the inequality holds for each coordinate of $\varepsilon$ and ${\varepsilon}^{'}$.

\begin{defni}
	An {\em $n$-dimensional hypercube of chain complexes} consists of a collection of vector spaces 
	$(C^{\varepsilon})_{\varepsilon \in \bbE_n},$ 
	together with a collection of linear maps for each pair of indices $\varepsilon$, ${\varepsilon}^{'}\in E_n$ such that $\varepsilon\leq{\varepsilon}^{'}$
	$$ D_{\varepsilon,{\varepsilon}^{'}}  : C^{\varepsilon} \to C^{{\varepsilon}^{'}},$$
	The maps are required to satisfy the relations whenever $\varepsilon\leq{\varepsilon}^{''}$
	\[
	\sum_{ \varepsilon\leq{\varepsilon}^{'}\leq{\varepsilon}^{''}}  D_{{\varepsilon}^{'},{\varepsilon}^{''}} \circ D_{\varepsilon,{\varepsilon}^{'}}   = 0.
    \]
\end{defni}

This construction can be generalized to define hypercubes of type-$DA$ bimodules.

\begin{defni}
	A \emph{hypercube of $DA$-bimodules} consists of a $DA$-bimodule $(C,\delta_*^1)$, such that $C$ decomposes as $C=\bigoplus_{\varepsilon\in \bbE_n}C_\varepsilon$, where each $C_\varepsilon$ is itself a $(\js,\ks)$-module. Furthermore, the structure map $\delta_{j+1}^1$ decomposes as a sum of maps
	\[\delta_{j+1,\varepsilon,\varepsilon'}^1:\calA^{\otimes j}\otimes C_\varepsilon\rightarrow C_{\varepsilon'}\otimes \calB,\]
	where $j\geq0$, and $\varepsilon\leq\varepsilon'$.
\end{defni}

\subsection{Homological perturbation}

In this subsection, we review the homological perturbation lemma used to simplify the modules below, following Zemke \cite{zemke2021bordered}.

\begin{lem}
	\label{lem:hompert2}
	Suppose that $\mathcal{A}$ and $\mathcal{B}$ are associative algebras and ${}_\mathcal{A}M^\mathcal{B}$ is a strictly unital type-$DA$ bimodule. Let $Z^\mathcal{B}$ be a type-$D$ module over $\mathcal{B}$ and consider morphisms of type-$D$ modules over $\mathcal{B}$ which fit into the diagram
	\[
	\begin{tikzcd}
		M^\mathcal{B} \arrow[loop above, "h^1"] \arrow[r, bend left=20, "\pi^1"] & Z^\mathcal{B} \arrow[l, bend left=20, "i^1"]
	\end{tikzcd}
	\]
	Moreover, we assume that $i^1$ and $\pi^1$ are homomorphisms of type-$D$ modules and that $i^1\circ \pi^1$ is a homotopy equivalence of type-$D$ modules via homotopy $h$, in the sense that 
	$i^1 \circ \pi^1 = \mathrm{Id}_M + \partial_\mathrm{Mor}(h^1),$
	which, in addition, satisfies
    \[
       \pi^1\circ i^1=\mathrm{Id}_Z,\qquad
       h^1\circ i^1=0,\qquad
       \pi^1\circ h^1=0,\qquad
       (h^1)\circ h^1=0.
    \]
	Then there is an induced $DA$-bimodule structure on ${}_\mathcal{A} Z^\mathcal{B}$ extending $Z^\mathcal{B}$ such that the maps above extend to $DA$-bimodule morphisms which satisfy all of the above relations as $DA$-bimodule morphisms.
\end{lem}

The type-$DA$ perturbation lemma gives explicit formulas for the transferred structure maps and morphisms. Schematically, these are
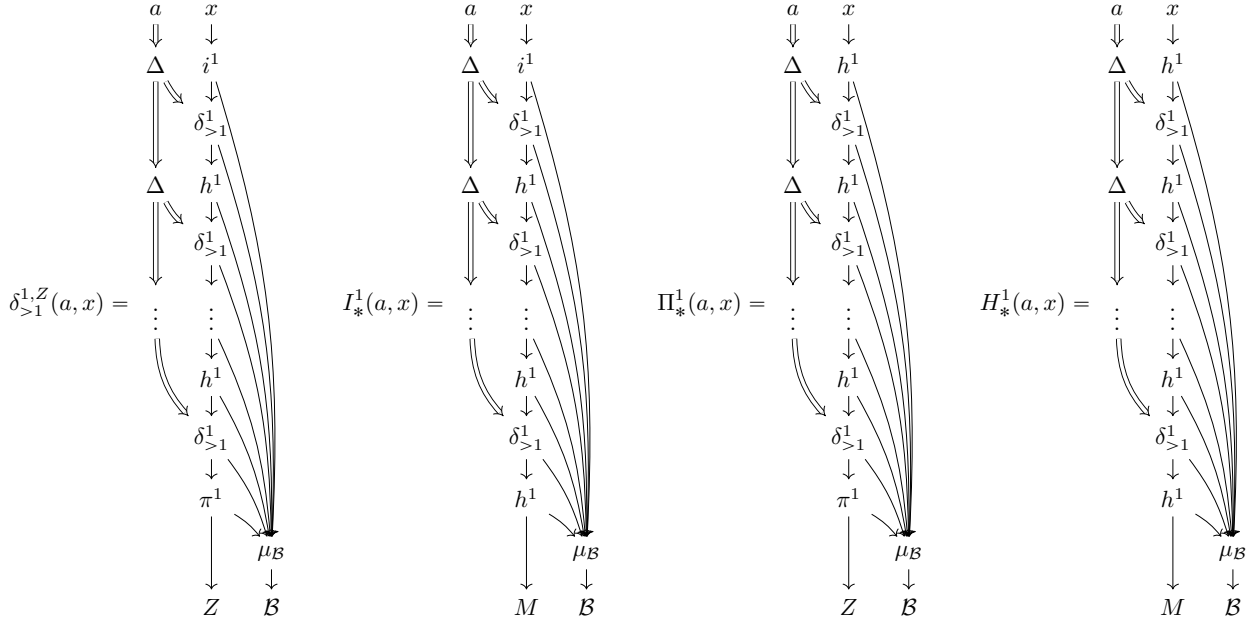
\begin{equation}\label{eq:DAperturb}
	\adjustbox{max width=\textwidth}{$\displaystyle
		\delta^{1,Z}_{>1}(a,x)
		=\DAperturbdiagram{i^1}{\pi^1}{Z}
		\qquad
		I^1_*(a,x)
		=\DAperturbdiagram{i^1}{h^1}{M}
		\qquad
		\Pi^1_*(a,x)
		=\DAperturbdiagram{h^1}{\pi^1}{Z}
		\qquad
		H^1_*(a,x)
		=\DAperturbdiagram{h^1}{h^1}{M}
		$}
\end{equation}
Here $a$ denotes the ordered left-algebra inputs, and $\delta^1_{>1}$ denotes the original structure maps with at least one left-algebra input. Each rightward arrow records an output in $\mathcal B$; these outputs are multiplied in order at $\mu_{\mathcal B}$.

There is also a version of the homological perturbation lemma for hypercubes of type-$DA$ bimodules.

\begin{lem}
	Suppose that $\calA$ and $\calB$ are associative algebras. Suppose also that $_\calA \calC^\calB=(_\calA C_\varepsilon^\calB,D_{*,\varepsilon,\varepsilon'}^1)_{\varepsilon\in\bbE_n}$ is a hypercube of $DA$-bimodules, and suppose that for each $\varepsilon\in\bbE_n$ we have chosen type-$D$ modules $(Z_\varepsilon^\calB)_{\varepsilon\in\bbE_n}$, as well as morphisms of type-$D$ modules
	\[\pi_\varepsilon^1:C_\varepsilon^\calB\rightarrow Z_\varepsilon^\calB, \quad i_\varepsilon^1:Z_\varepsilon^\calB\rightarrow C_\varepsilon^\calB, \quad \mathrm{and} \quad h_\varepsilon^1:C_\varepsilon^\calB\rightarrow C_\varepsilon^\calB\]
	satisfying
   \[\pi_\varepsilon^1\circ i_\varepsilon^1=\mathrm{Id}_{Z_\varepsilon}, \quad i_\varepsilon^1\circ\pi_\varepsilon^1+\mathrm{Id}_{C_\varepsilon}=\partial_{\mathrm{Mor}}(h_\varepsilon^1), \quad h_\varepsilon^1\circ h_\varepsilon^1=0, \quad h_\varepsilon^1\circ i_\varepsilon^1=0, \quad \mathrm{and} \quad \pi_\varepsilon^1\circ h_\varepsilon^1=0,\]
   and such that $\pi_\varepsilon^1$ and $i_\varepsilon^1$ are homomorphisms of type-$D$ modules. Such data determines structure maps $\delta_{j+1,\varepsilon,\varepsilon'}^1:\calA^{\otimes j}\otimes Z_\varepsilon\rightarrow Z_{\varepsilon'}\otimes\calB$, which make $_\calA \calZ^\calB=(_\calA\calZ_\varepsilon^\calB,\delta_{*,\varepsilon,\varepsilon'}^1)$ a hypercube of $DA$-bimodules. Furthermore, there are morphisms of hypercubes of $DA$-bimodules
   \[\Pi_*^1:_\calA \calC^\calB\rightarrow _\calA\calZ^\calB \quad I_*^1:_\calA \calZ^\calB\rightarrow _\calA\calC^\calB \quad \mathrm{and} \quad H_*^1:_\calA \calC^\calB\rightarrow _\calA\calC^\calB,\]
   which satisfy analogous relations to $i_\varepsilon^1,\pi_\varepsilon^1$ and $h_\varepsilon^1$.
\end{lem}

The induced maps again have concrete descriptions; here we focus on the structure map $\delta_{j+1,\varepsilon,\varepsilon'}^1$. This map is defined by modifying Equation \ref{eq:DAperturb} as follows. In place of each $\delta_{>1}^1$ labeled therein, we are allowed one of two maps:
\begin{enumerate}
	\item A hypercube map $D_{j+1,\varepsilon,\varepsilon'}^1$ for $\calC$, such that $\varepsilon<\varepsilon'$. Furthermore, we assume only that $j\geq0$.
	\item An internal structure map $D_{>1,\varepsilon,\varepsilon'}^1$ of some $_\calA C_\varepsilon^\calB$.
\end{enumerate}

\subsection{Linear topological spaces}\label{subsec:lineartopology}
To ensure that the chain homotopy equivalence in the link surgery formula holds, one needs to take completion on both sides. In \cite{ozsvath2008knot} and \cite{manolescu2010heegaard}, Manolescu, Ozsv\'{a}th, and Szab\'{o} achieved it by using direct product instead of direct sum. In \cite{zemke2021bordered} and \cite{zemke2023general}, Zemke discussed various linear topologies and took completions with respect to them. We review the necessary algebra background in this subsection.

\begin{defni}\label{def:lineartopology}
	A \emph{linear topological vector space} consists of a vector space $X$ equipped with a topology which admits a basis of open sets centered at 0 consisting of vector subspaces, such that the addition map $X\times X\rightarrow X$ is continuous.
\end{defni}

Given a vector space $X$, a linear topology can be specified by picking a decreasing filtration consisting of subspaces $X_i\subseteq X$ indexed by $i\in I$, where $I$ is some directed and partially ordered set. The corresponding basis consists of the sets of the form $x+X_i$, where $i\in I$ and $x\in X$.

The \emph{completion} of a linear topological vector space is given by the inverse limit
\[\bfX:=\lim_{\substack{ \leftarrow \\ i\in I}}X/X_i.\]

Note that, if $X$ and $Y$ are linear topological vector spaces, then a continuous map $f:X\rightarrow Y$ naturally induces a map on completions.

Given two linear topological vector spaces $X$ and $Y$, there are a few different ways to topologize the tensor product $X\otimes Y$.
\begin{defni}\label{def:tensortop}
	Suppose that $X$ and $Y$ are linear topological spaces.
	\begin{enumerate}
		\item \emph{standard topology} $X\otimes^!Y$: A subspace $E\subseteq X\otimes^!Y$ is open if and only if the following holds: There are open subspaces $U\subseteq X$ and $V\subseteq Y$ such that $U\otimes Y\subseteq E$ and $X\otimes V\subseteq E$.
		\item \emph{chiral-topology} $X\vec{\otimes}Y$: A subspace $E\subseteq X\vec{\otimes}Y$ is open if and only if the following holds: There is an open subspace $U\subseteq X$ such that $U\otimes Y\subseteq E$: and for each $x\in X$, there is an open $V_x\subseteq Y$ such that $x\otimes V_x\subseteq E$. 
	\end{enumerate}
\end{defni}

One important topology used in this paper is the cofinite basis topology.

\begin{defni}
	Let $(X_i)_{i\in I}$ be a family of $\bbF$-vector spaces, then the \emph{cofinite basis topology} on $\calX=\oplus_{i\in I}X_i$ is induced by the fundamental basis of opens, where $S$ are finite subsets of $I$,
	\[\calX_{\mathrm{co}(S)}=\bigoplus_{i\notin S}X_i.\]
\end{defni}
Note that the completion $\pmb{\calX}$ of $\calX$ with respect to the cofinite basis topology is the direct product
\[\pmb{\calX}\cong \prod_{i\in I}X_i.\]

Given two topologized spaces $\calX$ and $\calY$, each equipped with the cofinite basis topology, the standard topology $\calX\otimes^{!}\calY$ is equivalent to the cofinite basis topology on the tensor product.

\section{Link Floer complex and 2-handle cobordism map}\label{sec:linkfloer}
\subsection{Link Floer complex}
In this subsection, we review the ``full'' link Floer complex with the conventions of \cite{zemke2019link} and assume familiarity with multi-based Heegaard diagrams. Let $\bbL=(L,\ws,\zs)$ be an oriented multi-based link in a closed oriented $3$-manifold $Y$, with the $\ws$- and $\zs$-basepoints alternating along each component. Throughout this paper, we work with links $L=K_1\cup\cdots\cup K_l$ for which each component $K_i$ has exactly one pair of basepoints $w_i,z_i$. We call these \emph{based inks}.

Let $\calH=(\Sigma,\alphas,\betas,\ws,\zs)$ be a Heegaard diagram for $(Y,\bbL)$, and let $g$ be the genus of $\Sigma$. We use meridianal Heegaard diagrams as defined below.

\begin{defni}
	A Heegaard link diagram $(\Sigma,\alphas,\betas,\ws,\zs)$ is a \emph{meridianal link diagram} if for each component $K_i$ of $L$, there is a special beta curve $\beta_{0,i}$ such that $w_i$ and $z_i$ are on opposite sides of $\beta_{0,i}$, as in Figure \ref{fig:meridianalHD}.
\end{defni}

\begin{figure}[h]
	\centering
	\begin{tikzpicture}[scale=1.2, line cap=round]
		
		\draw[thick] (-2,0) to[out=85,in=95] (2,0);
		\draw[thick] (-1,0) to[out=50,in=130] (1,0);
		\draw[blue] (0,0.46) to[out=70,in=280] (0,1.15);
		\draw[red] (-1.5,0) to[out=50,in=130] (1.5,0);
		\node[blue,font=\small] at (0.25,1) {$\beta_0$};
		\node[red,font=\small] at (-1, 0.55) {$\alpha$};
		\node[right,font=\small] at (0,0.8) {$z$};
		\node[left,font=\small] at (0.1,0.8) {$w$};
		
	\end{tikzpicture}
	\caption{A local picture for the meridianal link diagram}
	\label{fig:meridianalHD}
\end{figure}
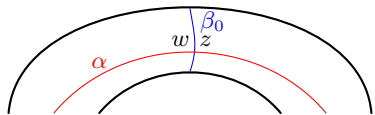
	
Given a pair $(Y,\bbL)$ with link Heegaard diagram $\calH=(\Sigma, \alphas , \betas, \ws, \zs)$, it determines tori
$$\Ta = \alpha_1 \times \cdots \times \alpha_{g+l-1}, \ \Tb=\beta_1 \times \cdots \times \beta_{g+l-1} \subset \Sym^{g+l-1}(\Sigma).$$ The \emph{link Floer complex} $\CFL(Y,L)$ is freely generated by the intersection point $\bfx\in \Ta \cap \Tb$ over the polynomial ring $\bbF[\scrU_1,\scrV_1,\ldots,\scrU_l,\scrV_l]$. The differential is given by counting pseudo-holomorphic disks via:
\begin{equation*}
	\label{def:DefHFLm}
	\partial \bfx = \sum_{\bfy \in \Ta \cap \Tb} \sum_{\substack{\phi \in \pi_2(\bfx, \bfy)\\ \mu(\phi)=1 }} \# (\calM(\phi)/\bbR) \scrU_{1}^{n_{w_1}(\phi)} \cdots   \scrU_{l}^{n_{w_l}(\phi)}\scrV_{1}^{n_{z_1}(\phi)} \cdots   \scrV_{l}^{n_{z_l}(\phi)}\bfy, 
\end{equation*}
where $\pi_2(\bfx, \bfy)$ is the set of homotopy classes of Whitney disks from $\bfx$ to $\bfy$ relative to $\Ta$ and $\Tb$, as in \cite{ozsvath2004holomorphic}. For each homotopy class of disks $\phi \in \pi_2(\bfx, \bfy)$, we denote by $n_{w_j}(\phi)$ and $n_{z_j} (\phi) \in \bbZ$ the multiplicity of $w_j$ (resp. $z_j$) in the domain of $\phi$. Furthermore, we let $\mu(\phi)$ be the Maslov index of $\phi$. The corresponding homology is denoted by $\HFL(Y,L)$.

In this paper, we work with links $L$ in $S^3$, and we denote the link Floer complex and its homology by $\CFL(L)$ and $\HFL(L)$ respectively. Each generator $\xs$ of the link Floer complex is bigraded by the Maslov grading $\gr_{\ps}$, which depends on a choice of collection of basepoints $\ps$, and a relative Alexander multi-grading $A=(A_1,\cdots,A_l)$, which takes value in the following set:
	
$$ \bbH(L)_i = \frac{\lk(L_i, L - L_i)}{2} + \bbZ \subset \bbQ, \ \ \bbH(L) = \bigtimes_{i=1}^\ell \bbH(L)_i,$$ 
where $\lk$ denotes linking number.

The relative Alexander grading is defined by
$$ A_i(\bfx) - A_i(\bfy) = n_{z_i}(\phi) - n_{w_i}(\phi),$$
where $\phi$ is any class in $\pi_2(\bfx, \bfy)$. Recall that each intersection point $\bfx\in \Ta \cap \Tb$ is assigned a relative $\spin^{c}$ structure on $(Y,L)$, via a construction of non-vanishing vector fields, which we denote by $\underline{\mathfrak{s}}(x)$.
The relative Alexander multi-grading can be lifted to an absolute grading using the identification between the relative $\spin^{c}$ structure with $\bbH(L)$, see \cite{ozsvath2008holomorphic}. The bigrading can be extended over the coefficient ring $\bbF[\scrU_1,\scrV_1,\cdots,\scrU_l,\scrV_l]$ by the following rules. For the Maslov grading, $\scrU_i$ has $\gr_{\ps}$-grading $-2$ if $w_i\in\ps$ and grading $0$ otherwise. Similarly, $\scrV_i$ has $\gr_{\ps}$-grading $-2$ if $z_i\in\ps$ and grading $0$ otherwise. For the Alexander grading, $A(\scrV_i)=-A(\scrU_i)=\es_i$, where $\es_i$ is the standard $i$-th coordinate vector in $\bbR^{l}$. The $\scrU_i\scrV_i$ actions are homotopic to each other for different $i$, and we denote it by $U:=\scrU_i\scrV_i$.
\subsection{2 handle cobordism maps}
In this subsection, we review the link cobordism map associated to a $4$-dimensional $2$-handle attachment. The handles are attached along a framed link disjoint from the based link, which is carried through the cobordism. Following \cite{zemke2019link}, we describe the map using a $\beta$-bouquet and a subordinate Heegaard triple. This description will be used in Section~\ref{sec:surgerycob} to identify the cobordism map with an inclusion of subcomplexes in the link surgery formula. Suppose $\bbL=(L,\ws,\zs)$ is a based link in $Y$ and $\bbS^1$ is a framed 1-dimensional link in $Y\setminus L$. We denote the link cobordism induced by the 2-handle attachments on $\bbS^1$ by $W(Y,\bbS^1)$, and the corresponding surgery 3-manifold by $Y(\bbS^1)$. Let $\bbL_\alpha$ denote the closure of the components of $L\setminus(\ws\cup\zs)$ which are oriented to start in $\ws$ and end in $\zs$. Let $\bbL_\beta$ denote the closure of the components of $L\setminus(\zs\cup\ws)$ which start in $\zs$ and end in $\ws$.

\begin{defni}
	An $\beta$-\textit{bouquet}, $\calB^\alpha$, for a framed link $\bbS^1$ in $Y\setminus L$ is a collection of arcs which connect each of the component of $\bbS^1$ to one of the components of $\bbL_\beta$. We assume that there is exactly one arc per components of $\bbS^1$, and each arc has one endpoint on $\bbS^1$ and one endpoint in $\bbL_\beta$. 
\end{defni}

Given a bouquet $\calB$ for a framed link $\bbS^1\subset Y\setminus L$, we use $\underline{\calB}$ to denote the union of the bouquet $\calB$ with the arcs in $\bbL$ that it connects to. We also use $\ws(\calB)$ and $\zs(\calB)$ for the collection of basepoints on the corresponding arcs. One can construct the sutured manifold $Y(\calB)$ obtained by removing a regular neighborhood of $\bbS^1\cup \calB$, and adding sutures corresponding to the basepoints $\ws(\calB)$ and $\zs(\calB)$.

\begin{defni}\label{def:bouquetsubHD}
	Suppose $\bbL=(L,\ws,\zs)$ is a based link in $Y$ and $\bbS^1$ is a framed 1-dimensional link in $Y\setminus L$. We say that a triple $(\Sigma,\alpha_1,\ldots,\alpha_n,\beta_1,\ldots,\beta_n,\beta_1',\ldots,\beta_n',\ws,\zs)$ is \emph{subordinate} to a $\beta$-bouquet $\calB^\beta$ for $\bbS^1$ if the following are satisfied:
	\begin{enumerate}
		\item After removing neighborhoods of the $\ws(\calB^\beta)$ and $\zs(\calB^\beta)$-basepoints, the diagram
		\[(\Sigma,\alpha_1,\ldots,\alpha_n,\beta_{k+1},\ldots,\beta_n,\ws,\zs)\]
		becomes a sutured diagram for $Y(\calB^\beta)$.
		\item The curves $\beta_1,\ldots,\beta_k$ are each a meridian of a different component of $\bbS^1$. In particular
		\[(\Sigma,\alpha_1,\ldots,\alpha_n,\beta_1,\ldots,\beta_n,\ws,\zs)\]
		is a diagram for $(Y,\bbL)$.
		\item The curves $\beta_{k+1}',\ldots,\beta_n'$ are each small Hamiltonian translates of the curves $\beta_{k+1},\ldots,\beta_n$ respectively, with $\abs{\beta_i\cap\beta_j'}=2\delta_{ij}$.
		\item For $i=1,\ldots,k$, the curve $\beta_i'$ is induced by the framing of the corresponding link component that $\beta_i$ is a meridian of.
	\end{enumerate}
\end{defni}

If $(\Sigma,\alpha,\beta,\beta',\ws,\zs)$ is subordinate to a $\beta$-bouquet, then $(\Sigma,\beta,\beta',\ws,\zs)$ is a link Heegaard diagram for a $\abs{\ws}$ components unlink in $(S^1\times S^2)^{\#(n-k)}$. Hence, $\HFL(\Sigma,\beta,\beta',\ws,\zs,\frs_0)$ admits a distinguished element $\Theta_{\beta,\beta'}^+$, see Lemma 3.6 in \cite{zemke2019link}. Here, $\frs_0$ is the unique torsion $\spin^c$ structure on $(S^1\times S^2)^{\#(n-k)}$, and $\Theta_{\beta,\beta'}^+$ is the generator with the maximum $\gr_{\ws}$ grading.

Given a Heegaard triple $(\Sigma,\alpha,\beta,\beta',\ws,\zs)$, one can construct a 4-manifold $X_{\alpha,\beta,\beta'}$ with boundary $-Y$, $Y(\bbS^1)$, and $(S^1\times S^2)^{\#(n-k)}$, see \cite{ozsvath2006holomorphic}. One can fill in $(S^1\times S^2)^{\#(n-k)}$ with 3- and 4- handles to get a 4-manifold which is diffeomorphic to $W(Y,\bbS^1)$. Hence, given $\frs\in \spin^c(W(Y,\bbS^1))$, we can restrict $\frs$ to $\spin^c(X_{\alpha,\beta,\beta'})$, and the $\beta$-subordinate 2-handle map
\[F_{Y,\bbL,\bbS^1,\frs}^{\beta}:\CFL(Y,\bbL,\frs|_Y)\rightarrow\CFL(Y(\bbS^1),\bbL,\frs|_{Y(\bbS^1)})\]
is defined via the formula
\begin{equation}\label{eq:linkcobmap}
	F_{Y,\bbL,\bbS^1,\frs}^{\beta}(\xs):=F_{\alpha,\beta,\beta',\frs}(\xs\otimes\Theta_{\beta,\beta'}^+)=\sum_{\bfy \in \Ta \cap \mathbb{T}_{\beta'}} \sum_{\substack{\psi\in\pi_2(\xs,{\Theta_{\beta,\beta'}^{+}},\ys)\\ \mu(\psi)=0 \\ \frs_{\ws}(\psi)=\frs }} \# (\calM(\psi)/\bbR) \scrU_{1}^{n_{w_1}(\psi)} \cdots   \scrU_{l}^{n_{w_l}(\psi)}\scrV_{1}^{n_{z_1}(\psi)} \cdots   \scrV_{l}^{n_{z_l}(\psi)}\bfy.
\end{equation}

\subsection{Link surgery formula}\label{subsec:surgeryformula}
Given a knot surgery description of a 3-manifold, Ozsv\'{a}th and Szab\'{o} developed the knot surgery formula \cite{ozsvath2008knot} to compute Heegaard Floer homology of the surgery manifold in terms of the knot Floer complex. This was later generalized by Manolescu and Ozsv\'{a}th \cite{manolescu2010heegaard} to link surgery formula, which applies when the 3-manifold is described as a surgery on a link in an integral homology 3-sphere. 

Let $(L,\Lambda)$ be a framed link in an integral homology 3-sphere $Y$, we also fix a meridianal link diagram of $L$. The input data for the link surgery complex is called a complete system of hyperboxes for the link $L$ (see \cite[Section~8]{manolescu2010heegaard}), which can be derived from the meridianal link diagram together with an arc system. The output is a hypercube of chain complexes. We follow the convention from \cite{zemke2023general}, and denote it by $\mathbb{X}_\Lambda(L)$. At each vertex of the hypercube, one assigns the Heegaard Floer chain complex associated to a reduction of the meridianal link diagram, while the differential maps between different vertices also depend on the chosen arc system. The homology of the link surgery hypercube is isomorphic to the Heegaard Floer homology of the surgery manifold. 
		
\begin{Theorem}\cite[Theorem~1.1]{manolescu2010heegaard}\label{linksurgeryformula}\label{theorem:lsf}
Fix a complete system of hyperboxes for an oriented link $L$ in an integral homology three-sphere $Y$,
and fix a framing $\Lambda$ of $L$. There is an isomorphism of
homology groups:
\begin {equation*}
\label {eq:surgery0}
H_*(\mathbb{X}_\Lambda(L)) \cong  \HFm_*(Y_\Lambda(L)),
\end {equation*}
where $\HFm$ is the
completed version of Heegaard Floer homology over the power series
ring $\bbF[[U]]$.
\end{Theorem}
	
\begin{Remark}
In Theorem \ref{linksurgeryformula}, the link surgery complex involves direct product, which is different from the knot surgery formula in \cite{ozsvath2008knot}, where a directed sum is used. Using direct sum is only appropriate for computing the plus and hat flavor using an analogous surgery formula. For the minus flavor, the isomorphism does not hold (see the unknot computation in \cite[Section~4]{manolescu2010heegaard}). To achieve the isomorphism for the minus flavor, we need to complete the direct sum, which gives us the direct product. For the same reason, when working with the bordered aspect from \cite{zemke2021bordered}, we also need to take completion with respect to a chosen linear topology, as defined in Subsection \ref{subsec:lineartopology}.
\end{Remark}

The link surgery formula of Manolescu, Ozsv\'{a}th and Szab\'{o} has the restriction that the framed link must be null-homologous. This was extended by Zemke in \cite{zemke2023general} to a general surgery formula for arbitrary links in closed 3-manifolds. The extension uses a new perspective, which is established by studying the exact triangles in the Fukaya category of the torus. We find this new perspective convenient for proving the 2-handle link cobordism map in Subsection \ref{sec:surgerycob}, and briefly review the construction here.

The objects under consideration are Lagrangians with local systems: to each Lagrangian $L$, one assigns a vector space $E$, together with a monodromy representation of the fundamental groupoid of $L$ into $\mathrm{Aut}(E)$. In the setting of general surgery formula, the relevant Lagrangians are tori $\bbT_\gamma=\gamma_1\times\dots\times\gamma_n\subseteq \mathrm{Sym}^n(\Sigma)$ for a set of attaching curves $\gs\subseteq\Sigma$. Here, $\Sigma$ is a surface with basepoints. The monodromy is constructed by choosing a codimension 1 hypersurface $S$ in $\bbT_\gamma$ and counting intersection numbers with $S$. The vector spaces are $\bbF[U]$ modules. More precisely, the hypersurface is
\[(K\times \mathrm{Sym}^{n-1}(\Sigma))\cap \bbT_\gamma,\]
where $K$ is a simple closed curve on the Heegaard diagram. Together with an algebra morphism
\[\rho:\bbF[\bbZ]\rightarrow\mathrm{End}_{\bbF[U]}(E),\]
one defines the monodromy representation by counting intersection numbers. When there is no risk of confusion, we abbreviate $\gs^{E}$ for the triple $(\bbT,E,\rho)$.

A Floer morphism from $\gs^{E}$ to $\gs^{E'}$ consists of a pair $\langle\xs,\phi\rangle$ where $\xs\in\bbT_\gamma\cap\bbT_{\gamma'}$ and $\phi:E\rightarrow E'$ is an $\bbF[U]$-equivariant map. Given a collection of Lagrangians with local systems $(\gs_1,\rho_1),\dots,(\gs_1,\rho_1)$, the holomorphic polygon maps are defined as the following. Let
\[\langle\xs_{i,i+1},\phi_{i,i+1}\rangle:\gs_i^{E_i}\rightarrow\gs_{i+1}^{E_{i+1}}\]
be a sequence of Floer morphisms, for $1\leq i\leq n-1$. The holomorphic polygon map counts holomorphic $n$-gons of index $3-n$ with the inputs $\xs_{i,i+1}$. The outputs will consist of Floer morphisms from $\gs_1^{E_1}\rightarrow\gs_{n}^{E_{n}}$, which are pairs $\langle\ys,f\rangle$. Given a polygon $\psi$, let $m_i:=\#(\partial_{\gamma_i}(\psi)\cap K)$. The morphism $f$ is the composition
\begin{equation}\label{eq:localsystem}
	U^{n_w(\psi)}(\rho_n(e^{m_n})\circ\phi_{n-1,n}\circ\rho_{n-1}(e^{m_{n-1}})\circ\dots\circ\phi_{1,2}\circ\rho_1(e^{m_1})),
\end{equation}
where $w$ is a basepoint on $\Sigma$.

Following \cite{zemke2023general}, we first review the knot surgery complex. Fix a pair $(Y,K)$ and a meridianal Heegaard diagram $(\Sigma,\as,\bs,w,z)$ for $(Y,K)$. A \emph{shadow} of $K$ in $\Sigma$ is a concatenation of two embedded arcs connecting $w$ to $z$ such that one arc avoids $\alphas$ and the other arc avoids $\betas$. Pushing the interior of the arc which avoids $\alphas$ (resp. $\betas$) into the alpha-handlebody (resp. beta-handlebody) yields a copy of the knot $K$. Abusing notation, we also denote the shadow by $K\in \Sigma$. Note that, picking a shadow of $K$ induces a Morse framing of the knot. Since we are working with a meridianal Heegaard diagram, the shadow $K$ intersects only the special meridianal beta curve $\beta_0$, and the neighborhood of $K\cup\beta_0$ is a punctured torus which is disjoint from the beta curves except for $\beta_0$. This torus is called the \emph{special toroidal region} for $K$. A local picture is given in Figure \ref{fig:localpic}, where we also wind $\beta_0$ to obtain a new curve, denoted by $\beta_1$. Let $\betas_0$ denote the collection of $\betas$ curves from the Heegaard diagram. Define $\betas_1$ to be the collection of curves obtained by taking $\beta_1$, and adjoining small translates of the curves from $\betas\setminus\beta_0$. These two collections of curves determine two Lagrangian tori in the symmetric product, which, by abusing of notation, we also denote by $\betas_0$ and $\betas_1$. The local systems are given by:
\begin{enumerate}
	\item $\betas_0$ is decorated with the module $E_0:=\bbF[\scrU,\scrV]$. The monodromy
	\[\rho_0:\bbF[\bbZ]\rightarrow \mathrm{Hom}_{\bbF[U]}(E_0,U^{-1}E_0)\]
	sends $e^n$ to $\scrV^n\cdot\mathrm{id}$.
	\item $\betas_1$ is decorated with the module $E_1:=\bbF[\scrU,\scrV,\scrV^{-1}]$. The monodromy
	\[\rho_1:\bbF[\bbZ]\rightarrow \mathrm{Hom}_{\bbF[U]}(E_1,E_1)\]
	sends $e^n$ to $\scrV^n\cdot\mathrm{id}$.
\end{enumerate}

\begin{figure}[h]
	\centering
	\begin{tikzpicture}
		\draw[dashed] (0,0) rectangle (4,4);
		\draw[dashed] (0,1)--(4,1);
		\draw[blue] (2,0)--(2,4);
		\draw[red] (0,2)--(4,2);
		\draw[violet] (3,4) to[out=190, in=0] (0,2.9);
		\draw[violet] (3,0) to[out=170, in=0] (0,1.1);
		\draw[violet] (4,1.1) to[out=170, in=190, looseness=5] (4,2.9);
		\node[fill=white, inner sep=1pt] at (1.5,1) {$w$};
		\node[fill=white, inner sep=1pt] at (2.5,1) {$z$};
		\node at (3.5,0.5) {$K$};
		\node[blue] at (2.3,3.3) {$\beta_0$};
		\node[violet] at (3.5,1.5) {$\beta_1$};
		\node[red] at (1,2.3) {$\beta_\lambda$};
		\fill (2,0.37) circle (2pt);
		\node[right=3pt] at (2,0.37) {$\theta_{\tau}^+$};
		\fill (2,1.58) circle (2pt);
		\node[right=3pt] at (2,1.7) {$\theta_{\sigma}^+$};
	\end{tikzpicture}
	\caption{The diagram $(\bbT^2,\beta_0,\beta_1,w,z)$}
	\label{fig:localpic}
\end{figure}
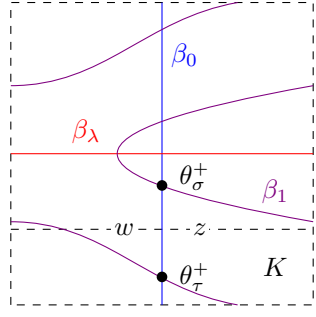

There are two morphisms
\[\langle\theta_{\sigma}^+,\phi^\sigma\rangle,\langle\theta_{\tau}^+,\phi^\tau\rangle:\betas_0^{E_0}\rightarrow\betas_1^{E_1},\]
used in the knot surgery formula. The algebra morphisms
\[\phi^\sigma,\phi^\tau:\bbF[\scrU,\scrV]\rightarrow\bbF[\scrU,\scrV,\scrV^{-1}]\]
are given by the formulas
\[\phi^\sigma(\scrU^i\scrV^j)=\scrU^i\scrV^j \qquad \mathrm{and}\\ \qquad \phi^\tau(\scrU^i\scrV^j)=\scrU^j\scrV^{2j-i}.\]
The intersection points of $\theta_{\sigma}^+$ and $\theta_{\tau}^+$ in the special toroidal region are shown in Figure \ref{fig:localpic}. Outside the local region, they coincide with top degree intersection points between curves of $\betas$ and their translates.

Denote the Morse framing of $K$ by $\lambda$. The knot surgery complex $\bbX_{\lambda}(K)$ is defined to be the $\bbF[U]$-chain complex
\[\bbX_{\lambda}(K)=\mathrm{Cone}(v+h:\boldsymbol{CF}^-(\alphas,\betas_0^{E_0})\rightarrow\boldsymbol{CF}^-(\alphas,\betas_1^{E_1})),\]
where $v$ counts holomorphic triangles with inputs $\langle\theta_{\sigma}^+,\phi^\sigma\rangle$ and $h$ counts holomorphic triangles with inputs $\langle\theta_{\tau}^+,\phi^\tau\rangle$.

Consider another collection of curves $\betas_\lambda$, which consists of a push-off of the knot shadow: $\beta_\lambda$, together with small translates of the curves from $\betas\setminus\beta_0$. We decorate it with the trivial local system:
\[E_\lambda=\bbF[U]\qquad \mathrm{and} \qquad \rho(e^n)=\mathrm{id}.\]
The Heegaard Floer chain complex of the surgery manifold is given by: $\boldsymbol{CF}^-(Y_\lambda(K))=\boldsymbol{CF}^-(\alphas,\betas_\lambda).$ The homotopy equivalence between $\boldsymbol{CF}^-(\alphas,\betas_\lambda)$ and $\mathrm{Cone}(v+h:\boldsymbol{CF}^-(\alphas,\betas_0^{E_0})\rightarrow\boldsymbol{CF}^-(\alphas,\betas_1^{E_1}))$ is proved in \cite{zemke2023general} by showing an equivalence in the Fukaya category.
\begin{Theorem}\cite[Theorem~1.1]{zemke2023general}
	There is an equivalence in the Fukaya category
	\[\betas_\lambda\simeq\mathrm{Cone}(\langle\theta_{\sigma}^+,\phi^\sigma\rangle+\langle\theta_{\tau}^+,\phi^\tau\rangle: \betas_0^{E_0}\rightarrow\betas_1^{E_1}).\]
\end{Theorem}

The proof proceeds by explicitly constructing a chain homotopy equivalence; see \cite[Section~5]{zemke2023general}. We recall only the homotopy map from the mapping cone to $\betas_\lambda$, which will be used in Subsection \ref{subset:surgerycob}. The homotopy map is given by:
\[
\Psi =
\begin{tikzcd}[column sep=large, row sep=large]
	\betas_0^{E_0}
	\arrow[r, dashed]
	& \betas_1^{E_1}
	\arrow[d, "{\langle\theta_{\betas_1,\betas_\lambda},\Pi\rangle}"']\\
	&\betas_\lambda
\end{tikzcd}
\]

The generator $\theta_{\betas_1,\betas_\lambda}$ consists of the unique intersection point between $\beta_\lambda$ and $\beta_1$, together with the top degree intersection points between curves of $\betas$ and their translates. $\Pi$ is the map given by
\[\Pi(\scrU^i\scrV^j)=\begin{cases}
	\scrU^i\scrV^j & \mathrm{if} \quad i=j,\\
	0 &\mathrm{otherwise}.
\end{cases}\]

The construction can be generalized to link surgery as follows. Let $(\Sigma,\alphas,\betas,\ws,\zs)$ be a meridianal Heegaard link diagram for $(Y,L)$ and pick embedded shadows of the components $K_1,\dots,K_n$ of $L$ which are pairwise disjoint. The union of each shadow with its special beta meridianal curve is a punctured torus, giving the same local picture as in Figure \ref{fig:localpic}. Write $\beta_{0,i}$ for the special meridian of $K_i$ and write $\beta_{1,i}$ for a copy of $\beta_1$ curve from the local punctured torus. One can then construct a hypercube of attaching curves $\betas_\varepsilon$ indexed by $\varepsilon\in \bbE_n$. This hypercube is denoted by $\mathscr{B}_\Lambda$. Let $\betas'$ be the union of the remaining beta curves which are not special meridianal curves, that is, $\betas=\betas'\cup\{\beta_{0,1},\dots,\beta_{0,n}\}$. At a vertex $\varepsilon$, $\betas_\varepsilon$ is given by taking small translates of $\betas'$, as well as a small translate of $\beta_{0,i}$ for $i$ such that $\varepsilon_i=0$, and a small translate of $\beta_{1,i}$ for $i$ such that $\varepsilon_i=1$. Small Hamiltonian translations are applied to the curves in each $\betas_\varepsilon$ to achieve admissibility and ensure that if two copies of a given beta curve appear in two collection, then they intersect in a pair of points.

Each collection $\betas_\varepsilon$ is decorated with the $\bbF[U_1,\dots,U_n]$-module
\[E_\varepsilon:=E_{\varepsilon_1}\otimes_\bbF\cdots\otimes_\bbF E_{\varepsilon_n}.\]
The algebra morphisms are extended to
\[
\phi_i^\sigma=\mathrm{id}^{\otimes(i-1)}\otimes\phi^\sigma\otimes\mathrm{id}^{\otimes(n-i)},
\qquad
\phi_i^\tau=\mathrm{id}^{\otimes(i-1)}\otimes\phi^\tau\otimes\mathrm{id}^{\otimes(n-i)},
\quad 1\leq i\leq n.
\]
If $\varepsilon<\varepsilon'$, and $\abs{\varepsilon'-\varepsilon}_1=1$ (where $\abs{-}_1$ denotes the $L^1$-norm), there are two morphisms
\[\langle\theta_{\sigma,\varepsilon,\varepsilon'}^+,\phi_i^\sigma\rangle \qquad \mathrm{and} \qquad \langle\theta_{\tau,\varepsilon,\varepsilon'}^+,\phi_i^\tau\rangle.\]
The intersection points are obtained by taking the corresponding intersection points in the $i$th local torus, where $i$ is the coordinate in which $\varepsilon$ and $\varepsilon'$ are different, together with top degree intersection points outside this local torus. We abbreviate these morphisms as $\theta_{\sigma,i}$ and $\theta_{\tau,i}$.

The link surgery complex is obtained by pairing
\[\bbX_{\Lambda}(Y,L):=\boldsymbol{CF}^-(\alphas,\mathscr{B}_\Lambda).\]
 
Given $\varepsilon<\varepsilon'$, let $M$ be the sublink of $L$, where $K_i\subseteq M$ if $\varepsilon_i<\varepsilon_i'$. The differential $D_{\varepsilon,\varepsilon'}$ between $\bbX_{\Lambda}^\varepsilon(Y,L)$ and $\bbX_{\Lambda}^{\varepsilon'}(Y,L)$ decomposes as a sum
\[D_{\varepsilon,\varepsilon'}=\sum_{\vec{M}}D_{\vec{M}},\]
ranging over the orientations of $M$. Each $D_{\vec{M}}$ is defined by counting holomorphic $(\abs{\varepsilon'-\varepsilon}_1+2)$-gons, with input $\theta_{\sigma,j}$ when the orientation of $\vec{M}_j$ agrees with the orientation inherited from $L$, and input $\theta_{\tau,j}$ if the orientation of $\vec{M}_j$ differs from it. An Alexander grading $\bbH(L)$ is assigned to $\bbX_{\Lambda}(Y,L)$ as well. Note that, each component $(K_i,\lambda_i)$ determines an element of $H_1(Y-L)\cong \bbZ^n$, which we denote by $\Lambda_i$. Then the change in Alexander grading is given by
\[A(D_{\vec{M}})=\sum_{j}\Lambda_j,\]
where the sum is taken over the components $M_j$, whose orientation differs from the orientation inherited from $L$.

Let $\betas_\Lambda$ denote the set of curves obtained by replacing the special meridianal curves with the corresponding shadows. Equip $\betas_\Lambda$ with the trivial local system $\bbF[U_1,\dots,U_n]$, then $\boldsymbol{CF}^-(\alphas,\betas_\Lambda)=\boldsymbol{CF}^-(Y_\Lambda(L))$. As in the knot surgery case, the chain homotopy equivalence between $\bbX_{\Lambda}(Y,L)$ and $\boldsymbol{CF}^-(\alphas,\betas_\Lambda)$ is induced by a homotopy equivalence between hypercubes of attaching curves with local systems
\[\betas_\Lambda\simeq \mathscr{B}_\Lambda.\]
The chain homotopy is constructed by taking the connected sum with a genus-one surface and using the chain homotopy defined on these summands, see \cite[Section~6]{zemke2023general}. The homotopy equivalence $\Gamma:\boldsymbol{CF}^-(\alphas,\mathscr{B}_\Lambda)\rightarrow\boldsymbol{CF}^-(\alphas,\betas_\Lambda)$ decomposes as a sum
\begin{equation}\label{eq:chainhomotopy}
	\Gamma=\sum_{\vec{M}\subseteq L}\Gamma^{\vec{M}}
\end{equation}
ranging over oriented sublinks $\vec{M}\subseteq L$ (the empty link is allowed, which corresponds to the  summand which is nontrivial on $\bbX_{\Lambda}^{(1,\dots,1)}(L)$). $\Gamma^{\vec{M}}$ is only nontrivial on the vertex $\varepsilon$ determined by
\[\varepsilon_i=0 \quad \xLeftrightarrow{ } \quad \pm K_i\subseteq\vec{M}.\]
The map $\Gamma^{\vec{M}}$ counts holomorphic $(\abs{\vec{M}}+3)$-gons which have a special input for each component of $\vec{M}$, and a final input of $\langle\theta_{\beta_{(1,\dots,1),\beta_\Lambda}}^+,\Pi\rangle$, where $\Pi$ is $U_i$ equivariant and satisfies
\[\Pi(\scrV_1^{i_1}\dots\scrV_n^{i_n})=\begin{cases}
	1 & \mathrm{if} \ i_1=\cdots=i_n=0\\
	0 & \mathrm{otherwise}.
\end{cases}\]

The special inputs corresponding to components of $\vec{M}$ are length one chains in $\mathscr{B}_\Lambda$. If $+K_i\subseteq\vec{M}$, then the corresponding generator is $\theta_{\sigma,i}$, and if $-K_i\subseteq\vec{M}$, then the corresponding generator is $\theta_{\tau,i}$.
\section{2 handle cobordism maps from the link surgery formula}\label{sec:surgerycob}
One important feature of the link surgery formula is that one can study 2 handle cobordism maps via the inclusion of a subcomplex of the link surgery formula \cite[Section~14]{manolescu2010heegaard}. Let $(L,\Lambda)=(L_1,\Lambda_{L_1})\cup(L_2,\Lambda_{L_2})$ be an oriented link in $S^3$, where $L$ has $n$ components, and $L_1$ (respectively, $L_2$) has $n_1$ (respectively, $n_2$) components. Let $W_{\Lambda_{L_2}}(L_1,L)$ be the cobordism from $S^3_{\Lambda_{L_1}}(L_1)$ to $S^3_{\Lambda}(L)$ given by 2 handle attachment along $L_2$, with framing $\Lambda_{L_2}$. There is an identification between the Alexander grading and $\mathrm{Spin}^c$ structures on $W_{\Lambda_{L_2}}(L_1,L)$ \cite[Lemma~11.2]{manolescu2010heegaard}:
\begin{equation}\label{eq:idspinA}
	\mathrm{Spin}^c(W_{\Lambda_{L_2}}(L_1,L))\cong \bbH(L)/ H(L_1,\Lambda_{L_1}),
\end{equation}
where $H(L_1,\Lambda_{L_1})$ is generated by the classes $\Lambda_i$ represented by the $\lambda_i$-framed push-offs of $K_i$, for $(K_i,\lambda_i)\in (L_1,\Lambda_{L_1})$. The identification is given by sending $[\ve{s}\in \bbH(L)/ H(L_1,\Lambda_{L_1})]$ to the $\mathrm{Spin}^c$ structure with first Chern class
\[c_1([\ve{s}])=[2\ve{s}-(\Lambda_1+\cdots+\Lambda_n)].\]
Moreover, the natural projection
\[\pi^{L,L_1}:(\bbH(L)/ H(L_1,\Lambda_{L_1}))\rightarrow(\bbH(L)/ H(L,\Lambda))\]
corresponds to restricting the $\spin^c$ structures to $S^3_{\Lambda}(L)$, and the map
\[\psi^{L-L_1}:(\bbH(L)/ H(L_1,\Lambda_{L_1}))\rightarrow(\bbH(L_1)/ H(L_1,\Lambda_{L_1})),\]
defined on each component $K_i$ of $L_1$ by 
\[{\psi}^{L-L_1}_{i}: s_{i}\rightarrow s_{i}-\frac{\mathrm{lk}(K_{i},L_2)}{2},\]
corresponds to restricting the $\spin^c$ structures to $S^3_{\Lambda_{L_1}}(L_1)$.

For the $n$-components link $L$, the corresponding link surgery formula $\bbX_{\Lambda}(S^3,L)$ is an $n$-dimensional hypercube. Let $\bbX_{\Lambda_{L_1}}(S^3,L_1)$ be the $n_1$-dimensional subcube for which all coordinates corresponding to $L_2$ are equal to 1. Recall that, the differentials of $\bbX_{\Lambda_{L_1}}(S^3,L_1)$ are labeled by oriented sublinks $\vec{M}\subseteq L_1$, and the Alexander grading change is given by the sum of $\Lambda_i$ over those components $K_i$ whose orientation is opposite to the orientation inherited from $L_1$. Hence, the subcube $\bbX_{\Lambda_{L_1}}(S^3,L_1)$ admits a splitting with respect to $\bbH(L)/ H(L_1,\Lambda_{L_1})$:
\[\bbX_{\Lambda_{L_1}}(S^3,L_1)=\bigoplus_{\ve{t}\in\bbH(L)/ H(L_1,\Lambda_{L_1})}\bbX_{\Lambda_{L_1}}(S^3,L_1,\ve{t}).\]
The cobordism maps are given by the inclusion of these subcomplexes.
\begin{Theorem}\cite[Theorem~14.3]{manolescu2010heegaard}\label{theorem:cubecob}
	For any $\ve{t}\in\bbH(L)/ H(L_1,\Lambda_{L_1})$, the following diagram commutes:
	\[
	\begin{tikzcd}[column sep=5em, row sep=5em]
		\bbX_{\Lambda_{L_1}}(S^3,L_1,\ve{t}) \arrow[r] \arrow[d, "\cong"'] & \bbX_{\Lambda}(S^3,L) \arrow[d, "\cong"] \\
		\boldsymbol{CF}^-(S^3_{\Lambda_{L_1}}(L_1),\ve{t}|_{S^3_{\Lambda_{L_1}}(L_1)}) \arrow[r, "F_{W_{\Lambda_{L_2}}(L_1,L),\ve{t}}"'] & \boldsymbol{CF}^-(S^3_{\Lambda}(L),\ve{t}|_{S^3_{\Lambda}(L)})
	\end{tikzcd}
	\]
\end{Theorem}
The original proof appears in \cite[Section~14]{manolescu2010heegaard}; here, we give an alternative proof by comparing the maps in the setting of Zemke's general surgery formula.
\begin{proof}
	The hypercube $\bbX_{\Lambda_{L_1}}(S^3,L_1)$ has all the coordinates corresponding to $L_2$ equal to 1. Therefore, for each $K_i\in L_2$, the Lagrangian associated to $K_i$ is $\beta_{1,i}$, and the local system is $(\bbF[\scrU_i,\scrV_i,\scrV_i^{-1}],\rho_1)$. Hence, after fixing a $\spin^c$ structure $\ve{t}\in\bbH(L)/ H(L_1,\Lambda_{L_1})$, $\bbX_{\Lambda_{L_1}}(S^3,L_1,\ve{t})$ is chain homotopy equivalent to $\boldsymbol{CF}^-(S^3_{\Lambda_{L_1}}(L_1),\ve{t}|_{S^3_{\Lambda_{L_1}}(L_1)})$ by the definition of these complexes and link surgery formula.
	
	Under this identification, proving the theorem reduces to comparing the 2-handle cobordism map with the composition of inclusion map and the chain homotopy equivalence \ref{eq:chainhomotopy} from $\bbX_{\Lambda}(S^3,L)$ to $\boldsymbol{CF}^-(S^3_{\Lambda}(L),\ve{t}|_{S^3_{\Lambda}(L)})$. Fix a generator $(\xs,\phi)$, where $\phi$ is an algebra morphism from the trivial local system $\bbF[U]$ to the corresponding local system of the $\beta$ curves which depends on the coordinates of the vertex. Denote the projection of $\phi$ to the local system associated to $L_2$ by $\phi_2$. Note that $\phi_2$ is determined by $\phi_2(1)=\prod_{i:K_i\in L_2}\scrU_i^{k_i}\scrV_i^{j_i}$. The chain homotopy map $\Gamma$ counts the holomorphic triangles with the input $\langle\theta_{\beta_{(1,\dots,1),\beta_\Lambda}}^+,\Pi\rangle$. The output algebra morphism of these triangles, which is given in \ref{eq:localsystem}, is
	\[U^{\sum_{i:K_i\in L_2}n_{w_i}}\circ(\Pi\circ(\prod_{i:K_i\in L_2}\scrV_i^{n_{w_i}-n_{z_i}}\circ\phi)).\]
	Hence, by the definition of the $\Pi$ map, a triangle contributes nontrivially only if
	\[k_i=j_i+n_{w_i}-n_{z_i}.\]
	This relation determines the $\spin^c$ structure associated to the holomorphic triangle in terms of the Alexander grading of the generator $(\xs,\phi)$. Under the identification \ref{eq:idspinA}, this $\spin^c$ structure agrees with the one specified by the Alexander grading. This completes the proof.
\end{proof}

Another important feature of the link surgery formula is that, one can identify the link Floer complex of sublinks with subcubes of the surgery hypercube (see \cite[Section~9]{zemke2023general}). The link Floer complex $\CFL(S^3_{\Lambda_{L_1}}(L_1),L_2)$ is a chain complex over $R_{n_2}:=\bbF[\scrU_1,\scrV_1,\dots,\scrU_{n_2},\scrV_{n_2}]$. Let $\bbX(S^3,L_1;L_2)$ be the subcube of $\bbX_{\Lambda}(S^3,L))$ for which all the coordinates corresponding to $L_2$ are equal to 0. Note that, the module $\bbX(S^3,L_1;L_2)$ has a natural action of $R_{n_2}$. We have the following:
\begin{Theorem}\cite[Corollary~9.2]{zemke2023general}\label{theorem:sublink}
	There is a chain homotopy equivalence of chain complexes over $R_{n_2}$
	\[\CFL(S^3_{\Lambda_{L_1}}(L_1),L_2)\cong \bbX(S^3,L_1;L_2).\]
\end{Theorem}

Combining Theorem \ref{theorem:cubecob} and \ref{theorem:sublink}, we can reinterpret the 2-handle link cobordism map as the inclusion of a subcube into the link surgery hypercube. Let $$(L,\Lambda)=(L_1,\Lambda_{L_1})\cup(L_2,\Lambda_{L_2})\cup(L_3,\Lambda_{L_3})$$ be an oriented framed link. Then $\bbX(S^3,L;L_{3})$ is chain homotopic to the link Floer chain complex of $(S^3_{\Lambda_{L_1}\cup\Lambda_{L_2}}(L_1\cup L_2),L_3)$. Take the subcube $\bbX_{\Lambda_{L_1}}(S^3,L_1\cup L_3;L_3)$, for which the coordinates corresponding to $L_2$ components are equal to 1 and those corresponding to $L_3$ are equal to 0. $\bbX_{\Lambda_{L_1}}(S^3,L_1\cup L_3;L_3)$ admits a decomposition of subcomplexes with respect to the grading $\psi^{L-L_3}(\bbH(L))/H(L_1,\Lambda_{L_1})$, which is also identified with the $\spin^c$ structures on the cobordism $W_{\Lambda_{L_2}}(L_1,L_1\cup L_2)$. Then we have the following reinterpretation of the 2-handle link cobordism map.

\begin{Theorem}\label{theorem:linkcobsurgery}
For any $\ve{t}\in\psi^{L-L_3}(\bbH(L))/ H(L_1,\Lambda_{L_1})$, the following diagram commutes:
\[
\begin{tikzcd}[column sep=5em, row sep=5em]
	\bbX_{\Lambda_{L_1}}(S^3,L_1\cup L_3;L_3,\ve{t}) \arrow[r] \arrow[d, "\cong"'] & \bbX_{\Lambda}(S^3,L;L_3) \arrow[d, "\cong"] \\
	\boldsymbol{CFL}(S^3_{\Lambda_{L_1}}(L_1);L_3,\ve{t}|_{S^3_{\Lambda_{L_1}}(L_1)}) \arrow[r, "F_{W_{\Lambda_{L_2}}(L_1,L_1\cup L_2),\ve{t}}"'] & \boldsymbol{CFL}(S^3_{\Lambda_{L_1}\cup\Lambda_{L_2}}(L_1\cup L_2),L_3,\ve{t}|_{S^3_{\Lambda_1\cup\Lambda_2}(L_1\cup L_2)})
\end{tikzcd}
\]	
\end{Theorem}
\begin{proof}
	The argument for the components of $L_2$ is the same as in Theorem~\ref{theorem:cubecob}. In particular, the projection $\Pi$ selects the specified $\spin^c$ structure.
	
	For each retained component $K_j\subset L_3$, the corresponding coordinate remains $0$, so its meridian carries the local system $(\beta_{0,j},\rho_{0,j})$. If $\psi$ is a triangle counted by the
	map, the local-system rule~\eqref{eq:localsystem} gives the coefficient
	\[
	U^{n_{w_j}(\psi)}\rho_{0,j}\!\left(e^{\,n_{z_j}(\psi)-n_{w_j}(\psi)}\right)=
	U^{n_{w_j}(\psi)}\scrV_j^{\,n_{z_j}(\psi)-n_{w_j}(\psi)}=
	\scrU_j^{\,n_{w_j}(\psi)}\scrV_j^{\,n_{z_j}(\psi)},
	\]
	where $U=\scrU_j\scrV_j$. Thus the coefficient on the retained link is multiplied by
	\[
	\prod_{K_j\subset L_3}\scrU_j^{\,n_{w_j}(\psi)}\scrV_j^{\,n_{z_j}(\psi)}.
	\]
	These are precisely the basepoint weights in the link cobordism map (\ref{eq:linkcobmap}). The comparison of the triangle maps therefore respects both the coefficient action and the $\spin^c$ decomposition.
\end{proof}
\section{Bordered surgery formula}\label{sec:bordered}
In \cite{zemke2021bordered}, Zemke interpreted the link surgery formula as type-$D$ modules over an associated algebra $\calK$. He also established a connected sum formula, which allows one to decompose a framed link into a connected sum of simpler pieces, and compute the link surgery formula as an $A_\infty$ tensor product of the type-$D$ modules of these pieces. Our computation of the cabling formula relies heavily on the connected sum formula, and we use this section to review both the bordered surgery formula and the connected sum formula.
\subsection{The knot and link surgery algebras}
There are a few equivalent definitions of the knot surgery algebra. In this paper, we follow the convention from \cite{zemke2021bordered}.
The \emph{knot surgery algebra} $\mathcal{K}$ is an algebra over the idempotent ring $\mathbf{I} = \mathbf{I}_0 \oplus \mathbf{I}_1$, where each of $\mathbf{I}_0$, $\mathbf{I}_1$ is rank one over $\mathbb{F} = \mathbb{Z}/2\mathbb{Z}$. Let $i_0$ and $i_1$ denote the identity elements of
$\mathbf I_0$ and $\mathbf I_1$, respectively. Then, we set
\begin{itemize}
	\item $\mathbf{I}_0 \cdot \mathcal{K} \cdot \mathbf{I}_0 = \mathbb{F}[\mathscr{U}, \mathscr{V}]$,
	\item $\mathbf{I}_1 \cdot \mathcal{K} \cdot \mathbf{I}_1 = \mathbb{F}[\mathscr{U}, \mathscr{V}, \mathscr{V}^{-1}]$,
	\item $\mathbf{I}_0 \cdot \mathcal{K} \cdot \mathbf{I}_1 = 0$, and
	\item $\mathbf{I}_1 \cdot \mathcal{K} \cdot \mathbf{I}_0 = \mathbb{F}[\mathscr{U}, \mathscr{V}, \mathscr{V}^{-1}] \otimes \langle \sigma, \tau \rangle$
\end{itemize}
for two special algebra elements $\sigma,\tau\in \mathbf{I}_1 \cdot \mathcal{K} \cdot \mathbf{I}_0$. For any $a \in \mathbb{F}[\mathscr{U}, \mathscr{V}]$, the elements $\sigma$ and $\tau$ satisfy
\[
\sigma \cdot a = \phi^\sigma(a) \cdot \sigma \qquad \text{and} \qquad \tau \cdot a = \phi^\tau(a) \cdot \tau
\]
where $\phi^\sigma, \phi^\tau: \mathbb{F}[\mathscr{U}, \mathscr{V}] \to \mathbb{F}[\mathscr{U}, \mathscr{V}, \mathscr{V}^{-1}]$ are the algebra homomorphisms with $\phi^\tau$ defined by
\[
\mathscr{U} \mapsto \mathscr{V}^{-1}\qquad \text{and} \qquad \mathscr{V} \mapsto \mathscr{U}\mathscr{V}^2
\]
and $\phi^\sigma$ is the canonical inclusion. These relations are usually rewritten
\[
\sigma \mathscr{U} = \mathscr{U} \sigma, \qquad \sigma \mathscr{V} = \mathscr{V} \sigma, \qquad \tau\mathscr{U} = \mathscr{V}^{-1} \tau, \qquad \tau \mathscr{V} = \mathscr{U}\mathscr{V}^2 \tau.
\]

We equip $\calK$ with the linear topology induced by the subspaces $J_n$ which are defined as follows.
\begin{defni}
	Suppose that $n\in\bbN$ is fixed. We define $J_n\subseteq \calK$ to be the $\bbF$-span of the following set of generators:
	\begin{enumerate}
		\item In $\bfI_0\cdot\calK\cdot\bfI_0$, the generators $\scrU^i\scrV^j$, for $i\geq n$ or $j\geq n$ (i.e. $\mathrm{max}(i,j)\geq n$).
		\item In $\bfI_1\cdot\calK\cdot\bfI_0$, the generators $\scrU^i\scrV^j\sigma$, for $i\geq n$ or $j\geq n$.
		\item In $\bfI_1\cdot\calK\cdot\bfI_0$, the generators $\scrU^i\scrV^j\tau$, for $j\leq 2i-n$ or $i\geq n$.
		\item In $\bfI_1\cdot\calK\cdot\bfI_1$, the generators $\scrU^i\scrV^j$, where $i\geq n$.
	\end{enumerate}
\end{defni} 
The \emph{link surgery algebra} is defined as
\[
\mathcal{L}_\ell := \underbrace{\mathcal{K} \otimes_{\mathbb{F}} \cdots \otimes_{\mathbb{F}} \mathcal{K}}_{\ell}
\]
which we view as an algebra over the idempotent ring $\mathbf{E}_\ell := \mathbf{I} \otimes_\mathbb{F} \cdots \otimes_\mathbb{F} \mathbf{I}$. The topology on $\calL_l$ is given by the !-topology;
\[\calL_\ell=\calK\otimes_\bbF^!\cdots\otimes_\bbF^!\calK.\]

We also need appropriate categories of type-$D$, type-$A$, and type-$DA$ modules over the surgery algebra to ensure that the map induced by taking box tensor product is continuous, hence induces well-defined map after taking completions. For a detailed treatment, see \cite[Section 7]{zemke2021bordered}. When the number of components is understood, we write $\mathcal L=\mathcal L_\ell$ and $\mathbf E=\mathbf E_\ell$.

\begin{defni}
	A \emph{type-D Alexander module} consists of a pair $(\calX,\delta^1)$ where $\calX$ is a linear topological right $\bfE$-module and
	\[\delta^1:\calX\rightarrow \calX\vec{\otimes}_{E}\calL\]
	is a linear topological morphism satisfying the type-$D$ relation in Definition \ref{def:typed}.
	Given two type-$D$ Alexander modules $\calX^\calL$ and $\calY^\calL$, let $\operatorname{Mor}(\mathcal X^\mathcal L,\mathcal Y^\mathcal L)$ denote the space of linear topological $\mathbf E$-module morphisms
	\[
	f^1:\mathcal X\longrightarrow\mathcal Y\vec\otimes_{\mathbf E}\mathcal L.
	\]
	This space is equipped with the type-$D$ morphism differential $\partial_{\mathrm{Mor}}$. A type-$D$ homomorphism is a morphism $f^1$ satisfying $\partial_{\mathrm{Mor}}(f^1)=0$.
\end{defni}

Note that, given two type-$D$ Alexander modules $\calX^{\calL_m}$ and $\calY^{\calL_n}$, their external product is again a type-$D$ Alexander module, when we use !-topology for both the underlying module and the algebra. More explicitly, the structure map $\delta_\calX^1\otimes\mathrm{id}_\calY\otimes 1_{\calL_n}+\mathrm{id}_\calX\otimes1_{\calL_m}\otimes\delta_\calY^1$ is continuous as a map:
\[\calX\otimes^!\calY\rightarrow(\calX\otimes^!\calY)\vec{\otimes}(\calL_m\otimes^!\calL_n).\] 

\begin{defni}
	A \emph{type-A Alexander module $(\calX,m_j)$} over $\calL$ consists of a linear topological left $\bfE$-module $\calX$ and a collection of linear topological morphisms
	\[m_{j+1}:\calL\vec{\otimes}_E\cdots\calL\vec{\otimes}_E\calL\vec{\otimes}_E\calX\rightarrow\calX,\]
	satisfying the $A_\infty$-relations in Definition \ref{def:Amodule}.
\end{defni}

In Subsection \ref{subsec:DAmodule}, we defined the box tensor product, which pairs a type-$D$ module with a type-$A$ module to produce a chain complex. Under the boundedness assumptions needed to define the box tensor
product, the induced differential is continuous with respect to the standard tensor topology \cite[Section 7.6]{zemke2021bordered}. We denote the resulting completed box tensor product by $\widehat{\boxtimes}$.

One can generalize the definitions to define a type-$DA$ Alexander bimodule. More generally, suppose $\bbP$ is a partition of $\{1,\ldots,l\}$ and $j_1,\ldots,j_n$ are the sizes of the partition elements, we also consider the \emph{$\bbP$-split type-DA Alexander module}, which is denoted by
\[_{\calL_{j_1}|\cdots|\calL_{j_n}}\calY^{\calL}.\]
The topology used here is a specified \emph{tree topology}, as defined in \cite[Section 5.6]{zemke2021bordered}. We omit the general definition here and instead using the fact that the bimodules under consideration are split type-$DA$ Alexander bimodules. An important result about continuity is the following. Given $n$ type-$D$ Alexander modules $\calX_1^{\calL_{i_1}},\ldots,\calX_n^{\calL_{i_n}}$ and a split type-$DA$ Alexander bimodule $_{\calL_{i_1}|\cdots|\calL_{i_n}}\calY^{\calL}$, the tensor product
\[(\calX_1^{\calL_{i_1}},\ldots,\calX_n^{\calL_{i_n}})\Boxtimes  _{\calL_{i_1}|\cdots|\calL_{i_n}}\calY^{\calL},\]
which is given by taking external product of $\calX_j^{\calL_{i_j}}$ first then taking box tensor product, is a type-$D$ Alexander module, when we topologize $(\calX_1\otimes\cdots\otimes\calX_n)\otimes\calY$ with a specified tree topology. Note that, the modules considered in this paper are equipped with cofinite basis topology, and the tree topology on their tensor product coincides with the standard tensor topology.

We list a few split type-$DA$ Alexander bimodules used in this paper. For simplicity, we follow the notation from \cite{zemke2021bordered} that we will write $|$ for $\otimes_\bbF$ and $\otimes$ for $\otimes_I$.
\begin{enumerate}
	\item \emph{Merge bimodule $_{\calK|\calK}M^\calK$}: As a vector space, $M$ is $\bfI=\bfI_0\oplus\bfI_1$. The left action of $\bfI|\bfI$ is given by \[(i_1|i_2)\cdot i=i_1i_2\cdot i,\] and the right $\bfI$ action is standard.
	The only nontrivial structure maps on $M$ are $\delta_2^1$ and $\delta_3^1$. The $\delta_2^1$ map is defined as follows. Suppose that $a_1|a_2$ is an elementary tensor in either
	\[(\bfI_0|\bfI_0)\cdot(\calK|\calK)\cdot(\bfI_0|\bfI_0) \qquad \mathrm{or} \qquad (\bfI_1|\bfI_1)\cdot(\calK|\calK)\cdot(\bfI_1|\bfI_1),\]
	then $\delta_2^1(a_1|a_2, i)=i\otimes a_1a_2$. $\delta_2^1$ vanishes on other elementary tensors.
	For $\delta_3^1$,
	\[\delta_3^1(1|\sigma,\sigma|1,i_0)=i_1\otimes \sigma,\qquad \mathrm{and} \qquad\delta_3^1(1|\tau,\tau|1,i_0)=i_1\otimes \tau.\]
	More generally, if $a,b,c,d$ are monomials concentrated in single idempotents, $\delta_3^1$ can be extended by 
	\[\delta_3^1(a|b\sigma,c\sigma|d,i_0)=i_1\otimes abcd\sigma,\qquad \mathrm{and} \qquad\delta_3^1(a|b\tau,c\tau|d,i_0)=i_1\otimes abc\tau d.\]
	$\delta_3^1$ vanishes on other sequences of elements from $\calK|\calK$.
	\item \emph{Pair-of-pants bimodule $_{\calK|\calK}W_{\alpha\beta,\alpha}^{\calK}$}: This bimodule has the same underlying module and $\delta_2^1$, $\delta_3^1$ structure maps as $_{\calK|\calK}M^\calK$, together with the additional structure map $\delta_5^1$ determined by the following formula:
	\[\delta_5^1(a|b,a'|b',1|\tau,\tau|1,i_0)=i_1\otimes \scrV^{-1}\partial_\scrU(ab)b'(\scrU\partial_\scrU+\scrV\partial_\scrV)(a')\tau.\]
	Here $a,b,a',b'\in\mathbf I_1\cdot\mathcal K\cdot\mathbf I_1$, and the derivatives are formal derivatives in $\mathbb F[\mathscr U,\mathscr V,\mathscr V^{-1}]$.
	\item \emph{Elliptic bimodule $_\calK[E]^\calK$}: This module is rank 1 over the idempotent ring $\bfI$ and has structure map given by
	\[\delta_{2}^{1}(a,1)=1\otimes E(a)\]
	where $E:\calK\rightarrow\calK$ is the algebra morphism satisfying
	\[E(\scrU)=\scrV,\qquad E(\scrV)=\scrU\]
	on the $\bfI_0$ idempotent, and
	\[E(\scrU)=\scrU\scrV^2,\qquad E(\scrV)=\scrV^{-1}\]
	on the $\bfI_1$ idempotent. $E$ also satisfies
	\[E(\sigma)=\tau, \quad \mathrm{and} \quad E(\tau)=\sigma.\]
\end{enumerate}

We conclude this subsection by recalling the definition of the solid-torus type-$A$ module and the type-$AA$ identity bimodule.

\begin{defni}[Solid-torus module]\label{def:solid-torus}
	For $\lambda\in\mathbb Z$, the solid-torus type-$A$ module ${}_{\mathcal K}\mathcal D_\lambda$ has idempotent summands
	\[\bfI_0\cdot\calD_\lambda=\bbF[\scrU,\scrV]\quad \mathrm{and}\quad \bfI_1\cdot \calD_\lambda=\bbF[\scrU,\scrV,\scrV^{-1}].\]
	The only nonzero structure map is $m_2$, defined as follows. If $f\in \bfI_i\cdot \calK\cdot \bfI_i$ and $x\in \bfI_j\cdot \calD_\lambda$, then $m_2(f,x)$ is given by the ordinary multiplication of polynomials $f\cdot x$ if $i=j$, and $0$ otherwise. If $x\in \bfI_0\cdot \calD_\lambda$,
    \[m_2(\sigma,x)=\phi^\sigma(x)\in\bbF[\scrU,\scrV,\scrV^{-1}]=\bfI_1\cdot\calD_\lambda \quad \mathrm{and}\quad m_2(\tau,x)=\scrV^\lambda\cdot\phi^{\tau}(x)\in\bbF[\scrU,\scrV,\scrV^{-1}]=\bfI_1\cdot\calD_\lambda.\] $\calD_\lambda$ is topologized with the cofinite basis topology.
\end{defni}

The type-$A$ module of the solid torus is used to recover the link surgery formula from the bordered surgery formula via the box tensor product.
\begin{defni}[Type-$AA$ identity bimodule]\label{def:identity-bimodule}
	The type-$AA$ identity bimodule is defined by pairing the merge bimodule with the solid-torus module
	${}_{\mathcal K}\mathcal D_0$:
	\[_{\calK|\calK}[\bbI^\Supset]:=_{\calK|\calK}M^\calK\Boxtimes _\calK[\calD_0].\]
\end{defni}

Pairing with the type-$AA$ identity bimodule converts a type-$D$ module over $\mathcal K$ into a type-$A$ module.
\subsection{Bordered surgery modules}\label{subsec:bordered module}
In Section \ref{subsec:surgeryformula}, we reviewed the definition of Manolescu and Ozsv\'{a}th's link surgery formula. Given an $\ell$-component oriented framed link $(L,\Lambda)$ in $S^3$, together with a choice of complete system of hyperboxes for $(L,\Lambda)$, one obtains the link surgery hypercube $\bbX_{\Lambda}(S^3,L)$. In \cite{zemke2021bordered}, Zemke rephrases this choice of auxiliary data by a \emph{system of arcs $\scrA$}, and denotes the corresponding chain complex by $\calC_\Lambda(L,\scrA)$. In this subsection, we review the construction from \cite[Section~8.6]{zemke2021bordered}, which constructs a type-$D$ module $\calX_\Lambda(L,\scrA)^{\calL_\ell}$ over the link surgery algebra $\calL_\ell$ from $\calC_\Lambda(L,\scrA)$.

Let $\xs_1,\dots,\xs_n$ be a free basis of $\CFL(L)$ over $\bbF[\scrU_1,\dots,\scrU_\ell,\scrV_1,\dots,\scrV_\ell]$, the underlying right $\bfE_\ell$-module of $\calX_\Lambda(\calL,\scrA)^{\calL_\ell}$ is
\[\mathrm{Span}_{\bbF}(\xs_1,\dots,\xs_n)\otimes_{\bbF}\bfE_\ell.\]
Thus each basis element $\xs$ gives a generator $\xs^{\varepsilon}$ for every $\varepsilon\in\bbE_{\ell}$.

The structure maps are given as follows. First, for each idempotent $\varepsilon\in\bbE_\ell$, there are internal maps $\delta_{(\varepsilon,\varepsilon)}^1$ from the differentials of $\CFL(L)$. If $\partial(\xs)$ contains a summand of $\ys\cdot f$, where $f\in \bbF[\scrU_1,\dots,\scrU_\ell,\scrV_1,\dots,\scrV_\ell]$, then $\delta_{\varepsilon,\varepsilon}^1(\xs^\varepsilon)$ contains a term of $\ys^\varepsilon\otimes f$.

Recall that, given a pair $\varepsilon<\varepsilon'$, the differential $D_{\varepsilon,\varepsilon'}$ decomposes as a sum $D_{\varepsilon,\varepsilon'}=\sum_{\vec{M}}D_{\vec{M}}.$ For each term $\ys^{\varepsilon'}\cdot f$ appears in $D_{\vec{M}}(\xs^\varepsilon)$, we associate an element $f\cdot t$ in $\bfI_{\varepsilon'}\cdot \calL_\ell\cdot \bfI_{\varepsilon}$ as follows. The algebra element $t$ is given as a tensor product, where the $i$-th factor is $\sigma_i$ if the component $K_i$ of $\vec{M}$ is oriented the same as $L$, and $\tau_i$ if $K_i$ is oriented oppositely. For a component $K_i$ not contained in $M$, the corresponding factor is the identity in $\mathbf I_{\varepsilon_i}\cdot\mathcal K\cdot\mathbf I_{\varepsilon_i}$. With this notation, the structure map $\delta_{\varepsilon,\varepsilon'}^1(\xs^\varepsilon)$ is given by the sum of these summands
\[\ys^{\varepsilon'}\otimes f\cdot t.\] In \cite[Section~8.6]{zemke2021bordered}, these structure maps are shown to satisfy the type-$D$ relations. Hence, $\calX_\Lambda(L,\scrA)^{\calL_\ell}$ is a type-$D$ module over the link surgery algebra $\calL_\ell$.

Choose a component $K\in L$, and view the bordered surgery formula as $\calX_\Lambda(L,\scrA)^{\calL_{\ell-1}\otimes \calK}$. The component corresponding to $K$ can be converted to type-$A$ by box tensoring with the type-$AA$ identity module
\[_\calK \calX_\Lambda(L,\scrA)^{\calL_{\ell-1}}:=\calX_\Lambda(L,\scrA)^{\calL_{\ell-1}\otimes \calK}\Boxtimes _{\calK|\calK}[\bbI^\Supset].\]
One can keep tensoring the type-$AA$ identity modules to get other type-$DA$ bimodules $_{\calL_{\ell_1}} \calX_\Lambda(L,\scrA)^{\calL_{\ell_2}}$.

More importantly, one can recover the link surgery formula by filling in each knot component with the type-$A$ module of solid torus. Given a collection of integers $\Lambda=(\lambda_1,\ldots,\lambda_\ell)$, there is a type-$A$ module
\[_{\calL_\ell}\calD_\Lambda\]
defined by taking the external product of the modules $_\calK[\calD_{\lambda_i}]$. Then, by the definition of $\calX_\Lambda(L,\scrA)^{\calL_\ell}$:
\begin{prop}\cite[Proposition~8.11]{zemke2021bordered}\label{prop:fillingA}
	\[\calC_\Lambda(L,\scrA)\cong \calX_\Lambda(L,\scrA)^{\calL_\ell}\Boxtimes_{L_\ell}\calD_{\{0,\ldots,0\}}.\]
\end{prop}
\subsection{Connected sum formula}\label{subsec:connectedsumformula}
An important feature of the bordered surgery formula which allows easier computation is the existence of the connected sum formula.
\begin{Theorem}[\cite{zemke2021bordered}, Theorem 12.1]\label{theorem:connect}
	Suppose that $L_1$ and $L_2$ are two links in $S^3$, which are equipped with system of arcs $\scrA_1$ and $\scrA_2$, respectively. Form their connected sum $L_1\# L_2$ along two distinguished components $K_1\subseteq L_1$ and $K_2\subseteq L_2$. Assume the arc for $K_1$ is alpha-parallel and that the arc for $K_2$ is beta-parallel. Then
	\[\calX_{\Lambda_1+\Lambda_2}(L_1\# L_2,\scrA_1\#\scrA_2)^{\calL_{\ell_1+\ell_2-1}}\cong (\calX_{\Lambda_1}(L_1,\scrA_1)^{\calL_{\ell_1}}\otimes_\bbF\calX_{\Lambda_2}(L_2,\scrA_2)^{\calL_{\ell_2}})\Boxtimes_{\calK|\calK}M^\calK.\]
\end{Theorem}

By Proposition \ref{prop:fillingA}, we can fill in the type-$A$ module to recover the link surgery formula. Apply it to the connected sum formula, we get the following:
\begin{align*}
	\calC_{\Lambda_1+\Lambda_2}(L_1\#L_2,\scrA_1\#\scrA_2)&\cong\calX_{\Lambda_1+\Lambda_2}(L_1\# L_2,\scrA_1\#\scrA_2)^{\calL_{\ell_1+\ell_2-1}}\Boxtimes _{\calL_{\ell_1+\ell_2-1}}\calD_{\{0,\ldots,0\}},\\ &\cong (\calX_{\Lambda_1}(L_1,\scrA_1)^{\calL_{\ell_1}}\otimes_\bbF\calX_{\Lambda_2}(L_2,\scrA_2)^{\calL_{\ell_2}})\Boxtimes_{\calK|\calK}M^\calK \Boxtimes _{\calL_{\ell_1+\ell_2-1}}\{_{\calK}\calD_0\otimes _{\calL_{\ell_1+\ell_2-2}}\calD_{\{0,\ldots,0\}}\},\\
	&\cong (\calX_{\Lambda_1}(L_1,\scrA_1)^{\calL_{\ell_1}}\otimes_\bbF\calX_{\Lambda_2}(L_2,\scrA_2)^{\calL_{\ell_2}})\Boxtimes _{\calK\otimes\calK}[\bbI^\Supset]\Boxtimes _{\calL_{\ell_1+\ell_2-2}}\calD_{\{0,\ldots,0\}},\\
	&\cong \calX_{\Lambda_{K_1}}(S^3_{\Lambda_{L_1\setminus K_1}}(L_1\setminus K_1),K_1,\scrA_{K_1})^{\calK}\Boxtimes _\calK \calX_{\Lambda_{K_2}}(S^3_{\Lambda_{L_2\setminus K_2}}(L_2\setminus K_2),K_2,\scrA_{K_2}).
\end{align*}
Note that the arc on the component $K_1\#K_2$ is the co-core of the connected sum band, and is neither alpha-parallel nor beta-parallel. Therefore, to perform connected sum formula for more than two components, we need to use the pair-of-pants bimodule instead. Let $\scrA_\alpha$ be the system of arcs for $L_1\#L_2$ which uses an alpha-parallel arc for $K_1\#K_2$. Then
\begin{Theorem}[\cite{zemke2021bordered}, Theorem 15.1]\label{Thm:pop}
	    Under the hypotheses of Theorem~\ref{theorem:connect}, there is a homotopy equivalence
		\[\calX_{\Lambda_1+\Lambda_2}(L_1\# L_2,\scrA_\alpha)^{\calL_{\ell_1+\ell_2-1}}\cong (\calX_{\Lambda_1}(L_1,\scrA_1)^{\calL_{\ell_1}}\otimes_\bbF\calX_{\Lambda_2}(L_1,\scrA_2)^{\calL_{\ell_2}})\Boxtimes_{\calK|\calK}W_{\alpha\beta,\alpha}^{\calK}.\]
\end{Theorem}

\subsection{Type-$DA$ bimodule of the negative Hopf link Floer complex}\label{subsec:Hopflink}
In this subsection, we review the type-$DA$ bimodule of the $(\lambda_1,0)$-framed negative Hopf link. For the computation, see \cite[Section 17]{zemke2021bordered}. As shown in Figure \ref{fig:negHopf}, the bimodule under consideration has type-$A$ structure on the component $L_1$, and type-$D$ structure on $L_2$. Moreover, the arc system is chosen so that $L_1$ is alpha-parallel, and $L_2$ is beta-parallel.

\begin{figure}[h]
	\centering
	
	\begin{tikzpicture}[scale=1]
		
		\draw[postaction={decorate},
		decoration={
			markings,
			mark=at position 0.67 with {\arrow{>}}
		}] (0,0) ++(-45:1) arc (-45:295:1);
		\draw [postaction={decorate},
		decoration={
			markings,
			mark=at position 0.15 with {\arrow{<}}
		}] (1,0) ++(130:1) arc (130:475:1);
		
		\node at (-0.1,1.5) {$(L_1,\lambda_1)$};
		\node at (1.3,1.5) {$(L_2,0)$};
		
		\node at (0,-1.5) {$(A,\alpha)$};
		\node at (1.2,-1.5) {$(D,\beta)$};
		
	\end{tikzpicture}
	\caption{Negative Hopf link}
	\label{fig:negHopf}
\end{figure}
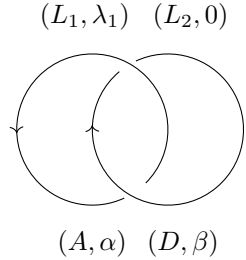
The type-$DA$ bimodule $\abZ_{(\lambda_1,0)}$ is defined over the idempotent ring $\bfE_2=\bfI\otimes_{\bbF}\bfI$. The underlying $(\mathbf I,\mathbf I)$-module is independent of the framing, and we denote the summand $\bfI_i\cdot\abZ_{(\lambda_1,0)}\cdot\bfI_j$ by $\overline{\calZ}_{i,j}$, suppressing the framing parameters from the notation.
The framing dependence is recorded in the type-$DA$ structure maps described below. The module is generated by two elements $\ve{a}$ and $\ve{d}$, such that the Alexander gradings $A=(A_{1},A_{2})$ and Maslov gradings $\gr=(\mathrm{gr}_{\ws},\mathrm{gr}_{\zs})$ are
\begin{align*}
	A(\ve{a})=(0,1)+(-\frac{1}{2},-\frac{1}{2}), \ \ & \gr(\ve{a})=(0,0),\\
	A(\ve{d})=(1,0)+(-\frac{1}{2},-\frac{1}{2}), \ \ & \gr(\ve{d})=(0,0).
\end{align*}

For each summand, we have
\begin{enumerate}
	\item $\overline{\calZ}_{0,0}$ is generated by $\langle\ve{a}\rangle[\scrU_{1}]\oplus\langle\ve{d}\rangle[\scrV_1]$,
	\item $\overline{\calZ}_{1,0}$ is generated by $\langle\ve{d}\rangle[\scrV_1,\scrV^{-1}_1]$,
	\item $\overline{\calZ}_{0,1}$ is generated by $\langle\ve{a}\rangle[\scrU_{1}]\oplus\langle\ve{d}\rangle[\scrV_1]$,
	\item $\overline{\calZ}_{1,1}$ is generated by $\langle\ve{d}\rangle[\scrV_1,\scrV^{-1}_1]$,
\end{enumerate}
where $\scrU_1$ and $\scrV_1$ are the corresponding elements in the knot surgery algebra $\calK_1$. $\langle\ve{a}\rangle[\scrU_1]$ denotes a copy of $\bbF[\scrU_1]$ generated by $\ve{a}$, and similarly for the others. We view them as direct sums of copies of $\bbF$ and equip them with cofinite basis topology defined in Subsection \ref{subsec:lineartopology}.
The structure maps decompose schematically as follows.
\[
\delta_1^1=\begin{tikzcd}[labels=description, row sep=1.1cm,column sep=.8cm]
	\overline{\calZ}_{0,0} \ar[d,"{}^{L_2}\! f_1^1+{}^{-L_2}\! f_1^1"] & \overline{\calZ}_{1,0}\ar[d, "{}^{L_2}\! g_1^1+{}^{-L_2}\! g_1^1"]\\
	\overline{\calZ}_{0,1}  &\overline{\calZ}_{1,1} 
\end{tikzcd}\qquad \qquad 
\delta_2^1=
\begin{tikzcd}[labels=description, row sep=1.1cm,column sep=2.3cm]
	\overline{\calZ}_{0,0} \ar[loop left, "m_2^1"] \ar[r,"{}^{L_1}\! p_2^1+{}^{-L_1}\! p_2^1"] & \overline{\calZ}_{1,0} \ar[loop right, "m_2^1"]
	\\
	\overline{\calZ}_{0,1} \ar[r, "{}^{L_1}\! q_2^1+{}^{-L_1}\! q_2^1"]  \ar[loop left, "m_2^1"] &\overline{\calZ}_{1,1} \ar[loop right, "m_2^1"] 
\end{tikzcd} 
\]

All higher structure maps vanish: $\delta_j^1=0$ for $j\geq3$.
\begin{prop}[\cite{zemke2021bordered}]\label{prop:hopf}
For the $(\lambda_1,0)$ framed negative Hopf link. The structure maps of the type-$DA$ bimodule $\abZ_{(\lambda_1,0)}$ are as follows:
\begin{enumerate}
	\item For the summand $\overline{\calZ}_{0,0}$ and $\overline{\calZ}_{0,1}$, the  $m_2^1$ is given by
          \begin{enumerate}
	\item $m_2^1(\scrU_1,\scrU_1^n\ve{a})=\scrU_1^{n+1}\ve{a}\otimes 1$.
	\item $m_2^1(\scrV_1,\scrU_1^n\ve{a})=
	      \begin{cases}
		  \scrU_1^{n-1}\ve{a}\otimes \scrU_2\scrV_2& \text{ if } n>0\\
		  \ve{d}\otimes \scrV_2& \text{ if } n=0.
		  \end{cases}$
	\item $m_2^1(\scrV_1, \scrV_1^m \ve{d})=\scrV_1^{m+1} \ve{d}\otimes 1$.
	\item $m_2^1(\scrU_1, \scrV_1^m \ve{d})=
	      \begin{cases}
	      \scrV_1^{m-1}\ve{d}\otimes \scrU_2 \scrV_2 &\text{ if } m>0\\
		  \ve{a}\otimes \scrU_2& \text{ if } m=0.
		  \end{cases}$
	      \end{enumerate}
          For the summand $\overline{\calZ}_{1,0}$ and $\overline{\calZ}_{1,1}$, the  $m_2^1$ is given by 
          $$m_2^1(\scrU_1^i \scrV_1^j, \scrV_1^n\ve{d})=\scrV_1^{j+n-i}\ve{d}\otimes \scrU_2^{i}\scrV_2^{i}.$$
	\item The maps ${}^{L_1}\! p_2^1$ and ${}^{L_1}\! q_2^1$ are given by the same formulas as each other, as are ${}^{-L_1}\! p_2^1$ and ${}^{-L_1}\! q_2^1$. They are determined by the following formulas:
			\begin{enumerate}
				\item ${}^{L_1}\!p_2^1(\sigma_1,\scrU_1^i \ve{a})=\scrV_1^{-i-1} \ve{d} \otimes \scrU_2^i \scrV_2^{i+1} $ and ${}^{L_1}\! p_2^1(\sigma_1,\scrV_1^j\ve{d})=\scrV_1^j \ve{d}\otimes 1$.
				\item ${}^{-L_1}\! p_2^1(\tau_1,\scrU_1^i\ve{a})=\scrV_1^{-i-1+\lambda_1} \ve{d}\otimes 1$ and ${}^{-L_1}\! p_2^1(\tau_1,\scrV_1^j \ve{d})=\scrV_1^{j+\lambda_1} \ve{d}\otimes \scrU_2^{j+1} \scrV_2^j.$
				\item ${}^{L_1} \!p_2^1(\tau_1,-)=0$ and ${}^{-L_1} \! p_2^1(\sigma_1,-)=0$.
			\end{enumerate}
			\item The maps for $\pm L_2$ are as follows:
			\begin{enumerate}
				\item ${}^{L_2}\! f_1^1(\scrU_1^i\ve{a})=\scrU_1^i \ve{a}\otimes \sigma_2$ and ${}^{L_2}\! f_1^1(\scrV_1^j \ve{d})=\scrV_1^j \ve{d}\otimes \sigma_2$.
				\item ${}^{-L_2} \!f_1^1(\scrU_1^i \ve{a})=\scrU_1^{i+1} \ve{a} \otimes \tau_2$ and
				\[
				{}^{-L_2}\! f_1^1(\scrV_1^j\ve{d})=
				\begin{cases}
					\ve{a}\otimes \scrV_2^{-1}\tau_2 &\text{ if } j=0\\
					\scrV_1^{j-1} \ve{d}\otimes \tau_2& \text{ if } j>0.
				\end{cases}
				\]
				\item  ${}^{L_2}\!g_1^1(\scrV_1^i \ve{d})=\scrV_1^i \ve{d}\otimes \sigma_2$.
				\item ${}^{-L_2}\!g_1^1(\scrV_1^i \ve{d})=\scrV_1^{i-1}\ve{d}\otimes \tau_2.$
			\end{enumerate}
		\end{enumerate}
	\end{prop}
	
The summand $\abZ_{(0,0)}\cdot \ve{I}_{0}$ is described in Figure \ref{fig:D0A}, where an arrow decorated with $a|b$ from $\xs$ to $\ys$ means that $\delta_2^1(a,\xs)=\ys\otimes b$ and we set $U_2=\scrU_2\scrV_2$. Note that $\abZ_{(0,0)}\cdot \ve{I}_{1}$ is the localization of $\abZ_{(0,0)}\cdot \ve{I}_{0}$ at $\scrV_{2}$. 
\begin{figure}[h]
\centering
\begin{tikzcd}
	& {\scrU_{1}^{2}\mathbf{a}} && {\scrU_{1}\mathbf{a}} && {\mathbf{a}} && {\mathbf{d}} && {\scrV_{1}\mathbf{d}} \\
	\cdots &&&&&& {} &&&& \cdots \\
	& {\scrV_{1}^{-3}\mathbf{d}} && {\scrV_{1}^{-2}\mathbf{d}} && {\scrV_{1}^{-1}\mathbf{d}} && {\mathbf{d}} && {\scrV_{1}\mathbf{d}}
	\arrow["{\scrV_{1}|U_{2}}"{description}, bend left=10, from=1-2, to=1-4]
	\arrow["{{\tau_{1}|1}}"{description}, bend right=30, from=1-2, to=3-2]
	\arrow["{\sigma_{1}|U_{2}^{2}\scrV_{2}}"{description}, bend left=30, from=1-2, to=3-2]
	\arrow["{\scrU_{1}|1}"{description}, shift left=3, bend left=10, from=1-4, to=1-2]
	\arrow["{\scrV_{1}|U_{2}}"{description}, shift right, bend left=10, from=1-4, to=1-6]
	\arrow["{{\tau_{1}|1}}"{description}, bend right=30, from=1-4, to=3-4]
	\arrow["{\sigma_{1}|U_{2}\scrV_{2}}"{description}, bend left=30, from=1-4, to=3-4]
	\arrow["{\scrU_{1}|1}"{description}, shift left=3, bend left=10, from=1-6, to=1-4]
	\arrow["{\scrV_{1}|\scrV_{2}}"{description}, bend left=10, from=1-6, to=1-8]
	\arrow["{{\tau_{1}|1}}"{description}, bend right=30, from=1-6, to=3-6]
	\arrow["{\sigma_{1}|\scrV_{2}}"{description}, bend left=30, from=1-6, to=3-6]
	\arrow["{\scrU_{1}|\scrU_{2}}"{description}, shift left=3, bend left=10, from=1-8, to=1-6]
	\arrow["{\scrV_{1}|1}"{description}, bend left=10, from=1-8, to=1-10]
	\arrow["{{\tau_{1}|\scrU_{2}}}"{description}, bend right=30, from=1-8, to=3-8]
	\arrow["{\sigma_{1}|1}"{description}, bend left=30, from=1-8, to=3-8]
	\arrow["{\scrU_{1}|U_{2}}"{description}, shift left=3, bend left=10, from=1-10, to=1-8]
	\arrow["{{\tau_{1}|\scrU_{2}U_{2}}}"{description}, bend right=30, from=1-10, to=3-10]
	\arrow["{\sigma_{1}|1}"{description}, bend left=30, from=1-10, to=3-10]
	\arrow["{\scrV_{1}|1}"{description}, bend left=10, from=3-2, to=3-4]
	\arrow["{\scrU_{1}|U_{2}}"{description}, shift left=3, bend left=10, from=3-4, to=3-2]
	\arrow["{\scrV_{1}|1}"{description}, bend left=10, from=3-4, to=3-6]
	\arrow["{\scrU_{1}|U_{2}}"{description}, shift left=3, bend left=10, from=3-6, to=3-4]
	\arrow["{\scrV_{1}|1}"{description}, bend left=10, from=3-6, to=3-8]
	\arrow["{\scrU_{1}|U_{2}}"{description}, shift left=3, bend left=10, from=3-8, to=3-6]
	\arrow["{\scrV_{1}|1}"{description}, bend left=10, from=3-8, to=3-10]
	\arrow["{\scrU_{1}|U_{2}}"{description}, shift left=3, bend left=10, from=3-10, to=3-8]
\end{tikzcd}
\caption{$\abZ_{(0,0)}\cdot \ve{I}_{0}$}
\label{fig:D0A}
\end{figure}
	
The summand $\bfI_0\cdot\abZ_{(0,0)}$ serves as the building block for the computation of the cabling module. Note that after restricting the type-A structure to the $\bfI_0$ summand, the only nontrivial type-A structure map is given by the $\bbF[\scrU_1,\scrV_1]$ coefficient action. Hence, we can view it as a $(\bbF[\scrU_1,\scrV_1],\calK_2)$-bimodule. For simplicity, we denote this bimodule by $\calH_1^\calK$ and suppress the left coefficient ring from the notation, which is depicted in Figure~\ref{fig:DA0}.
	
\begin{figure}[h]
	\begin{tikzcd}
			& {\scrU_{1}^{2}\bfa} && {\scrU_{1}\bfa} && \bfa && \bfd && {\scrV_{1}\bfd} \\
			\cdots &&&&&&&&&& \cdots \\
			& {\scrU_{1}^{2}\bfa} && {\scrU_{1}\bfa} && \bfa && \bfd && {\scrV_{1}\bfd}
			\arrow["{\sigma_{2}}"{description}, from=1-2, to=3-2]
			\arrow["{\tau_{2}}"{description}, from=1-4, to=3-2]
			\arrow["{\sigma_{2}}"{description}, from=1-4, to=3-4]
			\arrow["{\tau_{2}}"{description}, from=1-6, to=3-4]
			\arrow["{\sigma_{2}}"{description}, from=1-6, to=3-6]
			\arrow["{\scrV_{2}^{-1}\tau_{2}}"{description}, from=1-8, to=3-6]
			\arrow["{\sigma_{2}}"{description}, from=1-8, to=3-8]
			\arrow["{\tau_{2}}"{description}, from=1-10, to=3-8]
			\arrow["{\sigma_{2}}"{description}, from=1-10, to=3-10]
			\arrow["{\scrV_{1}|U_{2}}"{description}, bend left=10, from=1-2, to=1-4]
			\arrow["{\scrU_{1}|1}"{description}, shift left=3, bend left=10, from=1-4, to=1-2]
			\arrow["{\scrV_{1}|U_{2}}"{description}, shift right, bend left=10, from=1-4, to=1-6]
			\arrow["{\scrU_{1}|1}"{description}, shift left=3, bend left=10, from=1-6, to=1-4]
			\arrow["{\scrV_{1}|\scrV_{2}}"{description}, bend left=10, from=1-6, to=1-8]
			\arrow["{\scrU_{1}|\scrU_{2}}"{description}, shift left=3, bend left=10, from=1-8, to=1-6]
			\arrow["{\scrV_{1}|1}"{description}, bend left=10, from=1-8, to=1-10]
			\arrow["{\scrU_{1}|U_{2}}"{description}, shift left=3, bend left=10, from=1-10, to=1-8]	
			\arrow["{\scrV_{1}|U_{2}}"{description}, bend left=10, from=3-2, to=3-4]
			\arrow["{\scrU_{1}|1}"{description}, shift left=3, bend left=10, from=3-4, to=3-2]
			\arrow["{\scrV_{1}|U_{2}}"{description}, shift right, bend left=10, from=3-4, to=3-6]
			\arrow["{\scrU_{1}|1}"{description}, shift left=3, bend left=10, from=3-6, to=3-4]
			\arrow["{\scrV_{1}|\scrV_{2}}"{description}, bend left=10, from=3-6, to=3-8]
			\arrow["{\scrU_{1}|\scrU_{2}}"{description}, shift left=3, bend left=10, from=3-8, to=3-6]
			\arrow["{\scrV_{1}|1}"{description}, bend left=10, from=3-8, to=3-10]
			\arrow["{\scrU_{1}|U_{2}}"{description}, shift left=3, bend left=10, from=3-10, to=3-8]	
		\end{tikzcd}
		\caption{$\calH_1^\calK:=\bfI_0\cdot\abZ_{(0,0)}$}
		\label{fig:DA0}
\end{figure}
	
When applying the connected sum formula stated in Subsection \ref{subsec:connectedsumformula}, the arc systems on the two components must be compatible. Hence, we also consider the type-$DA$ bimodule when the arc system is alpha-parallel for both components of the negative Hopf link. The corresponding bimodule is denoted by $\Z$. The only difference between $\Z$ and $\abZ$ is that $\Z$ has extra nontrivial $\delta_3^1$ structure maps, see \cite[Section 17]{zemke2021bordered}. However, the $\delta_3^1$ terms only occur on the diagonal, namely the maps from $\calZ_{0,0}$ to $\calZ_{1,1}$. Hence, when we restrict it to the summand $\bfI_0\cdot\Z$, these terms disappear and the bimodule agrees with $\calH_1^\calK$.

\section{Bordered module for the $n$-strand cabling pattern}\label{sec:cabling}
Figure~\ref{fig:plumbingCn} depicts a plumbing diagram for the exterior of the $(n,0)$-pattern link $L_{(n,0)}$ in a solid torus. This exterior is obtained by removing disjoint tubular neighborhoods of $n$ parallel core curves. Gluing the boundary of the solid torus to the knot complement of any framed knot $(K,\lambda)$ yields the complement of the $(n,0)$-cabling of $(K,\lambda)$.

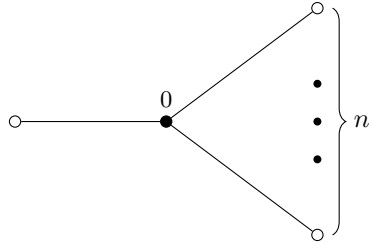
\begin{figure}[h]
	
	\centering
	\begin{tikzpicture}[scale=1, every node/.style={circle, draw, inner sep=1.5pt}]
		
		\node[fill=black, label=above:{\small $0$}] (v0) at (0,0) {};
		
		\node (vL) at (-2,0) {};
		\draw (v0) -- (vL);
		
		\node (vU) at (2,1.5) {};
		\draw (v0) -- (vU);
		
		\node (vD) at (2,-1.5) {};
		\draw (v0) -- (vD);
		
		\foreach \y in {-0.5,0,0.5} {
			\fill (2,\y) circle (1.5pt);
		}
		
		\draw [decorate,decoration={brace,amplitude=5pt}]
		(2.2,1.5) -- (2.2,-1.5)
		node[midway, right=5pt, draw=none] {$n$};
		
	\end{tikzpicture}
	\caption{A plumbing diagram for the exterior of $L_{(n,0)}$ in the solid torus.}
	\label{fig:plumbingCn}
	
\end{figure}
We associate a type-$DA$ bimodule to it, which is denoted by $_{\calL_n}\calP_{(n,0)}^{\calK}$. Here the type-$D$ side corresponds to the outer boundary of the solid torus, and the type-$A$ side corresponds to boundary tori surrounding the components of the pattern. Applying the connected sum formula, we can compute the link surgery module of $K_{(n,\lambda n)}$ by
\[_{\calL_n}K_{(n,\lambda n)}\cong _{\calL_n}\calP_{(n,0)}^\calK\Boxtimes_{\calK}\calX_{\lambda}(K).\]
Here ${}_{\mathcal K}\mathcal X_\lambda(K)$ is the type-$A$ version of the bordered invariant of the $\lambda$-framed knot complement, obtained by the conversion described in Subsection~\ref{subsec:bordered module}. $_{\calL_n}K_{(n,\lambda n)}$ denotes the corresponding type-$A$ module of the cabling link $K_{(n,\lambda n)}$. Since we do not need the full link surgery module, but only the link Floer complex of $K_{(n,\lambda n)}$. By Theorem \ref{theorem:sublink}, it is chain homotopic to the complex obtained by restricting $_{\calL_n}K_{(n,\lambda n)}$ to the $\otimes_{n}\bfI_{0}$ idempotent. More precisely, define 
\begin{center}
	$R_{n}:= \bbF[\scrU_{1},\scrV_{1},\cdots,\scrU_{n},\scrV_{n}]$,
\end{center}
which is the coefficient ring of the link Floer complex $\CFL(K_{(n,\lambda n)})$. We can view $\CFL(K_{(n,\lambda n)})$ as a left type-$A$ module over $R_{n}$, where the $R_{n}$ action is given by the coefficient action. From the discussion above, we have
\[	_{R_{n}}\CFL(K_{(n,\lambda n)})\cong \otimes_{n}\bfI_{0}\cdot_{\calL_n}K_{(n,\lambda n)}\cong\otimes_{n}\bfI_{0}\cdot_{\calL_n}\calP_{(n,0)}^\calK\Boxtimes_{\calK}\calX_{\lambda}(K).\]
We suppress the retained left $R_n$-action from the notation and denote $\otimes_{n}\bfI_{0}\cdot_{\calL_n}\calP_{(n,0)}^\calK$ by $C_{(n,0)}^\calK$. Then
\begin{equation}
	_{R_{n}}\CFL(K_{(n,\lambda n)})\cong C_{(n,0)}^\calK\Boxtimes_{\calK}\calX_{\lambda}(K).
\end{equation}
In the rest part of this section, we compute the type-$D$ module $C_{(n,0)}^{\calK}$.

\subsection{Connected sum decomposition}
To compute $C_{(n,0)}^{\calK}$, we decompose the plumbing diagram as connected sums of Hopf links, then apply the connected sum formula. We begin by forming the connected sum of two copies of Hopf links as in Figure \ref{fig:connectsum1}. The arcs on the components being joined must be
alpha-parallel and beta-parallel, respectively. Hence, we need to use one $\alpha\beta$ labeled Hopf link bimodule, and the other one $\alpha\alpha$ labeled. Since we will restrict to $\bfI_0$ summand, by the discussion in Subsection \ref{subsec:Hopflink}, both of them are equal to $\calH_1$. To iterate this construction, we require an alpha-parallel arc on the newly formed component. We therefore apply pair-of-pants bimodule ${}_{\mathcal K|\mathcal K}W_{\alpha\beta,\alpha}^{\mathcal K}$ in the connected sum formula \ref{Thm:pop}.  However, since the only difference between the pair of pants bimodule and the merge bimodule is the additional nontrivial $\delta_{5}^{1}$ map:
\[\delta_{5}^{1}(a|b,a'|b',1|\tau,\tau|1,i_{0})=i_{1}\otimes v^{-1}\partial_{\scrU}(ab)b'(\scrU\partial_{\scrU}+\scrV\partial_{\scrV})(a')\tau.\]
In the above, $a,b,a^{'},b^{'}\in \bfI_{1}\cdot\calK\cdot\bfI_{1}$. Hence, we need nontrivial internal differentials on the $\bfI_{1}$ idempotent as input for the $\delta_{5}^{1}$ map. Since there are no internal differentials on $\calH_1^\calK$, the $\delta_{5}^{1}$ map will not have extra output, and we can use the merge bimodule in the connected sum formula. We denote the resulting bimodule by $\calH_2^\calK$:
\[\calH_2^\calK:=(\calH_1^\calK\otimes_\bbF\calH_1^\calK)\Boxtimes _{\calK|\calK}M^\calK.\]

\begin{figure}[h]
	
	\centering
	\begin{tikzpicture}[scale=1, every node/.style={circle, draw, inner sep=1.5pt}]
		
		\node[fill=black, label=above:{\small $\alpha$}] (v1) at (-0.5,0.5) {};
		\node[fill=black, label=below:{\small $\beta$}] (v2) at (-0.5,-0.5) {};
        \node (v3) at (1,1) {};
        \node (v4) at (1,-1) {};
		\draw (v1) -- (v3);
		\draw (v2) -- (v4);
		\node[ draw=none] (v5) at (1.5,0) {$\cong$};
		\node[ draw=none] (v6) at (-0.5,0) {$\#$};
		\node[fill=black, label=above:{\small $\alpha$}] (v7) at (2,0) {};
		\node (v8) at (3.5,1) {};
		\node (v9) at (3.5,-1) {};
		\draw (v7) -- (v9);
		\draw (v7) -- (v8);
	\end{tikzpicture}
	\caption{A connected sum of two Hopf links with an alpha-parallel arc on the resulting component.}
	\label{fig:connectsum1}
	
\end{figure}
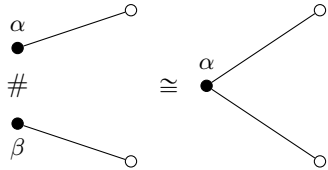
The resulting type-$D$ module also has no internal differentials. Hence, we can iteratively connected sum it with the $\alpha\beta$ labeled Hopf link, and define the bimodule $$\calH_n^\calK:=(\calH_1^\calK\otimes_\bbF\calH_{n-1}^\calK)\Boxtimes _{\calK|\calK}M^\calK.$$
For the $s$-th Hopf link copy, denote the left coefficient variables by $\mathscr U_{1,s},\mathscr V_{1,s}$ and the variables in the right surgery algebra by $\mathscr U_{2,s},\mathscr V_{2,s}$. Each of the two idempotent summands of $\calH_n^\calK$ has a basis indexed by tuples $(f_1,\ldots,f_n)$, where $f_{s}\in\{(\scrU_{1,s})^{i}\bfa_{s}$ or $(\scrV_{1,s})^{j}\bfd_{s}\}$, for $s\in\{1,2,\dots,n\}$ and $i,j\in \mathbb{N}_{0}$. We use the same tuple notation in both summands, suppressing the idempotent index. Suppose the structure map on the $s$-th type-$D$ module $\calH_1^\calK$ is given by
\begin{center}
	$\delta_{s}(f_{s})=f_{s,\sigma}\otimes\sigma_{s}+f_{s,\tau}\otimes c_{s}\tau_{s}$,
\end{center}
here $c_{s}\in \bfI_{1}\cdot\calK\cdot\bfI_{1}$, then by the definition of the merge bimodule, the structure map on $\calH_n^\calK$ is given by
\begin{equation}
	\delta(f_{1},f_{2},\dots,f_{n})=(f_{1,\sigma},\dots,f_{n,\sigma})\otimes\sigma+(f_{1,\tau},\dots,f_{n,\tau})\otimes (c_{1}\dots c_{n})\tau.
\end{equation}

Note that, $c_{i}$ is nontrivial only when $f_{i}=\bfd_{i}$, and we have $c_{i}=(\scrV_{2,i})^{-1}$ in this case. After taking connected sum, $\scrU_{2,i}$ (respectively $\scrV_{2,i}$) are identified with each other, which we simply denote by $\scrU_2$ (respectively $\scrV_2$). When forgetting the left coefficient action, the underlying module admits a split into summand of submodules, which are the zigzags connected by the structure maps $\sigma$ and $\tau$.

Finally, we attach the remaining $(0,0)$-framed Hopf link piece in Figure~\ref{fig:plumbingCn}, using the $\alpha\beta$ bimodule from Subsection~\ref{subsec:Hopflink}. The connected sum formula then gives a model
\begin{equation}\label{eq:connectedsum}
	C_{(n,0)}^{\mathcal K}\simeq\mathcal H_n^{\mathcal K}\widehat{\boxtimes}{}_{\mathcal K}\overline{\mathcal Z}_{(0,0)}^{\mathcal K}.
\end{equation}

\begin{Remark}\label{remark:orientation-}
	Since the bimodules used in this connected sum decomposition correspond to the negative Hopf link, the resulting cabling reverses the orientation from the introduction; that is, the pattern links have opposite orientations. We account for this orientation reversal when comparing the resulting Floer modules.
\end{Remark}

\subsection{Zigzag decomposition and homological reduction}\label{subsec:Dn0}
The connected sum decomposition in Equation~\ref{eq:connectedsum} gives an explicit, infinitely generated model for $C_{(n,0)}^{\mathcal K}$. In this subsection, we apply homological perturbation lemma to simplify $C_{(n,0)}^\calK$. The computation closely follows \cite[Section~18]{zemke2021bordered} \cite[Section~7]{chen2024space}. 
We first study the idempotent $C_{(n,0)}^\calK\cdot \bfI_0$. By Equation \ref{eq:connectedsum},
\[C_{(n,0)}^\calK\cdot \bfI_0\simeq \calH_n^\calK\Boxtimes\abZ_{(0,0)}\cdot\bfI_{0}.\]
Recall that the type-$A$ module $\abZ_{(0,0)}\cdot \ve{I}_{0}$ is as shown in Figure \ref{fig:D0A}. Each generator corresponds to a copy of $\bbF_{2}$, and each arrow correspond to $\delta_{2}^{1}$ map. To avoid confusion, we use $\tilde{\bfa}, \tilde{\bfd}$ for the generators of $\abZ_{(0,0)}\cdot \ve{I}_{0}$, and $\tilde{\scrU}_{1},\tilde{\scrV}_{1},\tilde{\scrU}_{2},\tilde{\scrV}_{2}$ for the coefficients. The generators of $C_{(n,0)}^\calK\cdot \bfI_0$ take the form of
\begin{center}
	$\bfe|(f_{1},f_{2},\dots,f_{n})$,
\end{center}
where $\bfe\in\{\tilde{\scrU}_{1}^{i}\tilde{\bfa}$ or $\tilde{\scrV}_{1}^{j}\tilde{\bfd}\}$ and $(f_{1},f_{2},\dots,f_{n})$ is a generator of $\calH_{n}^{\calK}$. The box tensor identifies the coefficient ring $\bbF_{2}[\scrU_{2},\scrV_{2}]$ with $\bbF_{2}[\tilde{\scrU}_{1},\tilde{\scrV}_{1}]$. We record this identification using a ring morphism $\phi$. The differentials on $C_{(n,0)}^\calK\cdot \bfI_0$ are given by:
\begin{align}\label{equa:diff}
	\delta_{1}^{1}(\bfe|(f_{1},f_{2},\dots,f_{n})) 
	&= \delta_{2}^{1}(\sigma,\bfe)|(f_{1,\sigma}\dots,f_{n,\sigma}) \notag \\
	&\quad + \delta_{2}^{1}((c_{1}\cdots c_{n})\tau,\bfe)|(f_{1,\tau}\dots,f_{n,\tau}) \notag \\
	&= \delta_{2}^{1}(\sigma,\bfe)|(f_{1,\sigma}\dots,f_{n,\sigma}) \notag \\
	&\quad + \phi(c_{1}\cdots c_{n})\delta_{2}^{1}(\bfe,\tau)|(f_{1,\tau}\dots,f_{n,\tau}).
\end{align}

Similar to $\calH_{n}^{\calK}$, $C_{(n,0)}^\calK\cdot\bfI_0$ also has a splitting into summand of subcomplexes, which are zigzags connected by the differentials in \ref{equa:diff}. In Figure \ref{fig:zigzagexample}, we have an example of the zigzag of $C_{(3,0)}^\calK\cdot \bfI_0$. Following the convention from \cite[Section~7]{chen2024space}, we call a zigzag \emph{contractible} if there exists a continuous nullhomotopy $h$, such that $\mathrm{Id}=\partial(h)$.
\begin{figure}
	\[\begin{tikzcd}
		& {\tilde{\scrU}_{1}^{2}\tilde{\mathbf{a}}|(\scrU_{1,1}\mathbf{a}_{1},\scrU_{1,2}\mathbf{a}_{2},\scrU_{1,3}\mathbf{a}_{3})} && {\tilde{\scrU}_{1}^{2}\tilde{\mathbf{a}}|(\mathbf{a}_{1},\mathbf{a}_{2},\mathbf{a}_{3})} && {\tilde{\mathbf{d}}|(\mathbf{d}_{1},\mathbf{d}_{2},\mathbf{d}_{3})} \\
		\cdots &&&&&& \cdots \\
		& {\tilde{\scrV}_{1}^{-3}\tilde{\mathbf{d}}|(\scrU_{1,1}\mathbf{a}_{1},\scrU_{1,2}\mathbf{a}_{2},\scrU_{1,3}\mathbf{a}_{3})} && {\tilde{\scrV}_{1}^{-3}\tilde{\mathbf{d}}|(\mathbf{a}_{1},\mathbf{a}_{2},\mathbf{a}_{3})} && {\tilde{\mathbf{d}}|(\mathbf{d}_{1},\mathbf{d}_{2},\mathbf{d}_{3})}
		\arrow["{\tilde{U}^{2}_{2}\tilde{\scrV}_{2}}"{description}, from=1-2, to=3-2]
		\arrow["1"{description}, from=1-4, to=3-2]
		\arrow["{\tilde{U}^{2}_{2}\tilde{\scrV}_{2}}"{description}, from=1-4, to=3-4]
		\arrow["{\tilde{\scrU}_2}"{description}, from=1-6, to=3-4]
		\arrow["1"{description}, from=1-6, to=3-6]
	\end{tikzcd}\]
	\caption{A zigzag of $C_{(3,0)}^\calK\cdot \bfI_0$}
	\label{fig:zigzagexample}
\end{figure}

\begin{lem}\label{lem:trivial}
	Given a zigzag, if all upper-row generators $\bfe|(f_{1},f_{2},\dots,f_{n})$ of this zigzag have $\bfe$ belonging to only $\{\tilde{\scrU}_{2}^{i}\tilde{\bfa}\}$ (respectively $\{\tilde{\scrV}_{2}^{j}\tilde{\bfd}\}$) , then it is contractible.
\end{lem}
\begin{proof}
	Suppose $\bfe$ belonging to $\{\tilde{\scrV}_{2}^{j}\tilde{\bfd}\}$, then checking the structure maps in Figure \ref{fig:D0A}, we see that the induced differentials on this zigzag have all the vertical differentials labeled by $1$, and the diagonal maps are labeled by $\tilde{U}_2^i\tilde{\scrU_2}^j$ with $j\geq1$. Define a map $h_{\uparrow}$ on the generators at the bottom of the zigzag by sending each generator to the generator directly above it. We also denote the diagonal maps by $\alpha_{\swarrow}$.
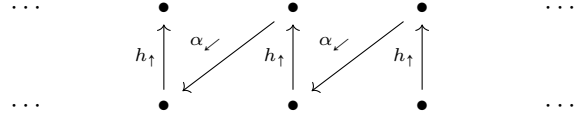
	\[\begin{tikzcd}[row sep=2.5em, column sep=3.5em]
		\cdots & \bullet & \bullet & \bullet & \cdots \\
		\cdots & \bullet & \bullet & \bullet & \cdots
		\arrow[from=2-2, to=1-2,"h_{\uparrow}"]
		\arrow[from=2-3, to=1-3,"h_{\uparrow}"]
		\arrow[from=2-4, to=1-4,"h_{\uparrow}"]
		\arrow[from=1-3, to=2-2,"\alpha_{\swarrow}"']
		\arrow[from=1-4, to=2-3,"\alpha_{\swarrow}"']
	\end{tikzcd}
	\]
	Then the null homotopy $h$ is defined as:
	\[h:=h_{\uparrow}\circ\Sigma_{n=0}^{\infty}(\alpha_{\swarrow}\circ h_{\uparrow})^n.\]
	It is straightforward to check that $\mathrm{Id}=\partial(h)$. To complete the proof, it remains to show that $h$ is continuous. Recall that $C_{(n,0)}^\calK$ is obtained by taking connected sums of copies of $\calH_1^\calK$. Each $\calH_1^\calK$ is equipped with cofinite basis topology. Hence, after tensoring these complexes together to form $C_{(n,0)}^\calK$, the resulting module is also equipped with cofinite basis topology. One can check that $h$ is continuous when $C_{(n,0)}^\calK$ is endowed with the cofinite basis topology and $\bbF[\tilde{\scrU}_2,\tilde{\scrV}_2]$ is endowed with the chiral topology. For more details, see \cite[Lemma~7.6]{chen2024space}.
	
	Similarly, if $\bfe$ belongs to $\{\tilde{\scrU}_{2}^{i}\tilde{\bfa}\}$, then all the diagonal differentials are labeled by $1$, and the vertical differentials are labeled by $\tilde{U}_2^i\tilde{\scrV_2}^j$ with $j\geq1$. By a similar argument, one can construct a null homotopy in this case as well.
\end{proof}

\begin{lem}\label{lem:red}
	Given a zigzag whose upper-row contains both $\tilde{\bfa}$ and $\tilde{\bfd}$ as its generator, then it is homotopic to a rank one submodule  with trivial structure map.
\end{lem}
\begin{proof}
		We check the place where the generator shifts from $\tilde{\bfa}$ to $\tilde{\bfd}$ in the following diagram.
	
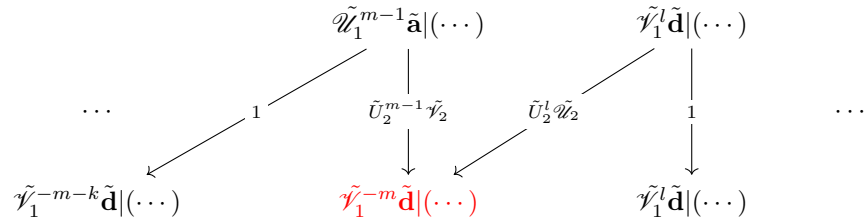
\[\begin{tikzcd}
		&& {\tilde{\scrU}_{1}^{m-1}\tilde{\mathbf{a}}|(\cdots)} && {\tilde{\scrV}_{1}^{l}\tilde{\mathbf{d}}|(\cdots)} \\
		\cdots &&&&& \cdots \\
		{\tilde{\scrV}_{1}^{-m-k}\tilde{\mathbf{d}}|(\cdots)} && \textcolor{red}{{\tilde{\scrV}_1^{-m}\tilde{\mathbf{d}}|(\cdots)}} && {\tilde{\scrV}_{1}^{l}\tilde{\mathbf{d}}|(\cdots)}
		\arrow["1"{description}, from=1-3, to=3-1]
		\arrow["{\tilde{U}_{2}^{m-1}\tilde{\scrV}_{2}}"{description}, from=1-3, to=3-3]
		\arrow["{\tilde{U}_{2}^{l}\tilde{\scrU}_{2}}"{description}, from=1-5, to=3-3]
		\arrow["1"{description}, from=1-5, to=3-5]
	\end{tikzcd}\]
	
To the left of the red generator, all the diagonal arrows labeled by $1$, and to the right of the red generator, all the vertical arrows are labeled by $1$. We claim that this zigzag is homotopy to the rank $1$ submodule generated by this red generator. Using the same notation as above, let $h_{\uparrow}$ denote the map which reverses the arrows labeled by 1, and let $\alpha_{\swarrow}$ denote the map corresponding to the arrows with nontrivial labels. Set
\[
G=\sum_{r=0}^{\infty}(\alpha\circ h_\uparrow)^r.
\]
The homotopy maps are given by:
\begin{enumerate}
	\item $i$ is the inclusion.
	\item $h$ is given by $h_\uparrow\circ G$.
	\item Let $\pi_r$ be the projection map from the zigzag to the red generator. Then $\pi$ is given by $\pi_r\circ G$.
\end{enumerate}
\end{proof}

Thus, in the $\bfI_0$ summand, zigzags whose upper-row generators are all of the same type can be removed, while each mixed zigzag contributes a single surviving generator. The analysis of the $\bfI_1$ summand is similar, except that $\tilde{\scrV}_2$ is inverted, making an additional family of zigzags contractible. In the next subsection, we describe the $\bfI_1$ summand and the coefficient actions on the surviving generators in both summands.
\subsection{Type-$A$ module $_\calK\calC_{(n,0)}$}
The preceding subsection simplifies the individual zigzags, but does not yet describe the coefficient actions on the surviving generators. To study these actions, we pass to the type-$A$ module
\[
{}_{\mathcal K}\mathcal C_{(n,0)} :=C_{(n,0)}^{\mathcal K}\widehat{\boxtimes}{}_{\mathcal K|\mathcal K}[\mathbb I^{\Supset}].
\]
We first describe the homology of its two idempotent summands by generators and relations, and then determine the induced actions of $\sigma$ and $\tau$. After pairing with the type-$AA$ identity module, each zigzag in $\bfI_0\cdot _\calK\calC_{n,0}$ converts to chain complex over $\bbF[\tilde{\scrU_2},\tilde{\scrV_2}]$. The homology is again generated by the red generator. To describe the coefficient actions, however, it is more convenient to retain all lower-row generators as representatives of homology classes. The differentials of the upper-row generators impose relations among these classes. Although there are infinitely many zigzags involving both $\tilde{\bfa}$ and $\tilde{\bfd}$, resulting in a large collection of bottom generators, these generators are finitely generated as a $R_n\otimes_{\bbF}\bbF[\tilde{\scrU_2},\tilde{\scrV_2}]$ module.
\begin{prop}\label{prop:1stbasis}
	The homology of $\bfI_0\cdot{}_{\mathcal K}\mathcal C_{(n,0)}$ is generated, as an $R_n\otimes_{\mathbb F}\mathbb F[\tilde{\scrU}_2,\tilde{\scrV}_2]$-module, by the classes represented by
	\begin{align*}
		x_{1}^{0} &:= \tilde{\scrV}_{1}^{-1} \tilde{\bfd} | (\bfa_{1},\dots,\bfa_{n}), \\
		\vdots &\qquad \qquad  \vdots  \\
		x_{n}^{0} &:= \tilde{\scrV}_{1}^{-n} \tilde{\bfd} | (\bfa_{1},\dots,\bfa_{n}). 
	\end{align*}
\end{prop}

\begin{proof}[Proof of Proposition \ref{prop:1stbasis}]
	
	Given a generator $\tilde{\scrV}_{1}^{-m}\tilde{\mathbf{d}}|(f_{1},f_{2},\dots,f_{n})$, such that $f_{i}$ takes values in $\{\scrU_{1,i}^{s_{i}}{\bfa}_{i}$ or $\scrV_{1,i}^{s_{i}}{\bfd}_{i}\}$. Let $\#\{f_{i}=\scrV_{1,i}^{s_{i}}{\bfd}_{i}\}=t$. Suppose it is a bottom generator from a zigzag which involves both $\tilde{\bfa}$ and $\tilde{\bfd}$, then we have
    \[-n\leq -m-t\leq -1.\]
    
	Without loss of generality, we assume $\{f_{1},\dots,f_{n-t}\}\in\{\scrU_{1,i}^{s_{i}}{\bfa}_{i}\}$, and $\{f_{n-t+1},\dots,f_{n}\}\in\{\scrV_{1,i}^{s_{i}}{\bfd}_{i}\}$. Using the coefficient action from Proposition \ref{prop:hopf}, it is not hard to show that
	\begin{align*}
		\tilde{\scrV}_{1}^{-m}\tilde{\mathbf{d}} | (f_{1}, f_{2}, \dots, f_{n}) 
		&= (\scrU_{1,1}^{s_{1}}, \dots, \scrU_{1,n-t}^{s_{n-t}}, \scrV_{1,n-t+1}^{s_{n-t}+1}, \dots, \scrV_{1,n}^{s_{n}+1}) 
		\tilde{\scrV}_{1}^{-m-t}\tilde{\mathbf{d}} | (\bfa_{1}, \bfa_{2}, \dots, \bfa_{n}) \\ 
		&= (\scrU_{1,1}^{s_{1}}, \dots, \scrU_{1,n-t}^{s_{n-t}}, \scrV_{1,n-t+1}^{s_{n-t+1}+1}, \dots, \scrV_{1,n}^{s_{n}+1}) x_{m+t}^{0}.
	\end{align*}		
\end{proof}

The relations between these generators are given as follows. The first relation arises from the Hopf link bimodule, where we identify all the $U$ actions, i.e.
\begin{equation}\label{eq:Urelation}
	\scrU_{1,i}\scrV_{1,i}=\tilde{\scrU}_{2}\tilde{\scrV}_{2}.
\end{equation}

The other relations between coefficient actions are from the zigzags. We need to check the pairs of coefficients from $R_n$ which send the generator to the same zigzag.
We list the pairs below:

\begin{enumerate}
	\item $(\scrU_{1,1},\scrU_{1,2},\cdots,\scrU_{1,n}),$ and $(1,1,\cdots,1)$\label{coeffrel(1)}
	\item $(\scrV_{1,1},\scrV_{1,2},\cdots,\scrV_{1,n}),$ and $(1,1,\cdots,1)$\label{coeffrel(2)}
	\item $(\scrU_{1,1},\cdots,\scrU_{1,m},1,\cdots,1),$ and $(1,\cdots,1,\scrV_{1,m+1},\cdots,\scrV_{1,n})$, for $m\in\{1,2,\cdots,n-1\}$\label{coeffrel(3)}
\end{enumerate}

We are also allowed to change the positions of $\scrU_{1,i}$ and $\scrV_{1,i}$ for the third pair of coefficients, provided their positions remain complementary. To simplify the notation without introducing ambiguity, we adopt the abbreviations
\begin{center}
	$(\scrU\cdots \scrU)_{m},\quad \mathrm{and}\quad(\scrV\cdots \scrV)_{n-m}$,
\end{center}
which denotes the coefficient with $m$ (respectively, $n-m$) $\scrU_{1,i}$ (respectively, $\scrV_{1,i}$).

We now check the relations within each pair of coefficient actions. We list the zigzags used in the computation and the relations below:

\subsubsection*{Coefficient actions of pair \ref{coeffrel(1)}}
\[\begin{tikzcd}
	&&& {\tilde{\scrU}_{1}^{k-1}\tilde{\mathbf{a}}|(\mathbf{a}_{1},\cdots,\mathbf{a}_{n})} \\
	\cdots &&&& \cdots \\
	& {\tilde{\scrV}_{1}^{-k}\tilde{\mathbf{d}}|(\scrU_{1,1}\mathbf{a}_{1},\cdots,\scrU_{1,n}\mathbf{a}_{n})} && {\tilde{\scrV}_{1}^{-k}\tilde{\mathbf{d}}|(\mathbf{a}_{1},\cdots,\mathbf{a}_{n})}
	\arrow["1"{description}, from=1-4, to=3-2]
	\arrow["{\tilde{U}_{2}^{k-1}\tilde{\scrV}_{2}}"{description}, from=1-4, to=3-4]
\end{tikzcd}\]
The corresponding relation is:
\begin{center}
	${(\scrU\cdots \scrU)}_{n}\cdot x_{k}^{0}=x_{k}^{0}\otimes \tilde{U}_{2}^{k-1}\tilde{\scrV}_{2}.$
\end{center}
Note that, the relation is induced from the structure map (the vertical map) of the Hopf link bimodule, which is determined by the factor $\tilde{\scrU}_{1}^{k-1}\tilde{\mathbf{a}}$. Hence, it extends naturally with respect to other $\scrU_{1,i}$ action. If there is a $\scrV_{1,i}$ action, then we can use relation \ref{eq:Urelation} to reduce it to the third pair \ref{coeffrel(3)}.

\subsubsection*{Coefficient actions of pair \ref{coeffrel(2)}}
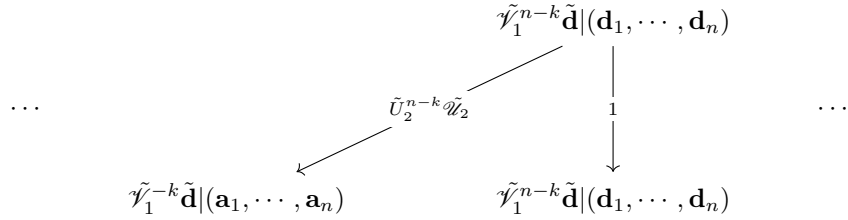
\[\begin{tikzcd}
	&&& {\tilde{\scrV}_{1}^{n-k}\tilde{\mathbf{d}}|(\mathbf{d}_{1},\cdots,\mathbf{d}_{n})} \\
	\cdots &&&& \cdots \\
	& {\tilde{\scrV}_{1}^{-k}\tilde{\mathbf{d}}|(\mathbf{a}_{1},\cdots,\mathbf{a}_{n})} && {\tilde{\scrV}_{1}^{n-k}\tilde{\mathbf{d}}|(\mathbf{d}_{1},\cdots,\mathbf{d}_{n})}
	\arrow["{\tilde{U}_{2}^{n-k}\tilde{\scrU}_{2}}"{description}, from=1-4, to=3-2]
	\arrow["1"{description}, from=1-4, to=3-4]
\end{tikzcd}\]

The corresponding relation is:

\begin{center}
	${(\scrV\cdots \scrV)}_{n}\cdot x_{k}^{0}=x_{k}^{0}\otimes \tilde{U}_{2}^{n-k}\tilde{\scrU}_{2}.$
\end{center}

By a similar argument, this relation is natural with respect to the $\scrV_{1,i}$ action, and any $\scrU_{1,i}$ action will reduce it to the third pair \ref{coeffrel(3)}.

\subsubsection*{Coefficient actions of pair \ref{coeffrel(3)}}
We study the relations of coefficient actions in two cases.\\

\textbf{Case 1: $n-m-k\geq0$}

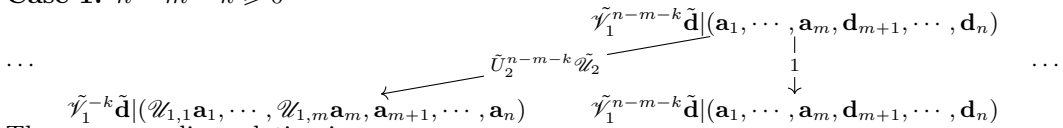
\begin{tikzcd}[every cell/.style={font=\small},column sep=small, row sep=small]
	&&& {\tilde{\scrV}_{1}^{n-m-k}\tilde{\bfd}|(\bfa_{1},\cdots,\bfa_m,\bfd_{m+1},\cdots,\bfd_n)} \\
	\cdots &&&& \cdots \\
	& {\tilde{\scrV}_{1}^{-k}\tilde{\bfd}|(\scrU_{1,1}\bfa_1,\cdots,\scrU_{1,m}\bfa_m,\bfa_{m+1},\cdots,\bfa_n)} && {\tilde{\scrV}_{1}^{n-m-k}\tilde{\bfd}|(\bfa_1,\cdots,\bfa_m,\bfd_{m+1},\cdots,\bfd_n)}
	\arrow["{\tilde{U}_{2}^{n-m-k}\tilde{\scrU}_{2}}"{description}, from=1-4, to=3-2]
	\arrow["1"{description}, from=1-4, to=3-4]
\end{tikzcd}

The corresponding relation is:

\begin{center}
	${(\scrU\cdots \scrU)}_{m}\cdot x_{k}^{0}\otimes \tilde{U}_{2}^{n-m-k}\tilde{\scrU}_{2}=(\scrV\cdots \scrV)_{n-m}\cdot x_{k}^{0}.$
\end{center}

Note that, when $n-m-k=0$, we have
\begin{equation}\label{coeffrel{3.1}}
	{(\scrU\cdots \scrU)}_{n-k}\cdot x_{k}^{0}\otimes \tilde{\scrU}_{2}=(\scrV\cdots \scrV)_{k}\cdot x_{k}^{0}.
\end{equation}

One can check that, by multiplying $\{\scrV_{1,i}\}$'s on both sides to equation \ref{coeffrel{3.1}}, and applying relation \ref{eq:Urelation}, relation \ref{coeffrel{3.1}} generates the other relations in this case.\\

\textbf{Case 2: $n-m-k<0$}

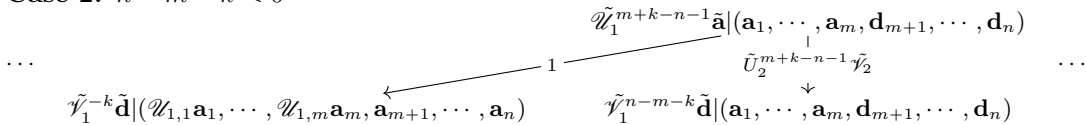
\begin{tikzcd}[every cell/.style={font=\small},column sep=small, row sep=small]
	&&& {\tilde{\scrU}_{1}^{m+k-n-1}\tilde{\bfa}|(\bfa_1,\cdots,\bfa_m,\bfd_{m+1},\cdots,\bfd_{n})} \\
	\cdots &&&& \cdots \\
	& {\tilde{\scrV}_{1}^{-k}\tilde{\bfd}|(\scrU_{1,1}\bfa_{1},\cdots,\scrU_{1,m}\bfa_{m},\bfa_{m+1},\cdots,\bfa_{n})} && {\tilde{\scrV}_{1}^{n-m-k}\tilde{\bfd}|(\bfa_{1},\cdots,\bfa_{m},\bfd_{m+1},\cdots,\bfd_{n})}
	\arrow["1"{description}, from=1-4, to=3-2]
	\arrow["{\tilde{U}_{2}^{m+k-n-1}\tilde{\scrV}_{2}}"{description}, from=1-4, to=3-4]
\end{tikzcd}

The corresponding relation is:

\begin{center}
	${(\scrU\cdots \scrU)}_{m}\cdot x_{k}^{0}=(\scrV\cdots \scrV)_{n-m}\cdot x_{k}^{0}\otimes \tilde{U}_{2}^{m+k-n-1}\tilde{\scrV}_{2}.$
\end{center}

Similarly to the discussion of the case above, when $n-m-k=-1$, we have
\begin{equation}\label{coeffrel{3.2}}
	{(\scrU\cdots \scrU)}_{n-k+1}\cdot x_{k}^{0}=(\scrV\cdots \scrV)_{k-1}\cdot x_{k}^{0}\otimes \tilde{\scrV}_{2},
\end{equation}
which generates the other relations in this case.

The computation of $\bfI_1\cdot_\calK C_{(n,0)}$ is almost the same as $\bfI_0\cdot_\calK C_{(n,0)}$, except that we change the coefficient ring from $\bbF_{2}[\tilde{\scrU}_{2},\tilde{\scrV}_{2}]$ to $\bbF_{2}[\tilde{\scrU}_{2},\tilde{\scrV}_{2},\tilde{\scrV}_{2}^{-1}]$. Hence, when checking the zigzag contains $\tilde{\scrV}_{1}^{-1}\tilde{\bfd}|(\bfa_{1},\cdots,\bfa_{n})$, since $\tilde{\scrV}_{2}$ is invertible, it is contractible.
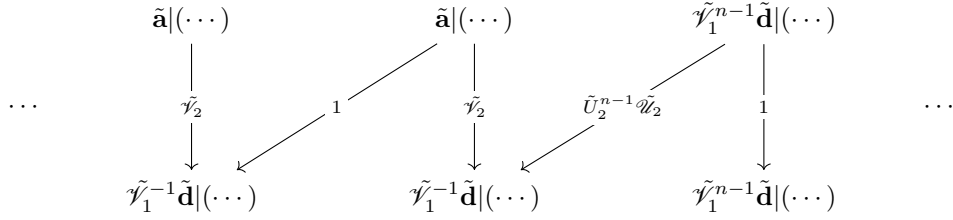
\[\begin{tikzcd}
	& {\tilde{\mathbf{a}}|(\cdots)} && {\tilde{\mathbf{a}}|(\cdots)} && {\tilde{\scrV}_{1}^{n-1}\tilde{\mathbf{d}}|(\cdots)} \\
	\cdots &&&&&& \cdots \\
	& {\tilde{\scrV}_{1}^{-1}\tilde{\mathbf{d}}|(\cdots)} && {\tilde{\scrV}_{1}^{-1}\tilde{\mathbf{d}}|(\cdots)} && {\tilde{\scrV}_{1}^{n-1}\tilde{\mathbf{d}}|(\cdots)}
	\arrow["{\tilde{\scrV}_{2}}"{description}, from=1-2, to=3-2]
	\arrow["1"{description}, from=1-4, to=3-2]
	\arrow["{\tilde{\scrV}_{2}}"{description}, from=1-4, to=3-4]
	\arrow["{\tilde{U}_{2}^{n-1}\tilde{\scrU}_{2}}"{description}, from=1-6, to=3-4]
	\arrow["1"{description}, from=1-6, to=3-6]
\end{tikzcd}\]
Analogously, we have the following.

\begin{prop}\label{prop:2ndbasis}
	The homology of $\bfI_1\cdot{}_{\mathcal K}\mathcal C_{(n,0)}$ is generated, as an $R_n\otimes_{\bbF}\bbF[\tilde{\scrU_2},\tilde{\scrV_2},\tilde{\scrV}_2^{-1}]$-module, by the classes represented by
	\begin{align*}
		x_{2}^{1} &:= \tilde{\scrV}_{1}^{-2} \tilde{\bfd} | (\bfa_{1},\dots,\bfa_{n}),  \\
		\vdots &\qquad \qquad  \vdots \\
		x_{n}^{1} &:= \tilde{\scrV}_{1}^{-n} \tilde{\bfd} | (\bfa_{1},\dots,\bfa_{n}). 
	\end{align*}
    The coefficient relations are obtained from the ones of $\bfI_0\cdot{}_{\mathcal K}\mathcal C_{(n,0)}$ by replacing $x_k^0$ with $x_k^1$, for $2\leq k\leq n$, and allowing $\tilde{\scrV}_2$ to be invertible.
\end{prop}

With this generators, the structure map also has a simple description. Applying the homological perturbation lemma, we can compute the structure maps from $\bfI_0\cdot_\calK C_{(n,0)}$ to $\bfI_1\cdot_\calK C_{(n,0)}$, which are given by:
\begin{align}
	m_2(\sigma,x_{i}^{0}) &= x_{i}^{1},  &\quad \text{for } i &\in \{2, \dots, n\},\notag \\
	m_2(\tau,x_{i}^{0})   &= x_{i+1}^{1}, &\quad \text{for } i &\in \{1, \dots, n-1\}.\notag
\end{align}
These structure maps are nature with respect to the $R_n$ action and extend to actions of $R_n\otimes_{\bbF}\bbF[\tilde{\scrU_2},\tilde{\scrV_2}]$, $R_n\otimes_{\bbF}\bbF[\tilde{\scrU_2},\tilde{\scrV_2},\tilde{\scrV}_2^{-1}]$ via the relations defined in the knot surgery algebra $\calK$.

\subsection{$_\calK C_{(n,0)}$ as a $R_n$ module}\label{subsec:An0}
Since our eventual goal is to pair $_\calK C_{(n,0)}$ with type-$D$ module of the knot complement $\calX_{\lambda}(K)^{\calK}$ to get the link Floer homology $\CFL(K_{(n,\lambda n)})$, which is an $R_n$ module, we want to represent $_\calK C_{(n,0)}$ as an  $R_n$ module. To do so, we replace the actions of $R_n\otimes_{\bbF}\bbF[\tilde{\scrU_2},\tilde{\scrV_2}]$, $R_n\otimes_{\bbF}\bbF[\tilde{\scrU_2},\tilde{\scrV_2},\tilde{\scrV}_2^{-1}]$ by enlarging the set of generators. We claim $_\calK C_{(n,0)}$ is homotopic to $_\calK\mathcal{C}_{(n,0)}$, which is defined as follows.

$\bfI_0\cdot_\calK\mathcal{C}_{(n,0)}^\calK$ is a quotient $R_n$ module, which is generated by 
\begin{enumerate}
	\item $x_{1,i}^{0}$ for $i\in\{0,-1,-2,\dots\}$;
	\item $x_{k,i}^{0}$ for $k\in\{2,\dots,n-1\}$ and $i\in \mathbb{Z}$;
	\item $x_{n,i}^{0}$ for $i\in\{0,1,2,\dots\}$.
\end{enumerate} 

The relations are generated by

\begin{enumerate}
	\item $(\scrU\cdots \scrU)_{n-1}\cdot {x_{1}^{0}}_{,i-1} = (\scrV)_{1}\cdot {x_{1}^{0}}_{,i}$;
	\item For $k\in\{2,\cdots,n-1\}$,\\
	$(\scrU\cdots \scrU)_{n-k+1}\cdot x_{k,i}^{0} = (\scrV\cdots\scrV)_{k-1}\cdot x_{k,i+1}^{0}  \quad \text{for } i \in \{0,1,2,\dots\},$\\
	$(\scrU\cdots\scrU)_{n-k}\cdot x_{k,i-1}^{0} = (\scrV\cdots\scrV)_{k}\cdot x_{k,i}^{0}  \quad \text{for } i \in \{0,-1,-2,\dots\}$;
	\item $(\scrU)_{1}\cdot x_{k,i}^{0} = (\scrV\cdots\scrV)_{n-1}\cdot x_{k,i+1}^{0}.$
\end{enumerate}

$\bfI_{1}\cdot{_{\mathcal{K}}\mathcal{C}_{(n,0)}}$ is generated by
\[x_{k,i}^{1}\quad \mathrm{for } \quad k\in\{2,\cdots,n\}\quad \mathrm{and } \quad i\in\mathbb{Z}.\]
And the relations
\[(\scrU\cdots\scrU)_{n-k+1}\cdot x_{k,i}^{1} = (\scrV\cdots\scrV)_{k-1}\cdot x_{k,i+1}^{1}.\]

The structure map is only trivial for $m_2$, which is given by the coefficient action of $R_n$ together with
\begin{enumerate}
	\item $m_2(\sigma, x_{1,i}^{0})=0,$\\
	$m_2(\tau, x_{1,i}^{0})=x_{2,i}^{1};$
	\item For $k\in\{2,\cdots,n-1\}$,\\
	$m_2(\sigma, x_{k,i}^{0})=$
	$\begin{cases}
		x_{k,i}^{1}, & \text{if } i\geq 0,\\
		x_{k,i}^{1}\otimes \tilde{U}_{2}^{\abs{i}}, & \text{if } i< 0,
	\end{cases}$\\
	$m_2(\tau,x_{k,i}^{0})=$
	$\begin{cases}
		x_{k+1,i}^{1}\otimes \tilde{U}_{2}^{i}, & \text{if } i\geq 0,\\
		x_{k+1,i}^{1}, & \text{if } i< 0;
	\end{cases}$
	\item $m_2(\sigma,x_{n,i}^{0})=x_{n,i}^{1},$\\
	$m_2(\tau,x_{n,i}^{0})=0$.
\end{enumerate}
\begin{Remark}
	Although the structure maps above involve powers of $\tilde{U}_2$, these powers can be represented by other elements of $R_n$. We omit this identification there, and perform the replacement only when needed.
\end{Remark}
\begin{prop}\label{propnewbasis}
	${_{\mathcal{K}}C_{(n,0)}}$ is homotopic to ${_{\mathcal{K}}\mathcal{C}_{(n,0)}}$.
\end{prop}
\begin{proof}
	Proposition \ref{prop:1stbasis} (respectively, Proposition \ref{prop:2ndbasis}) already provides a representation over $R_n\otimes_{\bbF}\bbF[\tilde{\scrU_2},\tilde{\scrV_2}]$ (respectively, over $R_n\otimes_{\bbF}\bbF[\tilde{\scrU_2},\tilde{\scrV_2},\tilde{\scrV}_2^{-1}]$). To prove the present proposition, it remains only to show the generators of ${_{\mathcal{K}}C_{(n,0)}}$ is a generating set over $R_n$ and check the consistency of these relations. We begin with the idempotent $\bfI_0\cdot{_{\mathcal{K}}\mathcal{C}_{(n,0)}}$.
	\begin{enumerate}
		\item Consider the generator $x_{1}^{0}$. Since we have
        \[(\scrU,\cdots,\scrU)_{n}\cdot x_{1}^{0}=x_{1}^{0}\otimes\tilde{\scrV}_{2},\]
		which realize the $\tilde{\scrV}_{2}$ action as the action by $(\scrU,\cdots,\scrU)_{n}$. It therefore suffices to determine a basis corresponding to the $\tilde{\scrU}_{2}$ action to obtain the full vector space structure of $\bbF_{2}(\tilde{\scrU}_{2},\tilde{\scrV}_{2})$.
		Define
		\[x_{1,i}^{0}:=x_{1}^{0}\otimes\tilde{\scrU}_{2}^{\abs{i}} \quad \mathrm{for} \quad i\in\{0,-1,-2,\cdots\}.\]
		Then we have
		\[(\scrU\cdots \scrU)_{n-1}\cdot {x_{1}^{0}}_{,i-1} = (\scrV)_{1}\cdot {x_{1}^{0}}_{,i},\]
		which generates the relations.
		\item For the generators $x_{k}^{0}$, where $2\leq k\leq n-1$. We first check the $(\scrU\cdots\scrU)_{n}$ and $(\scrV\cdots \scrV)_{n}$ action, to determine a basis.
		\begin{align*}\label{basis2ton-1}
			(\scrU\cdots\scrU)_{n}\cdot x_{k}^{0} &= x_{k}^{0}\otimes \tilde{U}_{2}^{k-1}\tilde{\scrV}_{2}, \\
			(\scrV\cdots\scrV)_{n}\cdot x_{k}^{0} &= x_{k}^{0}\otimes \tilde{U}_{2}^{n-k}\tilde{\scrU}_{2}. \notag
		\end{align*}
		Since both relations contain nontrivial $\tilde{U}_{2}$ power, to achieve all $\tilde{\scrU}_{2}$ and $\tilde{\scrV}_{2}$ action, we need the following basis.
		Define
		\begin{align*}
			{x_{k}^{0}}_{,i} & := x_{k}^{0} \otimes \tilde{\scrV}_{2}^{i}, \quad \text{for } i \in \{0,1,2,\dots\},\\
			{x_{k}^{0}}_{,i} & := x_{k}^{0} \otimes \tilde{\scrU}_{2}^{\abs{i}}, \quad \text{for } i \in \{-1,-2,\dots\}. \notag
		\end{align*}
		We have the following coefficient action relations generating all the other relations:
		\begin{align*}
			(\scrU\cdots \scrU)_{n-k+1}\cdot x_{k,i}^{0} &= (\scrV\cdots\scrV)_{k-1}\cdot x_{k,i+1}^{0}  \quad \text{for } i \in \{0,1,2,\dots\},\\
			(\scrU\cdots\scrU)_{n-k}\cdot x_{k,i-1}^{0} &= (\scrV\cdots\scrV)_{k}\cdot x_{k,i}^{0}  \quad \text{for } i \in \{0,-1,-2,\dots\}.
		\end{align*}
		\item For the generator $x_{n}^{0}$. Since
		\[(\scrV,\cdots,\scrV)_{n}\cdot x_{n}^{0}=x_{n}^{0}\otimes\tilde{\scrU}_{2},\]
		it suffices to use only the $\tilde{\scrV}_{2}$ action.
		Define
		\begin{center}
			${x_{n}^{0}}_{,i}:=x_{n}^{0}\otimes\tilde{\scrV}_{2}^{i}$ for $i\in\{0,1,2,\cdots\}$.
		\end{center}
		The relation of coefficient action is generated by:
		\[(\scrU)_{1}\cdot x_{n,i}^{0} = (\scrV\cdots\scrV)_{n-1}\cdot x_{n,i+1}^{0}.\]
	\end{enumerate}
	
	We now perform a similar computation to determine a generating set for $\bfI_{1}\cdot{_{\mathcal{K}}C_{(n,0)}}$. Note that, the relations of the coefficient actions remain the same as $\bfI_{0}\cdot{_{\mathcal{K}}C_{(n,0)}}$. However, upon localization at $\tilde{\scrV}_{2}$, it suffices to consider a generating set corresponding to the $\tilde{\scrV}_{2}$ power, and the relations of coefficient actions are generated by \ref{coeffrel{3.2}}. In summary, we define the generators of $\bfI_{1}\cdot{_{\mathcal{K}}C_{(n,0)}}$ as follows.
	Define
    \[x_{k,i}^{1}:=x_{k}^{1}\otimes\tilde{\scrV}_{2}^{i} \quad \mathrm{for} \quad k\in\{2,\cdots,n\} \quad \mathrm{and} \quad i\in\mathbb{Z}.\]
	Together with the relations:
    \[(\scrU\cdots \scrU)_{n-k+1}\cdot x_{k,i}^{1} = (\scrV\cdots\scrV)_{k-1}\cdot x_{k,i+1}^{1}.\]
	It is straightforward to verify that the induced structure maps on this generating set agree with the structure maps of ${_{\mathcal{K}}\mathcal{C}_{(n,0)}}$. Hence Proposition \ref{propnewbasis} follows.
\end{proof}

\section{Cabling modules and the full-twist map}\label{sec:coloredmodule}
Recall that the colored knot Floer homology is defined as the colimit of the direct system in Equation~\ref{eq:coloreddef}. Thus, to compute it, we need to understand both $\HFL(K_{n,nm})$ and the cobordism maps $\phi_0:\HFL(K_{n,nm})\rightarrow\HFL(K_{n,n(m+1)})$. In Section \ref{sec:cabling}, we computed the bordered module for the $(n,0)$-cabling pattern. The bordered modules corresponding to other framings will be denoted by
\[
C_{(n,m)}^{\calK} \qquad\text{and}\qquad {}_{\calK}\calC_{(n,m)}
\]
in the type-$D$ and type-$A$ descriptions, respectively. Here the
parameter $m$ corresponds to the cable $K_{n,nm}$. The type-$D$
module $C_{(n,m)}^{\calK}$ is obtained from $C_{(n,0)}^{\calK}$
by incorporating the type-$D$ module $\calD_m^{\calK}$ of the
$m$-framed solid torus. The module $\calD_m^\calK$ has two generators $\xs^0$ and $\xs^1$ which are in idempotent $\bfI_0$ and $\bfI_1$, respectively. The structure map is given by
\[\delta^1(\xs^0)=\xs^1\otimes(\sigma+\scrV^m\tau).\]
Thus, tensoring with $\calD_m^\calK$ simply modifies each structure map with $\tau$ output with an extra $\scrV^m$. For the type-$A$ description in Subsection \ref{subsec:An0}. The type-$A$ module $_\calK\calC_{(n,m)}$ has the same generators as $_\calK\calC_{(n,0)}$, and the structure maps with $\tau$ input have an extra $\tilde{\scrV}_2^m$ output. Thus, to simplify the notation, for different values of $m$, we use only $x_{(k,i)}^0$ and $x_{(k,i)}^1$ to denote the generators of $_\calK\calC_{(n,m)}$, suppressing $m$ from the notation. The structure maps are given by
\[m_2^{(m)}(\sigma,x)=m_2^{(0)}(\sigma,x),\qquad m_2^{(m)}(\tau,x)=\tilde{\scrV}_2^m m_2^{(0)}(\tau,x).\]

There are two useful ways to pass from framing $m$ to framing
$m+1$. One can describe the change algebraically using the
$1$-framed solid-torus module, or geometrically by performing
$(-1)$-surgery on the dual knot. The first description makes the
change in the bordered module transparent, whereas the second is
better suited to identifying the cobordism map $\phi_0$.

We therefore begin with the surgery description. We compute the
Alexander grading on the corresponding surgery complex and apply
Theorem~\ref{theorem:linkcobsurgery} to identify the subcomplex
representing $\phi_0$. We then compose this inclusion with the chain
homotopy equivalence relating the surgery description to the
bordered description above. In this way, we obtain an explicit
description of the cobordism map at the level of the type-$A$
cabling modules.

\subsection{Connected sum decomposition and Alexander grading of the subcomplex}
Figure \ref{fig:plumbingnk} gives a plumbing description of $_\calK\calC_{(n,m+1)}$ obtained from the plumbing diagram for $_\calK\calC_{(n,m)}$ by performing a $-1$ surgery on the dual knot. Algebraically, this gives
\begin{equation}\label{eq:cobconnectsum}
	_\calK\calC_{(n,m+1)}\cong (C_{(n,m)}^\calK\otimes (D_0^\calK\Boxtimes{\AbZ_{(-1,0)}}))\Boxtimes _{\calK|\calK}M^\calK\Boxtimes_{\calK|\calK}[I]^\Supset.
\end{equation}

\begin{figure}[h]
	
	\centering
	\begin{tikzpicture}[scale=1, every node/.style={circle, draw, inner sep=1.5pt}]
		
		\node[fill=black, label=above:{\small $0$}] (v0) at (0,0) {};
		
		\node[label=above:{\small $m+1$}] (vL) at (-2,0) {};
		\draw (v0) -- (vL);
		
		\node (vU) at (2,1.5) {};
		\draw (v0) -- (vU);
		
		\node (vD) at (2,-1.5) {};
		\draw (v0) -- (vD);
		
		\foreach \y in {-0.5,0,0.5} {
			\fill (2,\y) circle (1.5pt);
		}
		
		\draw [decorate,decoration={brace,amplitude=5pt}]
		(2.2,1.5) -- (2.2,-1.5)
		node[midway, right=5pt, draw=none] {$n$};
		
		\node[draw=none] at (3.5,0) {\Large$\cong$};
		\node[fill=black, label=above:{\small $0$}] (w0) at (7,0) {};
		
		\node[label=left:{\small $m$}] (wL) at (5,0) {};
		\draw (w0) -- (wL);
		
		\node[fill=black, label=above:{\small $-1$}] (wLu) at (5,1) {};
		\draw (wLu) -- (wL);
		
		\node (wU) at (9,1.5) {};
		\draw (w0) -- (wU);
		
		\node (wD) at (9,-1.5) {};
		\draw (w0) -- (wD);
		
		\foreach \y in {-0.5,0,0.5} {
			\fill (9,\y) circle (1.5pt);
		}
		
		\draw [decorate,decoration={brace,amplitude=5pt}]
		(9.2,1.5) -- (9.2,-1.5)
		node[midway, right=5pt, draw=none] {$n$};
		
		\node[draw=none] at (3.5,-4.5) {\Large$\cong$};
		\node[fill=black, label=above:{\small $0$}] (z0) at (7,-4.5) {};
		
		\node[label=left:{\small $m$}] (zL) at (5,-4.5) {};
		\draw (z0) -- (zL);
		
		\node[fill=black, label=above:{\small $-1$}] (zLu1) at (5,-2.5) {};
		\node[fill=black, label=left:{\small $0$}] (zLu2) at (5,-3.5) {};
		\draw (zLu2) -- (zLu1);
		\node[draw=none] at (5,-4) {\Large$\#$};
		\node (zU) at (9,-3) {};
		\draw (z0) -- (zU);
		
		\node (zD) at (9,-6) {};
		\draw (z0) -- (zD);
		
		\foreach \y in {-5,-4.5,-4} {
			\fill (9,\y) circle (1.5pt);
		}
		
		\draw [decorate,decoration={brace,amplitude=5pt}]
		(9.2,-3) -- (9.2,-6)
		node[midway, right=5pt, draw=none] {$n$};
	\end{tikzpicture}
	\caption{Plumbing diagram which describes the cobordism}
	\label{fig:plumbingnk}
	
\end{figure}
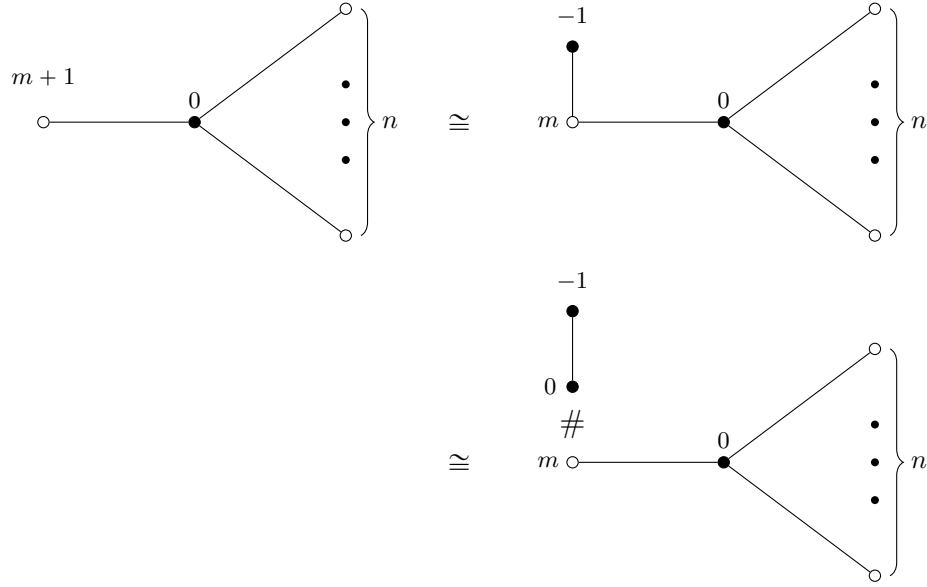

Pairing $_\calK\calC_{(n,m+1)}$ with the type-$D$ module of the given knot $\calX_0(K)^\calK$ gives the link Floer complex of $K_{n,n(m+1)}$. By Theorem~\ref{theorem:linkcobsurgery}, the cobordism map $\phi_0$ is represented by the inclusion of a subcomplex of the corresponding link surgery complex. Recall that $\phi_0$ is the cobordism map corresponding to the $\spin^c$ structure $\frs_0$ satisfying $$\langle c_1(\frs_0),[S^2]\rangle=1,$$ where $[S^2]$ is the generator of the second homology class of the cobordism. Because our construction of the cabling module reverses the
orientation convention, as explained in Remark \ref{remark:orientation-}, the corresponding $\spin^c$ structure in the present surgery description satisfies $$\langle c_1(\frs_0),[S^2]\rangle=-1.$$ Then using the correspondence between the Alexander grading of $_\calK\calC_{(n,m+1)}$ and the $\spin^c$ structure on cobordism \ref{eq:idspinA}, we can identify the subcomplex representing $\phi_0$. 

We label the vertices of the plumbing diagram in Figure \ref{fig:plumbingnk} as follows. The first three vertices are
the vertices involved in the framing-change cobordism, beginning with the vertex of weight $-1$, and are labeled $1,2,3$. The remaining $n$ vertices correspond to the components of the cable pattern and are labeled $4,\ldots,n+3$. Under this labeling, the Alexander grading for the corresponding link surgery formula is denoted by $(a_1,a_2,\dots,a_{n+3})\in\bbH(L)$, where
\[
\bbH(L)_i = \frac{\operatorname{lk}(L_i,L-L_i)}{2} +\bbZ.
\]
It will be convenient to remove the linking number shift and write
\[
a_i
=
\frac{\operatorname{lk}(L_i,L-L_i)}{2}+s_i,
\qquad
s_i\in\bbZ.
\]
We denote the resulting integral coordinates by
\[
\ve{s}=(s_1,\ldots,s_{n+3}).
\] In terms of $\ve{s}$, the restriction map $\psi$ is simply given by the projection to the remaining components. When taking connected sum of two links $L_1$ and $L_2$, the grading $\ve{s}$ of the resulting link is formed by concatenating $\ve{s}_1$ and $\ve{s}_2$, except that the entries corresponding to two components involved in the connected sum are added together. Since the vertices labeled from $4$ to $3+n$ correspond to the cable components, the $\spin^c$ structure of the framing change cobordism is determined by the first three coordinates. We therefore write
\[
\bar{\ve{s}}=(s_1,s_2,s_3).
\] The corresponding $\spin^c$ structure is
\[c_1([\bar{\ve{s}}])=[2\bar{\ve{s}}-(-1,m,0)],\]
where the equivalence relation is given by
\begin{equation}\label{eq:eqrelation}
	s_2=0 \quad \mathrm{and}\quad s_1+s_3=ms_2.
\end{equation}
Hence, the generators that induce $\frs_0$ are precisely those with Alexander grading 
\[2(s_1-s_3)=-2\Longrightarrow s_1=s_3-1.\]

Finally, taking the connected sum with the companion knot $K$ changes only the $s_2$-coordinate. This coordinate becomes trivial in the quotient~\eqref{eq:eqrelation}, and therefore does not affect the $\spin^c$ structure relevant to the cobordism map. Consequently, the subcomplex defining $\phi_0$ is already determined by the cabling module ${}_{\calK}\calC_{(n,m+1)}$ and is independent of the companion
knot. Using the connected sum  decomposition of Equation \ref{eq:cobconnectsum}, $\bar{\ve{s}}$ depends only on $(C_{(n,m)}^\calK\otimes (D_0^\calK\Boxtimes{\AbZ_{(-1,0)}}))$. We compute these Alexander gradings in the following subsection.

\subsection{More on Alexander grading and the cobordism map}\label{subsec:cob}
The type-$D$ module of the $1$-framed solid torus can be obtained from the $(0,-1)$-framed Hopf link; that is,
\[\calD_{1}^\calK\cong \calD_{0}^{\calK}\Boxtimes\AbZ_{(-1,0)}.\]
The chain homotopy equivalence is given in \cite[Lemma~18.2]{zemke2021bordered}, which we also review here. In the $\bfI_0$ idempotent, $\calD_{1}^\calK\cdot\bfI_0$ is generated by $\xs^0$, whereas $\calD_{0}^{\calK}\Boxtimes\AbZ_{(-1,0)}$ is represented by the following zigzag.
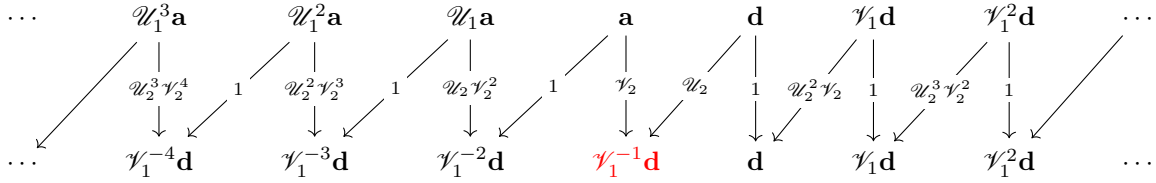
\[\begin{tikzcd}
	\cdots&{\scrU_{1}^{3}\mathbf{a}}& {\scrU_{1}^{2}\mathbf{a}} & {\scrU_{1}\mathbf{a}} & {\mathbf{a}} & {\mathbf{d}} & {\scrV_{1}\mathbf{d}}&{\scrV_{1}^2\mathbf{d}} &\cdots\\
	\\
	\cdots&{\scrV_{1}^{-4}\mathbf{d}}& {\scrV_{1}^{-3}\mathbf{d}} & {\scrV_{1}^{-2}\mathbf{d}} & \color{red}{\scrV_{1}^{-1}\mathbf{d}} & {\mathbf{d}} & {\scrV_{1}\mathbf{d}}&{\scrV_{1}^{2}\mathbf{d}}&\cdots
	\arrow[from=1-2, to=3-1]
	\arrow["{\scrU_{2}^3\scrV_2^4}"{description}, from=1-2, to=3-2]
	\arrow["{1}"{description}, from=1-3, to=3-2]
	\arrow["{\scrU_{2}^2\scrV_2^3}"{description}, from=1-3, to=3-3]
	\arrow["{1}"{description}, from=1-4, to=3-3]
	\arrow["{\scrU_2\scrV_2^2}"{description}, from=1-4, to=3-4]
	\arrow["{1}"{description}, from=1-5, to=3-4]
	\arrow["{\scrV_2}"{description}, from=1-5, to=3-5]
	\arrow["{\scrU_{2}}"{description}, from=1-6, to=3-5]
	\arrow["{1}"{description}, from=1-6, to=3-6]
	\arrow["{\scrU_{2}^2\scrV_2}"{description}, from=1-7, to=3-6]
	\arrow["{1}"{description}, from=1-7, to=3-7]
	\arrow["{\scrU_{2}^3\scrV_2^2}"{description}, from=1-8, to=3-7]
	\arrow["{1}"{description}, from=1-8, to=3-8]
	\arrow[from=1-9, to=3-8]
\end{tikzcd}\]

As in Lemma \ref{lem:red}, this zigzag is chain homotopic to the module which is generated by the red generator and with trivial structure map. The homotopy equivalence between $\calD_{1}^\calK\cdot\bfI_0$ and $\calD_{0}^{\calK}\Boxtimes\AbZ_{(-1,0)}$ is again given in Lemma \ref{lem:red}. The $\bfI_1$ idempotent is the same as $\bfI_0$ component, except that it is localized with respect to $\scrU_2$.

The cobordism map is then given by the inclusion of the subcomplex, composed with the projection map from $\calD_{0}^{\calK}\Boxtimes\AbZ_{(-1,0)}$ to $\calD_{1}^\calK\cdot\bfI_0$. The descriptions on the two idempotent summands are identical up to localization, so it suffices to carry out the calculation in the $\bfI_0$ summand. 

We first study the Alexander grading of the generators of $_\calK\calC_{(n,m)}$. We use $\ve{s}'$ to denote the grading on the vertices $2$ through $n+3$, and $\bar{\ve{s}}'$ to denote its projection onto the first two components.
\begin{enumerate}
	\item $x_{1,i}^0=x_1^0\otimes \tilde{\scrU}_2^{\abs{i}}$. Recall that $x_1^0=\tilde{\scrV}_{1}^{-1} \tilde{\bfd} | (\bfa_{1},\dots,\bfa_{n})$. Hence, $\ve{s}'(x_1^0)=(0,n,0,\dots,0)$, and $\bar{\ve{s}}'(x_1^0)=(0,n)$. Since $\tilde{\scrU}_2$ decrease the Alexander grading on the first component by $1$, we have $\bar{\ve{s}}'(x_{1,i}^0)=(i,n)$.
	\item $x_{k,i}^{0}$, for $k\in\{2,\dots,n-1\}$ and $i\in \mathbb{Z}$, is different from $x_{1,i}^0$ by an extra $\tilde{\scrV}_{1}^{-k+1}$, which changes the Alexander grading on the second component by $1-k$. Hence, $\bar{\ve{s}}'(x_{k,i}^{0})=(i,n+1-k)$.
	\item For $x_{n,i}^{0}$, $i\in\{0,1,2,\dots\}$, by the similar discussion, we have $\bar{\ve{s}}'(x_{n,i}^{0})=(i,1)$.
\end{enumerate}

We next consider the factor $\calD_0^{\calK}\Boxtimes\AbZ_{(-1,0)}$. Since the projection map is nontrivial only on the bottom of each zigzag, it is enough to compute the Alexander grading of these generators. We denote this grading by $\ve{s}''$. For a bottom generator, $\ve{s}''(\scrV^j\bfd)=(1+j,0)$. Using the connected-sum
grading convention from the previous subsection, the generators corresponding to $\frs_0$ are
\begin{enumerate}
	\item $\scrV^{n-2}\bfd\otimes x_{1,i}^0$.
	\item $\scrV^{n-k-1}\bfd\otimes x_{k,i}^{0}$.
	\item $\scrV^{-1}\bfd\otimes x_{n,i}^{0}$.
\end{enumerate}

Composing with the projection map, the cobordism map is given by
\begin{enumerate}
	\item $\phi_0(x_{1,i}^0)=x_{1,i}^0\otimes\tilde{\scrU}_2^{n-1}\tilde{U}_2^{\frac{(n-2)(n-1)}{2}}$.
	\item $\phi_0(x_{k,i}^0)=x_{k,i}^0\otimes\tilde{\scrU}_2^{n-k}\tilde{U}_2^{\frac{(n-k-1)(n-k)}{2}}$.
	\item $\phi_0(x_{n,i}^{0})=x_{n,i}^{0}.$
\end{enumerate}

\section{Computation of cabling knot Floer homology and colored knot Floer homology}\label{sec:cabling_computations}
In this section, we use the cabling modules to compute several examples of the link Floer homology $\HFL(K_{n,nm})$ and the corresponding colored Floer homology. Recall from Remark \ref{remark:orientation-}, there is an orientation reversal for the cable link. Hence, after the pairing, we need to apply the elliptic bimodule $_\calK[E]^\calK$ to recover the desired
orientation convention. By definition, this amounts to interchanging $\scrU_i$ and $\scrV_i$ for each component.
\subsection{Unknot}
The type-$D$ module$\calX_0^\calK(U)$ has two generators $x^0$ and $x^1$ which are in idempotent $\bfI_0$ and $\bfI_1$ respectively. The structure map takes
\[\delta^1(x^0)=x^1\otimes\sigma+x^1\otimes\tau.\]
Hence, as a $\bbF[U]$ module, $\calX_0^\calK(U)\Boxtimes {_\calK\calC_{(n,m)}}$ has the same set of generators as ${_\calK\calC_{(n,m)}}$. For simplicity, we use the notation from ${_\calK\calC_{(n,m)}}$ to denote the corresponding generator in $\calX_0^\calK(U)\Boxtimes {_\calK\calC_{(n,m)}}$. See Figure \ref{fig:C(3,1)} for the set of generators of $\calX_0^\calK(U)\Boxtimes {_\calK\calC_{(3,1)}}$. This chain complex admits a splitting as direct sums of zigzags. There are three types of zigzags, which are depicted by different colors in Figure \ref{fig:C(3,1)}.
\begin{figure}[h]
	\centering
	\begin{adjustbox}{ max width=0.9\textheight}
		\begin{tikzcd}[column sep=0cm, row sep=2cm]
			\cdots & x_{1,-3}^0 & x_{1,-2}^0 & x_{1,-1}^0 & x_{1,0}^0 &&& \cdots &x_{2,-2}^0 & x_{2,-1}^0 & x_{2,0}^0 & x_{2,1}^0 & x_{2,2}^0 &\cdots &&& \quad  & x_{3,0}^0  & x_{3,1}^0  & x_{3,2}^0 & x_{3,3}^0 &\cdots \\
			&&&&&&& \cdots & x_{2,-2}^1 & x_{2,-1}^1 & x_{2,0}^1 & x_{2,1}^1 & x_{2,2}^1 &\cdots &&& x_{3,-1}^1 &  x_{3,0}^1  & x_{3,1}^1  & x_{3,2}^1 & x_{3,3}^1&\cdots
			\arrow[from=1-2, to=2-9,red]
			\arrow[from=1-9, to=2-9,red]
			\arrow[from=1-9, to=2-17,red]
			\arrow[from=1-3, to=2-10,blue]
			\arrow[from=1-10, to=2-10,blue]
			\arrow[from=1-10, to=2-18,blue]
			\arrow[from=1-18, to=2-18,blue]
			\arrow[from=1-13, to=2-13]
			\arrow[from=1-13, to=2-21]
			\arrow[from=1-21, to=2-21]
		\end{tikzcd}
	\end{adjustbox}
	\caption{$\calX_0^\calK(U)\Boxtimes {_\calK\calC_{(3,1)}}$ and different type of zigzags.}
	\label{fig:C(3,1)}
\end{figure}

\begin{enumerate}
	\item The red zigzags end at a generator $x_{n,i}^1$, for $i<0$ at bottom. Hence, this kind of zigzags contain an even number of generators. Since the differentials do not decrease the $i$ index, all generators in these zigzags have $i<0$. By the definition of $x_{n,i}^1$, these generators are multiplied by a power of $\tilde{\scrU}_2$. Hence, the differentials arising from $\tau$ are identity maps. Thus, this type of zigzags consists of canceling pairs and have trivial homology.
	\item The black zigzags start at a generator $x_{2,i}^1$, for $i>m$ at bottom. Similar to the red zigzags, there are even number of generators in this kind of zigzags. In this case, all the generators have $i>0$, which implies that they are multiplied by a power of $\tilde{\scrV}_2$ and the differentials arising from $\sigma$ are identity maps. Thus, the black zigzags have trivial homology, too.
	\item The blue zigzags have an odd number of generators and nontrivial homology. The homology of each blue zigzag is rank one as a $\bbF[U]$,  and it is generated by a linear combination of generators on top of the zigzag. Note that, fixing the framing $m$, the blue zigzag containing $x_{n,i}^0$ for $0\leq i \leq m$ is chain homotopy equivalent, by the Gaussian cancellation lemma, to the chain complex generated by $x_{n,i}^0$ with trivial differential. There are also  blue zigzags containing $x_{n,i}^0$ for $m<i\leq (n-1)m$. For example, in Figure \ref{fig:C(3,1)}, there is a zigzag containing $x_{3,2}^0$. For these zigzags, applying Gaussian cancellation again, the homology is generated by a suitable $U$-multiple of $x_{n,i}^0$.
\end{enumerate}

This determines $\HFL(U_{n,nm})$. To compute the colored knot Floer
homology, it remains to understand the cobordism maps on the blue
zigzags. Recall that $\phi_0$ acts as the identity on the generators
$x_{n,i}^0$. Moreover, for every fixed $i$, once $m$ is sufficiently
large the class represented by $x_{n,i}^0$ is the surviving generator
of its blue zigzag. Thus these classes stabilize in the direct system.

The relations among the stable generators are those described in
Subsection~\ref{subsec:An0}. After applying the elliptic bimodule
${}_{\calK}[E]^{\calK}$ on each link component, which interchanges
$\scrU_i$ and $\scrV_i$, these relations agree with
Equation~\ref{eq:Arelation}. We therefore recover the computation of
the colored knot Floer homology of the unknot from
\cite{alishahi2025colored}.
\subsection{Algebraic lemmas for the computations}
\label{subsec:alglem}
In this subsection, we establish two algebraic lemmas used in the computations below: a reduction
of the localized knot complex and the flatness of the cabling summands.
Throughout this subsection, set
\[
R=\bbF[\tilde{\scrU}_{2},\tilde{\scrV}_{2}],
\qquad S=R[\tilde{\scrV}_{2}^{-1}],
\qquad C=\CFKm(K),
\]
and write $C^{0}=C$ and $C^{1}=C\otimes_R S$ for the knot complexes
in the two idempotents.

\begin{lem}\label{lem:localizedknot}
	For a knot $K\subset S^{3}$, the complex $C^{1}$ admits a
	deformation retraction onto a rank one complex $S\langle\theta\rangle$ with trivial differential.
\end{lem}

\begin{proof}
	Choose a vertically simplified basis for $C$. Let $\partial_{\mathrm v}$ denote the
	part of the differential containing no factor of $\tilde{\scrU}_{2}$.
	In this basis, the generators consist of vertical pairs
	\[
	\partial_{\mathrm v}a_j=\tilde{\scrV}_{2}^{\ell_j}b_j,
	\qquad
	\partial_{\mathrm v}b_j=0,
	\qquad
	\ell_j\geq 0,
	\]
	and a single unpaired generator $\theta$. After tensoring with $S$, each coefficient $\tilde{\scrV}_2^{\ell_j}$ becomes a unit. Hence every vertical pair is contractible, with contracting homotopy
    \[
      b_j\longmapsto\tilde{\scrV}_2^{-\ell_j}a_j.
    \]
    Applying Gaussian cancellation successively to these pairs gives a deformation retraction of $C^1$ onto the complex generated by a single element $\theta$.
\end{proof}

Let $M_k^\varepsilon$ be the $R_n$-submodule of ${}_{\calK}\calC_{(n,m)}$ generated by all $x_{k,i}^\varepsilon$.

\begin{lem}\label{lem:cablingflat}
	The module $M_k^0$ is flat over $R$. The module $M_k^1$ is flat over
	$S$, and hence also flat over $R$.
\end{lem}

\begin{proof}
	By the zigzag description in Subsection~\ref{subsec:Dn0}, the surviving
	red generators form a basis for the homology over $R$ in the $\bfI_0$
	summand and over $S$ in the $\bfI_1$ summand. This decomposition respects
	the first index $k$. Thus $M_k^0$ and $M_k^1$ are direct sums of copies of
	$R$ and $S$, respectively, and hence are free, and therefore flat, over
	these rings. Since $S$ is a localization of $R$, $M_k^1$ is also flat
	over $R$.
\end{proof}

Tensoring the deformation retraction and applying flatness give
\begin{equation*}\label{eq:algreduction}
	C^1\otimes_S M_k^1\simeq M_k^1\langle\theta\rangle,
	\qquad
	H_*(C^0\otimes_R M_k^0)
	\cong H_*(C^0)\otimes_R M_k^0.
\end{equation*}

\subsection{$L$-space knots}
Recall that a rational homology sphere $Y$ is an L-space if $\dim_{\bbF}\widehat{\mathrm{HF}}(Y)=|H_1(Y;\bbZ)|$. A link $L\subseteq S^3$ is called an L-space link if all sufficiently large positive surgeries on $L$ are L-spaces. L-space link is characterized by link Floer homology. If $L\subseteq S^3$, then L is an L-space link iff its link Floer homology has no $U$-torsion, see \cite[Subsection~1.3]{borodzik2024lattice}.

When $K$ is an L-space knot in $S^3$, $\mathcal{CFK}^-(K)$ admits a staircase model \cite{ozsvath2005knot} determined by the symmetrized Alexander polynomial. Write
\[
\Delta_K(z)=\sum_{j=0}^{2r}(-1)^j z^{\alpha_j},\qquad g=\alpha_0>\alpha_1>\cdots>\alpha_{2r}=-g,
\]
where $\alpha_{2r-j}=-\alpha_j$ and $g:=g(K)$. The staircase has bottom generators $b_0,\dots,b_r$ and top generators $t_1,\dots,t_r$, with Alexander gradings
\[
A(b_s)=\alpha_{2s},\qquad A(t_s)=\alpha_{2s-1}.
\]
Its differential is
\begin{equation}\label{eq:Lspacestaircase}
	\partial(t_s)=\tilde{\scrU}_2^{a_s}b_{s-1}+\tilde{\scrV}_2^{c_s}b_s,\qquad \partial(b_s)=0,
\end{equation}
where
\[
a_s=\alpha_{2s-2}-\alpha_{2s-1}, \qquad c_s=\alpha_{2s-1}-\alpha_{2s}.
\]
The symmetry of the Alexander polynomial gives
\[
A(b_{r-s})=-A(b_s),\qquad a_s=c_{r+1-s}.
\]
The bottom generators of the staircase represent the homology classes of $\mathcal{CFK}^-(K)$, with relations given by \eqref{eq:Lspacestaircase}.

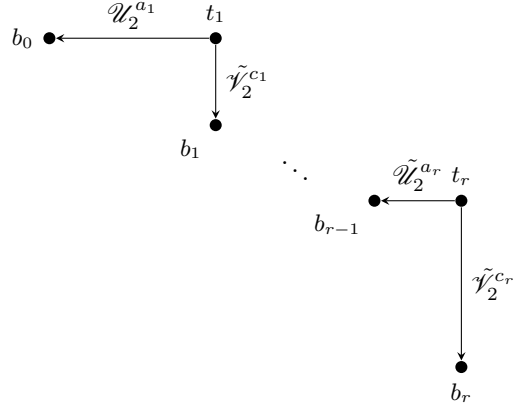
\begin{figure}[ht]
	\centering
	\begin{tikzpicture}[
		x=1cm,y=1cm,>=stealth,
		vertex/.style={circle,fill,inner sep=1.6pt},
		every label/.style={font=\small}]
		\node[vertex,label=left:{$b_0$}] (b0) at (0,0) {};
		\node[vertex,label=above:{$t_1$}] (t1) at (2.2,0) {};
		\node[vertex,label=below left:{$b_1$}] (b1) at (2.2,-1.15) {};
		\node at (3.25,-1.65) {$\ddots$};
		\node[vertex,label=below left:{$b_{r-1}$}] (bp) at (4.3,-2.15) {};
		\node[vertex,label=above:{$t_r$}] (tr) at (5.45,-2.15) {};
		\node[vertex,label=below:{$b_r$}] (br) at (5.45,-4.35) {};
		\draw[->] (t1) -- node[above] {$\tilde{\scrU}_2^{a_1}$} (b0);
		\draw[->] (t1) -- node[right] {$\tilde{\scrV}_2^{c_1}$} (b1);
		\draw[->] (tr) -- node[above] {$\tilde{\scrU}_2^{a_r}$} (bp);
		\draw[->] (tr) -- node[right] {$\tilde{\scrV}_2^{c_r}$} (br);
	\end{tikzpicture}
	\caption{The staircase model. Each arrow is a differential term; the step lengths are schematic.}
	\label{fig:Lspacestaircase}
\end{figure}

With our knot surgery conventions, the flip map
$\tau:\mathcal{CFK}^-(K)\longrightarrow \mathcal{CFK}^-(K)[\tilde{\scrV}_2^{-1}]$ is given by
\begin{equation}\label{eq:Lspaceflip}
	\tau(b_s)=\tilde{\scrV}_2^{2A(b_s)}b_{r-s}.
\end{equation}
The factor $\tilde{\scrV}_2^{2k}$ makes the flip preserve Alexander grading.
The map extends to coefficients by
\[
\tau(\tilde{\scrU}_2x)=\tilde{\scrV}_2^{-1}\tau(x),\qquad\tau(\tilde{\scrV}_2x)=\tilde{\scrU}_2\tilde{\scrV}_2^2\tau(x).
\]
For framing $m$, the flip term acquires the additional factor $\tilde{\scrV}_2^m$, so its value on $b_s$ is
\begin{equation*}\label{eq:Lspaceframedflip}
	\tilde{\scrV}_2^m\tau(b_s)=\tilde{\scrV}_2^{m+2A(b_s)}b_{r-s}.
\end{equation*}

We now describe the deformation retraction of the localized staircase explicitly. We follow the notation from Lemma \ref{lem:localizedknot}. Set
\[
p_s=\sum_{j=1}^{s}a_j,\qquad q_s=\sum_{j=1}^{s}c_j,
\]
with $p_0=q_0=0$. The staircase symmetry implies
\[
p_r=q_r=g,\qquad p_{r-s}=g-q_s,\qquad q_{r-s}=g-p_s.
\]
Moreover,
\[
A(b_s)=g-p_s-q_s.
\]
Recall from Lemma \ref{lem:localizedknot}, $C^1=C\otimes_R S$ deformation retracts onto $S\langle b_0\rangle$ with trivial differential:
\[
S\langle b_0\rangle\underset{\pi_1}{\overset{i_1}{\rightleftarrows}} C^1,\qquad i_1(b_0)=b_0.
\]
where the projection is given by
\begin{equation*}\label{eq:Lspaceretraction}
	\pi_1(b_s)=\tilde{\scrU}_2^{p_s}\tilde{\scrV}_2^{-q_s}b_0,\qquad\pi_1(t_s)=0.
\end{equation*}

Composing the flip map with this projection gives
\[
\overline{\tau}=\pi_1\circ\tau:C\longrightarrow S\langle b_0\rangle.
\]
Using \eqref{eq:Lspaceflip}, we obtain
\begin{align*}
	\overline{\tau}(b_s)&=\tilde{\scrV}_2^{2A(b_s)}\tilde{\scrU}_2^{p_{r-s}}\tilde{\scrV}_2^{-q_{r-s}}b_0 \notag\\
	&=\tilde{\scrU}_2^{g-q_s}\tilde{\scrV}_2^{g-p_s-2q_s}b_0.
	\label{eq:Lspacereducedflip}
\end{align*}
In particular,
\[
\overline{\tau}(b_0)=\tilde U_2^g b_0,\qquad\overline{\tau}(b_r)=\tilde{\scrV}_2^{-2g}b_0,
\]
where $\tilde U_2=\tilde{\scrU}_2\tilde{\scrV}_2$.

Let $\iota:C\to C^1$ denote the localization map. After reducing the target to $S\langle b_0\rangle$, the inclusion and the framed flip therefore take the form
\begin{equation}\label{eq:Lspacereducedmaps}
	\begin{aligned}
		(\pi_1\circ\iota)(b_s)&=\tilde U_2^{p_s}\tilde{\scrV}_2^{A(b_s)-g}b_0,\\
		\pi_1\bigl(\tilde{\scrV}_2^m\tau(b_s)\bigr)&=\tilde U_2^{g-q_s}\tilde{\scrV}_2^{m+A(b_s)-g}b_0.
	\end{aligned}
\end{equation}

We now compute the cabling homology. Assume that $n\geq2$ and that $K$ is a nontrivial $L$-space knot, so $g>0$. The unknot was treated above, while the case $n=1$ is simply the original knot. By Lemma~\ref{lem:localizedknot}, we may replace each localized knot complex in the $\bfI_1$ summand by its rank one model generated by $b_0$. The resulting complex has underlying module
\[
\calD_m=\bigoplus_{k=1}^{n}(C\otimes_R M_k^0)\;\oplus\;\bigoplus_{k=2}^{n}M_k^1\langle b_0\rangle.
\]
Its differential consists of the staircase differentials in the $\bfI_0$ summand and the inclusion and framed flip maps.
Set
\[
y_{k,j}=b_0\otimes x_{k,j}^1.
\]
The generators $y_{k,j}$ are cycles. Combining \eqref{eq:Lspacereducedmaps} with the structure maps of the cabling module gives
\begin{equation}\label{eq:Lspacecablingdiff}
	\begin{aligned}
		d_m(b_s\otimes x_{k,i}^0)
		={}&U^{p_s+\max\{-i,0\}}\,
		y_{k,i-p_s-q_s}\\
		&+U^{g-q_s+\max\{i,0\}}\,
		y_{k+1,i+m-p_s-q_s}.
	\end{aligned}
\end{equation}
Here the first term corresponds to the inclusion map and is omitted when $k=1$. The second corresponds to the flip map and is omitted when $k=n$. The top staircase generators retain their internal
differentials and have no component in the reduced $\bfI_1$ summand.

Suppose first that $m\leq2g-2$. We claim that
\[
[y_{n,-1}]=[b_0\otimes\tilde{\scrV}_2^{-1}x_n^1]
\]
is a nonzero $U$-torsion class. To prove nonvanishing, project the differential to the rightmost $\bfI_1$ summand and reduce modulo
\[
\mathfrak a=(\scrU_1,\ldots,\scrU_n,\scrV_1,\ldots,\scrV_n).
\]
By the coefficient relations in Subsection~\ref{subsec:An0}, the images of the generators $y_{n,j}$ form an $\bbF$-basis of this quotient. In particular, $U=0$ there. Therefore, an inclusion term in
\eqref{eq:Lspacecablingdiff} survives only when $s=0$ and $i\geq0$; its target has index $j\geq0$. A flip term survives only when $s=r$ and $i\leq0$; its target has index $j=i+m-2g\leq m-2g$. Thus $y_{n,-1}$ is not in the image when $m\leq2g-2$. On the other hand,
\[
d_m(b_r\otimes x_{n,2g-1}^0)=U^g y_{n,-1}.
\]
Consequently,
\[
[y_{n,-1}]\neq0,\qquad U^g[y_{n,-1}]=0,
\]
and $K_{n,nm}$ is not an $L$-space link.

Now suppose that $m\geq2g-1$. We show that every generator in the $\bfI_1$ summand is a boundary. The two endpoint staircase generators give
\begin{align}
	d_m(b_0\otimes x_{k,i}^0)
	&=y_{k,i}+U^{g+i}y_{k+1,i+m},
	&& i\geq0,\quad 2\leq k\leq n,
	\label{eq:Lspacepositivecancel}\\
	d_m(b_r\otimes x_{k,i}^0)
	&=U^{g-i}y_{k,i-2g}
	+y_{k+1,i+m-2g},
	&& i\leq0,\quad 1\leq k<n.
	\label{eq:Lspacenegativecancel}
\end{align}
As before, nonexistent endpoint terms are omitted.

Starting with $k=n$ and proceeding to the left, \eqref{eq:Lspacepositivecancel} shows that every $y_{k,j}$ with $j\geq0$ is a boundary. Similarly, starting with $k=1$ and proceeding to the right,
\eqref{eq:Lspacenegativecancel} shows that every $y_{k,j}$ with $j\leq m-2g$ is a boundary.
Since
\[
m-2g\geq-1,
\]
these two ranges cover all integer indices. In particular, the argument includes the endpoint $m=2g-1$.

It remains to compute the surviving homology in the $\bfI_0$ summand. Set
\[
H_K=H_*(C),
\qquad
A_k=H_K\otimes_R M_k^0.
\]
By Lemma~\ref{lem:cablingflat},
\[
H_*(C\otimes_R M_k^0)\cong A_k.
\]
Since the staircase is a free resolution of $H_K$, these modules are generated by the bottom staircase generators, subject to the relations
\[
\tilde{\scrU}_2^{a_s}[b_{s-1}]
=\tilde{\scrV}_2^{c_s}[b_s].
\]
The top generators impose these relations and contribute no additional homology.

The inclusion and flip maps induce
\[
F_m:\bigoplus_{k=1}^{n}A_k\longrightarrow\bigoplus_{k=2}^{n}M_k^1\langle b_0\rangle.
\]
The cancellation above proves that $F_m$ is surjective. Hence
\begin{equation*}\label{eq:Lspacecablingkernel}
	H_*(\calD_m)\cong\ker F_m,
	\qquad m\geq2g-1.
\end{equation*}
This is the same elimination used in the unknot computation: the $\bfI_1$ generators disappear, and the surviving classes are combinations of bottom staircase generators with cabling coefficients.

Multiplication by $U$ is injective on $H_K$. Flatness of $M_k^0$ preserves this injectivity, so every $A_k$ is $U$-torsion-free. Its submodule $\ker F_m$ is therefore also $U$-torsion-free. By the torsion-free characterization of $L$-space links, we conclude that
\[
K_{n,nm}\text{ is an }L\text{-space link}\quad\Longleftrightarrow\quad m\geq2g(K)-1.
\]
Thus the cabling computation gives another proof of \cite[Theorem~2]{gorsky2017cable} for this family of cable links.

We next take the colimit. For $m\geq2g-1$, projection onto the rightmost $\bfI_0$ summand gives a map
\[
\rho_m:H_*(\calD_m)\cong\ker F_m \longrightarrow A_n=H_K\otimes_R M_n^0.
\]
Similar to the computation of unknot from the previous subsection, by the study of different kinds of zigzags, one can show that this map is injective and every fixed generator of $A_n$ lies in the image of $\rho_m$ for all sufficiently large $m$.  Finally, the cobordism map acts trivially on the knot factor and
satisfies
\[
\phi_0(x_{n,i}^0)=x_{n,i}^0.
\]
Consequently,
\[
\rho_{m+1}\circ(\phi_0)_*=\rho_m.
\]
The images of the injective maps $\rho_m$ therefore form an increasing family of submodules whose union is all of $A_n$. We obtain
\begin{equation*}\label{eq:Lspacecoloredlimit}
	\varinjlim_m H_*(\calD_m)\cong H_K\otimes_R M_n^0.
\end{equation*}

Recall from the previous subsection, $M_n^0$ is equivalent to Alishahi-Gorsky-Liu's colored unknot algebra after applying the elliptic bimodule. Hence, our computation recovers \cite[Theorems~6.6]{alishahi2025colored}.

\subsection{The figure-eight knot and the left-handed trefoil}
\label{subsec:figureeightlefttrefoil}
Write
\[
\mathscr C_m(K)=\calX_0^\calK(K)\Boxtimes{}_{\calK}\calC_{n,m}
\]
for the paired complex before the elliptic interchange.

We first compute the cabling homology of the figure-eight knot. Its knot Floer complex splits as
\begin{equation}\label{eq:figureeightsplitting}
	\CFKm(4_1)\simeq R\langle z\rangle\oplus Q,
\end{equation}
where $R\langle z\rangle$ is the unknot complex and $Q$ is the square complex generated by $a,b,c,d$, with
\begin{equation*}\label{eq:figureeightsquare}
	\partial a=\tilde{\scrU}_2b+\tilde{\scrV}_2c,
	\qquad
	\partial b=\tilde{\scrV}_2d,
	\qquad
	\partial c=\tilde{\scrU}_2d,
	\qquad
	\partial d=0.
\end{equation*}
The square is the Koszul free resolution of
$R/(\tilde{\scrU}_2,\tilde{\scrV}_2)$, so
\[
H_*(Q)\cong R/(\tilde{\scrU}_2,\tilde{\scrV}_2)\langle[d]\rangle.
\]
In particular, $Q$ becomes contractible after inverting either $\tilde{\scrU}_2$ or $\tilde{\scrV}_2$.

The inclusion preserves the splitting \eqref{eq:figureeightsplitting}, and the contractibility of the localized square allows us to choose a representative of the flip that preserves it as well. The summand $R\langle z\rangle$ therefore gives the unknot cabling complex. For the square summand, we have
\[
Q\otimes_R S\simeq0.
\]
Thus its $\bfI_1$ copies can be contracted. Since there are no maps from $\bfI_1$ back to $\bfI_0$, this leaves the separate complexes
\[
Q\otimes_R M_k^0,\qquad 1\leq k\leq n.
\]
By Lemma~\ref{lem:cablingflat},
\begin{equation*}\label{eq:figureeightTor}
	H_*(Q\otimes_R M_k^0)\cong\frac{M_k^0}{(\tilde{\scrU}_2,\tilde{\scrV}_2)M_k^0}\langle[d]\rangle.
\end{equation*}

Let $I_j^{\scrU}\subset R_n$ be the ideal generated by all square-free monomials of degree $j$ in
$\scrU_1,\ldots,\scrU_n$, and define $I_j^{\scrV}$ similarly. Setting $\tilde{\scrU}_2=\tilde{\scrV}_2=0$ kills every $x_{k,i}^0$ with $i\neq0$. The remaining class is generated by $x_{k,0}^0$, with coefficient module
\begin{equation*}\label{eq:figureeightTk}
	T_k:=	\frac{M_k^0}{(\tilde{\scrU}_2,\tilde{\scrV}_2)M_k^0}\cong\frac{R_n}{I_{n-k+1}^{\scrU}+I_k^{\scrV}
		+\langle\scrU_j\scrV_j\mid1\leq j\leq n\rangle}.
\end{equation*}
Consequently, before the elliptic interchange, the cabling homology decomposes as
\begin{equation*}\label{eq:figureeightcabling}
	H_*(\mathscr C_m(4_1))
	\cong
	H_*(\mathscr C_m(U))
	\oplus
	\bigoplus_{k=1}^{n}T_k\langle[d]\rangle.
\end{equation*}

We next take the colimit. For $k<n$, the transition factor on the $k$th column contains $\tilde{\scrU}_2^{n-k}$, which acts trivially on $T_k$. On the rightmost column, $\phi_0$ is the identity.
Hence
\begin{equation*}\label{eq:figureeightcobordism}
	(\phi_0)_*|_{T_k}=
	\begin{cases}
		0,&k<n,\\
		\mathrm{id},&k=n.
	\end{cases}
\end{equation*}
All torsion summands except $T_n$ therefore vanish in the colimit, while
\[
T_n\cong\frac{R_n}{(\scrU_1,\ldots,\scrU_n,\scrV_1\cdots\scrV_n)}
\]
survives. The unknot summand contributes the colored unknot homology. Applying the elliptic bimodule interchanges $\scrU_j$ and $\scrV_j$, giving the following result.

\begin{prop}\label{prop:figureeightcolored}
	As an ungraded $R_n$-module,
	\[
	\calH_n(4_1)\cong\calH_n(U)\oplus\frac{R_n}{(\scrV_1,\ldots,\scrV_n,\scrU_1\cdots\scrU_n)}.
	\]
\end{prop}

We now consider the left-handed trefoil, restricting attention to its colored colimit. Here $\scrU$ and $\scrV$ denote the companion variables, and we write $C=\CFKm(T_{2,-3})$. The complex $C$ has generators $a,b,c$, with
\begin{equation*}\label{eq:lefttrefoilcomplex}
	\partial a=\scrV b,\qquad
	\partial c=\scrU b,\qquad
	\partial b=0.
\end{equation*}
Its homology is
\begin{equation*}\label{eq:lefttrefoilhomology}
	H_*(C)\cong R\langle e\rangle\oplus R/(\scrU,\scrV)\langle b\rangle,\qquad e=\scrU a+\scrV c.
\end{equation*}
Thus the left-handed trefoil and the figure-eight knot have isomorphic knot Floer homology as ungraded $R$-modules, although their chain complexes are different. Unlike the figure-eight complex $R\langle z\rangle\oplus Q$, the trefoil complex is not a free resolution of its homology module. Its free cycle $e$ involves the same generators that impose the relations on $ b$. Consequently, the homology decomposition alone does not give a splitting of the cabling computation: the chains imposing these relations can also have components in $\bfI_1$.

We first apply Lemma~\ref{lem:localizedknot} to reduce the $\bfI_1$ complex. After inverting $\scrV$, the surviving generator is
\[
\theta=c+\scrU\scrV^{-1}a,
\]
and we may choose the retraction
\[
\pi_1(a)=\pi_1(b)=0,\qquad \pi_1(c)=\theta.
\]
The localized flip can be represented by
\[
\tau(a)=\scrV^2c,\qquad
\tau(b)=b,\qquad
\tau(c)=\scrV^{-2}a.
\]
It follows that the reduced inclusion and unframed flip satisfy
\[
\pi_1\iota(e)=\scrV\theta,\qquad
\pi_1\tau(e)=\scrV\theta,\qquad
\pi_1\iota(b)=\pi_1\tau(b)=0.
\]
In particular, the torsion class lies in the kernel of both maps, while the maps on the free class $e$ agree with the unknot maps up to the invertible factor $\scrV$ in the target.

By Lemma~\ref{lem:cablingflat}, the homology of each $\bfI_0$ column is
\begin{equation*}\label{eq:lefttrefoilI0}
	H_*(C\otimes_R M_k^0)\cong	M_k^0\langle e\rangle\oplus T_k\langle b\rangle,\qquad T_k=\frac{M_k^0}{(\scrU,\scrV)M_k^0}.
\end{equation*}
By the preceding cabling computation, there is no surviving $\bfI_1$ contribution in the colimit. The free part behaves as in the unknot case and is identified with the rightmost summand $M_n^0\langle e\rangle$, corresponding to $x_n^0$. As in the figure-eight computation, the cobordism maps annihilate
$T_k$ for $k<n$ and act as the identity on $T_n$. Hence only the torsion in the rightmost summand survives, and we obtain
\[
\varinjlim_m H_*(\mathscr C_m(T_{2,-3}))\cong M_n^0\langle e\rangle\oplus T_n\langle b\rangle.
\]

After the elliptic interchange, this gives an isomorphism of ungraded $R_n$-modules
\begin{equation*}\label{eq:lefttrefoilcolored}
	\calH_n(T_{2,-3})\cong\calH_n(U)\oplus\frac{R_n}{(\scrV_1,\ldots,\scrV_n,\scrU_1\cdots\scrU_n)}.
\end{equation*}

\section{The colored colimit for general knots}
\label{sec:chaincolimit}

The computations in the preceding section suggest that the colored colimit is governed by the rightmost cabling summand $M_n^0$ and depends only on knot Floer homology. We prove this for the cabling chain model of a general knot, establishing Theorem~\ref{theorem:colored-general-intro}.

\subsection{Colimits and localization}
\label{subsec:colimitlocalization}

We begin with two standard algebraic facts.

\begin{lem}\label{lem:homologycolimit}
	Let
	\[
	D_0\xrightarrow{f_0}D_1\xrightarrow{f_1}D_2
	\longrightarrow\cdots
	\]
	be a directed system of chain complexes of $A$-modules and chain
	maps. Set $D_\infty=\varinjlim_m D_m$, using the ordinary algebraic
	colimit. Then the natural map
	\begin{equation*}\label{eq:homologycolimit}
		\varinjlim_m H_*(D_m)
		\xrightarrow{\;\cong\;}
		H_*(D_\infty)
	\end{equation*}
	is an isomorphism.
\end{lem}

\begin{proof}
	Directed colimits of modules are exact; see
	\cite[Chapter~2]{atiyah2018introduction}.
	They therefore commute with taking kernels, images, and their
	quotient, which gives the claimed isomorphism.
\end{proof}

For a system whose transition maps are multiplication by a fixed
element, the colimit has an explicit description as a localization.
We use the notation $M[q^{-1}]$ for localization at
$\{1,q,q^2,\ldots\}$; see
\cite[Chapter~3]{atiyah2018introduction}.

\begin{lem}\label{lem:sequentiallocalization}
	Let $A$ be a commutative ring, $q\in A$, and $M$ an $A$-module.
	There is a natural $A$-linear isomorphism
	\begin{equation*}\label{eq:colimitlocalization}
		\varinjlim\bigl(
		M\xrightarrow{q}M\xrightarrow{q}\cdots
		\bigr)
		\cong M[q^{-1}],
		\qquad
		[x,m]\longmapsto\frac{x}{q^m},
	\end{equation*}
	where the first copy of $M$ is stage $0$.
\end{lem}

\begin{proof}
	The maps $M\to M[q^{-1}]$ given at stage $m$ by
	$x\mapsto x/q^m$ are compatible with the transition maps, since
	\[
	\frac{qx}{q^{m+1}}=\frac{x}{q^m}.
	\]
	They therefore induce the stated map from the colimit.
	Every localized element has the form $x/q^m$, so this map is
	surjective. If $x/q^m=0$, then $q^r x=0$ for some $r\geq0$.
	Thus the representative $x$ becomes zero at stage $m+r$,
	proving injectivity.
\end{proof}

Under this identification, the inverse of multiplication by $q$ is
\[
[x,m]\longmapsto[x,m+1].
\]
Thus passing to the colimit makes $q$ invertible and kills every
element annihilated by a power of $q$.

In our application, the transition factors are monomials in $\scrU$ and $\scrV$. If $q=\scrU^a\scrV^b$ with $a,b>0$,
localization at $q$ inverts both variables, since
\[
\scrU^{-1}
=\frac{\scrU^{a-1}\scrV^b}{q},
\qquad
\scrV^{-1}
=\frac{\scrU^a\scrV^{b-1}}{q}.
\]
Conversely, inverting both variables makes $q$ invertible.
We will apply this observation separately to the cabling columns,
using their respective transition factors.

\subsection{The cabling chain complex in the colimit}
\label{subsec:colimitmodel}
We retain the cabling summands $M_k^\varepsilon$ defined above.
At framing $m$, the $\bfI_0$ columns are $C\otimes_R M_k^0$,
for $1\leq k\leq n$, and the $\bfI_1$ columns are
$C^1\otimes_S M_k^1$, for $2\leq k\leq n$.
Besides their internal differentials, the paired complex has the maps
\begin{equation*}\label{eq:colimitarrows}
	\sigma_k:C\otimes_R M_k^0\longrightarrow C^1\otimes_S M_k^1,
	\qquad
	\scrV^m\tau_k:C\otimes_R M_k^0
	\longrightarrow C^1\otimes_S M_{k+1}^1.
\end{equation*}
Here $\sigma_k$ pairs the inclusion $\iota:C\to C^1$ with the cabling-module map $x_k^0\mapsto x_k^1$, while $\tau_k$ pairs the
localized flip with the cabling-module map $x_k^0\mapsto x_{k+1}^1$. Their ranges are $2\leq k\leq n$ and $1\leq k<n$, respectively.

Recall that the cobordism map $\phi_m:\mathscr C_m(C)\to\mathscr C_{m+1}(C)$ acts on either idempotent in column $k$ by multiplication by
\begin{equation}\label{eq:colimitq}
	q_k=\scrU^{n-k}U^{\binom{n-k}{2}},
	\qquad q_n=1.
\end{equation}
We check that this diagonal map commutes with the differential.
The inclusion term $\sigma_k$ is the same at both framings and
has column $k$ as both its source and target. For the flip term,
the required commutative square is
\[
\begin{tikzcd}[column sep=large, row sep=large]
	C\otimes_R M_k^0
	\arrow[r,"\scrV^m\tau_k"]
	\arrow[d,"q_k"']
	&
	C^1\otimes_S M_{k+1}^1
	\arrow[d,"q_{k+1}"]
	\\
	C\otimes_R M_k^0
	\arrow[r,"\scrV^{m+1}\tau_k"']
	&
	C^1\otimes_S M_{k+1}^1 .
\end{tikzcd}
\]

\begin{lem}\label{lem:colimitcompatibility}
	For $1\leq k<n$, we have
	\begin{equation*}\label{eq:qflipcompatibility}
		q_{k+1}=\scrV T(q_k).
	\end{equation*}
	Hence the square above commutes, and the diagonal cobordism
	map $\phi_m$ is a chain map.
\end{lem}

\begin{proof}
	Set $d=n-k$. Since $T(\scrU)=\scrV^{-1}$ and $T(U)=U$,
	\[
	\scrV T(q_k)
	=\scrV^{1-d}U^{\binom d2}
	=\scrU^{d-1}U^{\binom{d-1}{2}}
	=q_{k+1}.
	\]
	The flip is $T$-semilinear. Thus, for $x$ in the $\bfI_0$
	column $k$,
	\[
	(\scrV^{m+1}\tau_k)(q_kx)
	=\scrV^{m+1}T(q_k)\tau_k(x)
	=q_{k+1}(\scrV^m\tau_k)(x).
	\]
	This is exactly the commutativity of the displayed square.
	Also, $\sigma_k(q_kx)=q_k\sigma_k(x)$, since the same
	factor $q_k$ acts on both columns joined by $\sigma_k$.
	The internal differentials commute with multiplication by
	$q_k$, so $\phi_m$ commutes with the full differential.
\end{proof}

We now take the colimit of each column. By
Lemma~\ref{lem:sequentiallocalization}, the resulting chain
complexes are
\[
A_k:=
\varinjlim_m\bigl(
C\otimes_R M_k^0\xrightarrow{q_k}
C\otimes_R M_k^0\xrightarrow{q_k}\cdots
\bigr)
\cong (C\otimes_R M_k^0)[q_k^{-1}]
\]
for $1\leq k\leq n$, and
\[
B_k:=
\varinjlim_m\bigl(
C^1\otimes_S M_k^1\xrightarrow{q_k}
C^1\otimes_S M_k^1\xrightarrow{q_k}\cdots
\bigr)
\cong (C^1\otimes_S M_k^1)[q_k^{-1}]
\]
for $2\leq k\leq n$. In particular,
\[
A_n=C\otimes_R M_n^0.
\]

The commutative squares above let the inclusion and framed flip
descend to maps between colimit columns. If $[x,m]$ denotes an
element represented at framing $m$, these maps are
\[
\overline\sigma_k([x,m])=[\sigma_k(x),m],
\qquad
\overline\tau_k([x,m])
=[\scrV^m\tau_k(x),m].
\]
Under the localization identifications, they become
\[
\overline\sigma_k\left(\frac{x}{q_k^m}\right)
=\frac{\sigma_k(x)}{q_k^m},
\qquad
\overline\tau_k\left(\frac{x}{q_k^m}\right)
=\frac{\scrV^m\tau_k(x)}{q_{k+1}^m}
=\frac{\tau_k(x)}{T(q_k)^m}.
\]
The last expression lies in $B_{k+1}$ because
$T(q_k)=\scrV^{-1}q_{k+1}$ and $\scrV$ is invertible there.
Thus the colimit complex has off-diagonal maps
$\overline\sigma_k:A_k\to B_k$ and
$\overline\tau_k:A_k\to B_{k+1}$.
\begin{lem}\label{lem:localizedpatternflip}
	For $1\leq k<n$, the flip map extends to an $R_n$-linear,
	$T$-semilinear isomorphism
	\[
	M_k^0[\scrU^{-1}]\xrightarrow{\cong}M_{k+1}^1,
	\qquad x_k^0\longmapsto x_{k+1}^1.
	\]
	It induces an isomorphism
	\begin{equation}\label{eq:patternflipinfinity}
		M_k^0[q_k^{-1}]\xrightarrow{\cong}
		M_{k+1}^1[q_{k+1}^{-1}].
	\end{equation}
	Consequently, the induced map
	$\overline{\tau}_k\colon A_k\to B_{k+1}$ is a chain homotopy
	equivalence.
\end{lem}

\begin{proof}
	It is straightforward to check that the flip map respects the relations in Subsection~\ref{subsec:An0}. The same check with $T^{-1}$ gives the inverse. Hence, it induces the isomorphism. Since
	$q_k$ contains a positive power of $\scrU$, localizing at $q_k$
	also inverts $\scrU$. Moreover,
	$q_{k+1}=\scrV T(q_k)$, and $\scrV$ is already invertible in
	$M_{k+1}^1$. Further localization therefore gives
	Equation~\ref{eq:patternflipinfinity}.
	
	Finally, the localized knot flip
	\[
	\Psi\colon C[\scrU^{-1}]\longrightarrow C[\scrV^{-1}]
	\]
	is a $T$-semilinear chain homotopy equivalence, giving the assertion
	for $\overline{\tau}_k$.
\end{proof}

\subsection{Cancellation to the rightmost column}

\begin{proof}[Proof of Theorem~\ref{theorem:colored-general-intro}]
	Let
	\[
	\mathscr C_\infty(C)=\varinjlim_m\mathscr C_m(C).
	\]
	By the preceding subsection, its columns are $A_1,\ldots,A_n$
	and $B_2,\ldots,B_n$, with off-diagonal maps
	$\overline\sigma_k:A_k\to B_k$ and
	$\overline\tau_k:A_k\to B_{k+1}$.
	
	Consider the subcomplex containing every column except $A_n$:
	\[
	E=
	\bigoplus_{k=1}^{n-1}A_k
	\oplus
	\bigoplus_{k=2}^{n}B_k.
	\]
	It is a subcomplex because no differential from these columns
	enters $A_n$. Filter it by
	\[
	E_r=
	\bigoplus_{k=1}^{r}A_k
	\oplus
	\bigoplus_{k=2}^{r+1}B_k,
	\qquad 0\leq r\leq n-1,
	\]
	where $E_0=0$. In $E_r/E_{r-1}$, the inclusion map
	$\overline\sigma_r$ disappears because its target $B_r$ lies
	in $E_{r-1}$. Hence this quotient is the mapping cone of
	\[
	\overline\tau_r:A_r\longrightarrow B_{r+1}.
	\]
	By Lemma~\ref{lem:localizedpatternflip}, this map is a chain
	homotopy equivalence, so its cone is contractible. Successively
	canceling these finitely many pairs shows that $E$ is
	contractible.
	
	The quotient of $\mathscr C_\infty(C)$ by $E$ is $A_n$.
	Consequently, projection to the rightmost column is a chain
	homotopy equivalence
	\[
	\mathscr C_\infty(C)\simeq A_n
	=C\otimes_R M_n^0,
	\]
	where the last equality uses $q_n=1$. In particular, the
	internal differential of the full knot complex $C$ remains
	in the rightmost column.
	
	Lemma~\ref{lem:homologycolimit} and
	Lemma~\ref{lem:cablingflat} now give
	\begin{align*}
		\varinjlim_m H_*(\mathscr C_m(C))
		&\cong H_*(\mathscr C_\infty(C))\\
		&\cong H_*(C\otimes_R M_n^0)\\
		&\cong H_*(C)\otimes_R M_n^0.
	\end{align*}
	The final isomorphism uses the flatness of $M_n^0$ over $R$.
	
	For completeness, the rightmost coefficient module has the
	presentation
	\[
	M_n^0\cong\bbF[\scrV,\scrV_1,\ldots,\scrV_n],
	\qquad
	\scrU=\prod_{i=1}^n\scrV_i,
	\qquad
	\scrU_i=\scrV\prod_{j\ne i}\scrV_j.
	\]
	Applying the elliptic interchange to the cable variables gives
	the colored homology in the orientation convention of the
	theorem. This completes the $R_n$-module computation.
\end{proof}

\bibliographystyle{alpha}  
\bibliography{references}
\end{document}